\documentclass[11pt]{article}

\usepackage{geometry}
\usepackage{mathtools}
\usepackage{tikz}
\usetikzlibrary{arrows}
\usetikzlibrary{patterns}
\usetikzlibrary{hobby}
\usetikzlibrary{decorations.pathreplacing}

\usepackage{graphicx}
\usepackage{amsmath,amssymb,amsthm}
\usepackage{enumerate}
\usepackage{hyperref}
\usepackage[active]{srcltx}
\usepackage[T1]{fontenc}
\usepackage{color}
\usepackage[makeroom]{cancel}

\newtheorem{theo}{Theorem}

\newtheorem{lem}{Lemma}[section]
\newtheorem{defi}[lem]{Definition}

\newtheorem{prop}[lem]{Proposition}
\newtheorem{rmk}[lem]{Remark}

\newcommand{\eps}{\varepsilon}
\newcommand{\qtext}[1]{\quad\mbox{#1}\quad}
\newcommand{\qqtext}[1]{\qquad\mbox{#1}\qquad}
\newcommand{\TV}[1]{{\|#1\|}}
\newcommand{\rd}{\mathrm d}
\newcommand{\R}{\mathbb{R}}

\newcommand{\narrowcv}{\overset{\ast}{\rightharpoonup}}
\newcommand{\grad}{\operatorname{grad}}
\newcommand{\pO}{{\partial\Omega}}
\newcommand{\bO}{{\bar\Omega}}
\newcommand{\QO}{{Q_{\bar\Omega}}}
\newcommand{\QG}{{Q_{\Gamma}}}

\newcommand{\g}{\gamma}
\newcommand{\G}{\Gamma}
\renewcommand{\k}{\kappa}
\newcommand{\M}{\mathcal{M}}
\newcommand{\W}{{\mathcal{W}}}
\newcommand{\Wrr}{\mathcal{W}_\k}
\newcommand{\WFR}{\mathcal{WFR}_\k}
\newcommand{\FR}{\mathcal{FR}_\k}
\newcommand{\A}{{\mathcal{A}}}
\newcommand{\C}{{\mathcal{C}}}
\newcommand{\E}{{\mathcal{E}}}
\newcommand{\F}{{\mathcal{F}}}
\renewcommand{\H}{{\mathcal{H}}}
\newcommand{\argmin}{\operatorname{Argmin}}
\newcommand{\supp}{\operatorname{supp}}
\newcommand{\dist}{\operatorname{dist}}
\newcommand{\CE}{{\mathcal{CE}}}
\newcommand{\Pp}{\mathcal{P}^{\oplus}}

\newcommand{\mrest}{
  \,\raisebox{-.127ex}{\reflectbox{\rotatebox[origin=br]{-90}{$\lnot$}}}\,%
}
\newcommand{\pf}{{}_{\#}}

\renewcommand{\P}{\mathcal{P}}
\renewcommand{\o}{\omega}
\renewcommand{\O}{\Omega}

\def\dive{\operatorname{div}}
\def\diveg{\operatorname{div}_{\Gamma}}

\numberwithin{equation}{section}
\begin{document}
\title{A new transportation distance with bulk/interface interactions and flux penalization}
\date{}
\author{L\'eonard Monsaingeon}

\maketitle
\abstract{
\noindent
We introduce and study a new optimal transport problem on a bounded domain $\bO\subset \R^d$, defined via a dynamical Benamou-Brenier formulation.
The model handles differently the motion in the interior and on the boundary, and penalizes the transfer of mass between the two.
The resulting distance interpolates between classical optimal transport on $\bO$ on the one hand, and on the other hand between two independent optimal transport problems set on $\O$ and $\pO$.
}


\tableofcontents

%%%%%%%%%%%%%%%%%%%%%%%%%%%%%%%%%%%%%%%%%%%%%%%%%%%%%%%%%%%%
\section{Introduction}
In its Monge-Kantorovich formulation \cite{kantorovich1942translocation,monge1781memoire}, classical optimal transport consists in minimizing a transportation cost
$$
\min\limits_\pi\ \iint\limits_{\mathcal X\times\mathcal X}c(x,y)\rd\pi(x,y)
$$
among all transference plans $\pi\in\mathcal P(\mathcal X\times\mathcal X)$ with prescribed left and right marginals $\pi_x=\varrho_0\in\P(\mathcal X)$ and $\pi_y=\varrho_1\in\P(\mathcal X)$, two given probability measures over the base space $\mathcal X$.
Although the theory covers very general settings we shall focus in this paper exclusively on the quadratic cost $c(x,y)=\frac 12 d^2(x,y)$, the squared Euclidean distance on a smooth bounded (closed) domain $\bO\subset\R^d$.
The above minimization then defines the quadratic (squared) Wasserstein distance
$$
\W_{\bO}^2(\varrho_0,\varrho_1)=\min\limits_\pi\ \frac 12\iint\limits_{\bO\times\bO}d^2(x,y)\rd\pi(x,y).
$$
We refer to \cite{villani2003topics} for a rather soft introduction and to \cite{villani_BIG} for a comprehensive account of the theory and full bibliography, see also \cite{OTAM,peyre2019computational} for a more applied point of view.
\\

The classical Benamou-Brenier formula \cite{BB} allows to rewrite the static problem as a dynamical fluid-mechanics problem, namely the minimization of the kinetic energy
\begin{multline}
\label{eq:def_W2_rho_bO}
 \W^2_\bO(\varrho_0,\varrho_1)=
 \min _{\varrho,w}\left\{
\frac 12\iint\limits_{[0,1]\times\bO} \varrho|w|^2
\qqtext{s.t. }\partial_t\varrho +\dive (\varrho w)=0
\right\}
\\
=\min _{\varrho,H}\left\{
\frac 12\iint\limits_{[0,1]\times\bO} \frac{|H|^2}{\varrho}
\qqtext{s.t. }\partial_t\varrho +\dive H=0
\right\}
\end{multline}
with no-flux boundary conditions $H\cdot n=\varrho w\cdot n=0$ on $\pO$ and initial/terminal data $\varrho_0,\varrho_1$.
We refrain from writing any rigorous definitions and statements at this stage and refer to \cite{BB,brenier2003extended,villani2003topics,OTAM}.
Using the physical mass/momentum variables $(\varrho,H)=(\varrho,\varrho w)$ in \eqref{eq:def_W2_rho_bO} allows to recast the original minimization as a \emph{convex} optimization problem in the space of measures, and also paves the way for efficient numerical implementations \cite{peyre2019computational,OTAM} enjoying extremely general convergence properties \cite{lavenant2019unconditional}.
\\

In this work we introduce a new transportation model on $\bO$ that behaves differently in the interior and on the boundary while allowing for interactions between the two.
On can think of $\bO$ as an inner city and of $\G=\pO$ as a surrounding ring road, and the pair of nonnegative measures
$$
\rho=(\o,\g)\in\M^+(\bO)\times\M^+(\G)
$$
describes the densities of cars in the city and on the ring road, respectively.
The overall car density is described by the probability measure
$$
\varrho=\o+\g\in\P(\bO),
$$
the sum of the inner density $\o$ plus the density $\g$ of cars on the ring.
We will try as much as possible to use the same notational distinction between pairs $\rho=(\o,\g)$ and total density $\varrho=\o+\g$ in the whole paper.
Upon entering or leaving the ring road, drivers should pay a toll penalizing the car flux.
We will give a rigorous definition in Section~\ref{sec:definition_existence}, but at this stage our model can be informally written
\begin{multline}
\label{eq:formal_def_Wrr}
\Wrr^2(\rho_0,\rho_1):=
\min 
\Bigg\{
\quad
\iint\limits_{[0,1]\times\bO} \frac{|F|^2}{2\o} 
+ \iint\limits_{[0,1]\times \pO}\frac{|G|^2}{2\g}
+ \k^2\iint\limits_{[0,1]\times \pO}\frac{|f|^2}{2\g}
\\
\mbox{s.t. }
\begin{array}{ll}
 \partial_t\o +\dive F =0 & \mbox{in }\O
 \\
 F\cdot n =f & \mbox{on }\pO
\end{array}
\qqtext{and}
\partial_t\g +\dive G =f
\mbox{ in }\pO
\Bigg\}
\end{multline}
where the endpoints $\rho_0=(\o_0,\g_0),\rho_1=(\o_1,\g_1)$ are prescribed and such that $\varrho_0=\o_0+\g_0$ and $\varrho_1=\o_1+\g_1$ are probability measures.
Here $\k>0$ is a toll parameter, $F$ is the momentum in the interior, and $G$ is the momentum on the road.
The variable $f$ has two possible interpretations: When viewed from the interior, $f$ is just the normal outflux $F\cdot n$ of the city cars, but when viewed from $\pO$ it is rather a source term encoding the intake of cars entering from the city.
Correspondingly, the set $\G=\partial\O$ can be thought of in two different ways: First, as the boundary of the interior set $\O$ where fluxes might arise from/to the interior; And second, as an intrinsic set where $\g$ lives and moves, possibly exchanging mass with the outer world $\O$.
Depending on the context we will write $\partial\O$ or $\G$ to emphasize this idea.
\\

By construction our model preserves the total mass: Denoting $\o_t,\g_t$ the inner and boundary densities at time $t$ and evolving according ot the continuity equations appearing in \eqref{eq:formal_def_Wrr}, it is easy to check at least formally that the overall density $\varrho_t=\o_t+\g_t$ has constant mass.
Indeed since $\pO=\G$ is without boundaries, integration by parts gives
\begin{multline*}
\frac{d}{dt}\left(\int_{\bO} \varrho_t\right)=
\frac{d}{dt}\left(\int_\O \o_t + \int_\G \g_t\right)
\\
=\int_\O\{- \dive F_t\} +\int_\G \{-\dive G_t+f_t\}
=-\int_\pO F_t\cdot n +\int_\G f_t =0 .
  \end{multline*}
However of course, neither the mass of $\o_t$ nor that of $\g_t$ is conserved a priori, the whole point is precisely that mass can be exchanged between $\O$ and $\G$.
In order to tackle these respective mass variations we will leverage the theory of \emph{unbalanced optimal transport}, which has recently attracted considerable attention and significant efforts and resulted in particular in the construction of the so-called \emph{Wasserstein-Fisher-Rao} metric \cite{chizat2018interpolating,kondratyev2016new}, also known as the \emph{Hellinger-Kantorovich} distance \cite{liero2016optimal,liero2018optimal}.
The latter is a distance between arbitrary positive measures $\g_0,\g_1\in\M^+(\G)$ allowing for different masses, and can be roughly defined (here over the base space $\G=\pO$) as
\begin{equation}
\label{eq:def_WFR}
\WFR^2(\g_0,\g_1)=
\min \limits_{\g,G,f}
\Bigg\{
\quad
\iint\limits_{[0,1]\times \G}\frac{|G|^2}{2\g}
+ \k^2\iint\limits_{[0,1]\times \G}\frac{|f|^2}{2\g}
\qqtext{s.t.}
\partial_t\g +\dive G =f
\mbox{ in }\G
\Bigg\}.
\end{equation}
This can be seen as an \emph{infimal convolution} of the Fisher-Rao distance
\begin{equation}
\label{eq:def_formal_FR}
\FR^2(\g_0,\g_1)=
\min  \limits_{\g,f}
\Bigg\{
\quad
\iint\limits_{[0,1]\times \G}\k^2\frac{|f|^2}{2\g}
\qqtext{s.t.}
\partial_t\g =f
\mbox{ in }\G
\Bigg\}
\end{equation}
and the Wasserstein distance
\begin{equation}
\label{eq:def_formal_WG}
\W_\G^2(\g_0,\g_1)=
\min  \limits_{\g,G}
\Bigg\{
\quad
\iint\limits_{[0,1]\times \G}\frac{|G|^2}{2\g}
\qqtext{s.t.}
\partial_t\g +\dive G =0
\mbox{ in }\G
\Bigg\},
\end{equation}
both written here over the manifold $\G=\pO$.
A third quantity also appears in disguise in \eqref{eq:formal_def_Wrr}, namely the Wasserstein distance in $\bO$ between interior densities
\begin{equation}
\label{eq:def_formal_WO}
\W_\bO^2(\o_0,\o_1)=
\min\limits_{\o,F}
\Bigg\{
\quad
\iint\limits_{[0,1]\times \bO}\frac{|F|^2}{2\o}
\qqtext{s.t.}
\begin{array}{ll}
\partial_t\o +\dive F =0 & \mbox{in }\O
\\
F\cdot n =0 & \mbox{on }\pO
\end{array}
\Bigg\},
\end{equation}
and it should be no surprise that these $\WFR,\FR,\W_\G,\W_\bO$ distances will appear frequently in this work.
We refer to \cite{liero2013gradient,liero2016optimal,kondratyev2016new,liero2016optimal,chizat2018unbalanced,chizat2018scaling,laschos2019geometric,gallouet2017jko} and references therein and thereof for a detailed account of the unbalanced theory and various applications \cite{kondratyev2016fitness,kondratyev2017new,kondratyev2019spherical,kondratyev2019convex,kondratyev2020nonlinear,gallouet2019unbalanced,gallouet2018camassa,gallouet2019generalized} (see also \cite{gangbo2019unnormalized} for the so-called \emph{unnormalized} optimal transport).
For the sake of completeness let us also cite \cite{piccoli2014generalized,piccoli2016properties,figalli2010new,profeta2018heat} for related generalized Wasserstein distances allowing for unequal masses, and \cite{caffarelli2010free,figalli2010optimal} for \emph{partial} optimal transport where only a given fraction of the prescribed marginals is moved.
\\

We will make a case in section~\ref{sec:varying toll} that our distance $\Wrr^2(\rho_0,\rho_1)$ interpolates monotonically between $\W^2_\bO(\varrho_0,\varrho_1)=\W^2_\bO(\o_0+\g_0,\o_1+\g_1)$ and $\W^2_\bO(\o_0,\o_1)+\W^2_\G(\g_0,\g_1)$ as the toll parameter $\k$ increases from $0$ to $+\infty$.
In the limits of small and large tolls we will recover both problems as
$$
\Wrr^2(\rho_0,\rho_1)
\xrightarrow[\k\to 0]{} 
\W_\bO^2(\varrho_0,\varrho_1)
\qqtext{and}
\Wrr^2(\rho_0,\rho_1)
\xrightarrow[\k\to +\infty]{} 
\W^2_{\bO}(\o_0,\o_1) + \W^2_{\G}(\g_0,\g_1).
$$
Note carefully that the Wasserstein distances $\W^2_\bO(\o_0,\o_1)$ and $\W^2_\G(\g_0,\g_1)$ allow for an arbitrary value of the common masses $\o_0(\bO)=\o_1(\bO)$ and $\g_0(\G)=\g_1(\G)$, but implicitly take on the value $\inf\emptyset=+\infty$ whenever the endpoints have unequal mass (since in that case they cannot be continuously interpolated by solutions of conservative continuity equations as required in \eqref{eq:def_formal_WG}\eqref{eq:def_formal_WO}).
This will be crucial later on when we take the large toll limit $\k\to+\infty$, roughly speaking because in the limit the exchange of mass between $\O$ and $\G$ is prohibited due to the infinitely expensive cost of any flux.

Let us stress that at this point that, given $\varrho\in\P(\bO)$, there is no uniqueness of the decomposition $\varrho=\o+\g$ into the sum of a measure $\o\in\M^+(\bO)$ plus a measure $\g\in\M^+(\G)$.
A natural and tempting choice would be given by the restrictions $\o=(\varrho\mrest \O),\g=(\varrho\mrest\pO)$.
Our distance $\Wrr^2(\rho_0,\rho_1)$ between pairs $\rho_i=(\o_i,\g_i)$ would accordingly induce a distance $\tilde{\mathcal W}_\k^2(\varrho_0,\varrho_1):=\Wrr^2(\rho_0,\rho_1)
$ for $\rho_i:=(\varrho_i\mrest\O,\varrho_i\mrest\pO)$
between probability measures $\varrho_i\in\P(\bO)$.
The latter framework is however less tractable and lacks desirable properties (e.g completeness and constant speed characterization of geodesics, but we shall not elaborate on this).
Our use of the \emph{pairs} $(\o,\g)$ as a primary variable instead of the more classical scalar densities $\varrho=\o+\g\in\P(\bO)$ allows for more flexibility in the arbitrary choice of such decompositions.
From a practical perspective, this amounts to saying that cars on the ring road can be of two sorts: Cars on the \emph{inner} ring $\o\mrest\pO$ that have not yet paid the toll, and cars on the \emph{outer} ring $\g$ that have already gone through the toll gates.
Both are needed to describe the complete state of the system (in addition to the interior density $\o\mrest\O$, of course).

Our construction cannot be recovered as a simple application of the general abstract theory of optimal transport over Polish spaces:
In order to discriminate between interior and boundary points one could define a partially discrete distance $d_\k(x,y)$ extending the Euclidean distance on $\bO$ and satisfying $d_\k(x,y)=\k>0$ if $x\in\O,y\in\pO$, and then try to construct a transportation distance based upon $d_\k$.
The resulting metric space however fails to be complete, and the standard theory does not apply.
One could also try to use standard optimal transport on an extended space $\bO\cup\Gamma$, where an additional copy $\G$ is glued to the boundary $\pO$ while assigning a fixed $\mathcal O(\kappa^2)$ cost to move particles from one boundary to the other.
In this case underlying space would now fail to be geodesic, and our construction is really simply a different model.
It is however worth pointing out that our basic objects will actually be probability measures over this extended space $\bO\cup\Gamma$, and that the narrow convergence over this probability space is equivalent to the narrow convergences over $\bO$ and $\G$, separately.
(We shall later on refer to this as ``double narrow convergence'' for exposition purposes.)
Another striking difference of our model with classical optimal transport is that, due to the built-in flux penalization, mass cannot enter or exit the boundary at once and must therefore split along the way.
We will show in section~\ref{sec:explicit_geodesics_dirac_masses} that this happens even for two point-masses $\rho_0=(\delta_{x_\O},0),\rho_1 =(0,\delta_{x_\G})$ for two points $x_\O\in\O,x_\G\in\G$.
This phenomenon is in sharp contrast with classical Wasserstein transport, where it is known that mass splitting can only occur at $t=0$ or $t=1$.

Our model is vaguely similar in spirit to \cite{profeta2018heat}, where a transportation distance between \emph{sub}probabilities was constructed by gluing together two copies $\O^+,\O^-$ of the domain $\O$.
One copy $\O^-$ stores or releases mass from/to $\O^+$, the effective density $\varrho=\varrho^+-\varrho^-$ is a subprobability, and the total mass of $\rho=\varrho^++\varrho^-$ is conserved.
Our setup also sees two species interacting together while ensuring conservation of the total mass, but our interaction is singularly located on the boundary and the mathematical analysis is therefore quite different.
Related variational models including bulk/interface interactions have also been considered in \cite{mielke2011thermomechanical,glitzky2013gradient} for reaction-diffusion problems on heterostructures, but the interactions were different and as far as we can tell no rigorous mathematical analysis of the metric structure was carried.
\\
\paragraph{Possible extensions}
\noindent
In \cite[Chapter 4]{zurek2019problemes} a one-dimensional concrete carbonation model with boundary interaction was considered, and an ad-hoc transportation distance $\mathbf d_2$ was constructed over $\bO=[0,\infty)$.
This distance discriminates the boundary $x=0$ by artificially extending the domain to $\{-a\}\cup [0,+\infty)$ for a small $a>0$.
Mass intake is then allowed at $x=-a$, while prohibiting any motion in $(-a,0]$.
The resulting positive cost for jumping from $x\geq 0$ to $x=-a$ corresponds somehow to a space discretization of our flux toll to jump from $\O$ to $\pO$ via a thin boundary layer of thickness $a\ll 1$:
In fact in \cite{zurek2019problemes} the thickness is taken as $a=\sqrt\tau\to 0$ in the small time-step limit for a modified minimizing movement scheme.
This partially motivated the present work, and we hope to use our results herein to handle in the future more realistic models.

In order to carry out the rigorous analysis without overburdening the exposition we chose here to discuss bulk/interface interactions located on the boundary only, but we believe that the approach should cover more general settings.
In particular it seems natural to include internal cracks supported on reasonably smooth lower-dimensional sets (in which case suitable boundary conditions may be required on the tips of the cracks).

Similarly, and as suggested to us by C. Canc\`es, one could possibly consider more general weights $\k^2\frac{|f|^2}{\theta(\o,\g)}$ depending on both densities in the flux penalization, for some one-homogeneous function such as the logarithmic mean $\theta(\o,\g)=\frac{\o-\g}{\log\o-\log\g}$ or upwinding/downinding-weights $\theta(\o,\g)=\lambda^+[\o-\g]^++\lambda^-[\o-\g]^-$ for some coefficients $\lambda^\pm\geq 0$.
For example $\lambda^+=0,\lambda^->0$ could realistically encode the fact that the toll closes its gates in case of traffic congestion on the road: $\g>\o\Rightarrow  \theta(\o,\g)=\lambda^-[\o-\g]^-=0$ whenever the density of cars on the road $\g$ exceeds that in the city $\o$.

Finally, motion is usually more efficient on real-life ring roads than inside cities.
It would therefore be natural to include a new parameter $\delta>0$ and reconsider our model using the weighted action $\frac{|F|^2}{2\o}+\delta^2\frac{|G|}{2\g}+\k^2\frac{f^2}{2\g}$ in order to encode this difference in mobility.
For fixed $\k>0$ the whole analysis presented here immediately carries through.
For large tolls $\k\to+\infty$ we would retrieve $\Wrr^2(\rho_0,\rho_1)\to \W^2_\bO(\o_0,\o_1)+\W^2_{\G,\delta}(\g_0,\g_1)$, where the Wasserstein distance $\W^2_{\G,\delta}$ on the boundary should now be modeled on the underlying scaled distance $d_{\G,\delta}=\delta d_\G$.
The small toll limit should be more delicate:
Indeed in this case we expect to recover $\Wrr^2(\rho_0,\rho_1)\to \W^2_{\bO,\delta}(\varrho_0,\varrho_1)$, where $\W_{\bO,\delta}$ should now be induced by the distance $d_{\bO,\delta}$ on $\bO$ based on the heterogeneous mobility tensor $\mathbb K_\delta(x)$ taking values $1$ in $\O$ and $\delta>0$ on $\pO$.
This falls out of the scope of classical optimal transport on smooth Riemannian manifolds \cite{villani_BIG}, and how exactly the flux cost competes with this difference in mobility is not immediately clear.
(In particular the limits $\delta\to0,\delta\to+\infty$ should be far from being trivial).
\\

The rest of the paper is organized as follows: Section~\ref{sec:notations_preliminaries} fixes some notations and conventions to be used throughout.
In section~\ref{sec:definition_existence} we give the rigorous definition of our distance in a measure-theoretic context, prove that minimizers always exist, and characterize them in terms of a coupled system of Hamilton-Jacobi equations.
Section~\ref{sec:explicit_geodesics_dirac_masses} computes the distance between two Dirac masses, one on the boundary and one in the interior.
This allows to grasp the delicate balance between kinetic motion and flux in the minimization problem, and will also be useful for technical purposes in the sequel.
We then proceed in section~\ref{sec:geometric_topological_properties} with a qualitative study of the model, in particular we compare our distance with several other transportation distances and we investigate topological and geometrical properties of our metric space.
In section~\ref{sec:varying toll} we vary the flux parameter, and prove the convergence of the distance and geodesics in the small and large toll limits, $\k\to 0$ and $\k\to+\infty$.
Our last section~\ref{sec:Riemannian_formalism} contains a heuristic discussion on the formal Riemannian structure inherited from our new transport distance, which is very similar to and reminiscent from F. Otto's celebrated approach for Wasserstein optimal transport \cite{otto2001geometry}.
%%%%%%%%%%%%%%%%%%%%%%%%%%%%%%%%%%%%%%%%%%%%%%%%%%%%%%%%%%%%%%%%%%%%%%%%%%%%%%%%%%%%%%%%%%%%%%%%%%%%%%%%%%%%%%%%%%%%%%%%%%%%%%%%%%%%%%%%%%
%%%%%%%%%%%%%%%%%%%%%%%%%%%%%%%%%%%%%%%%%%%%%%%%%%%%%%%%%%%%%%%%%%%%%%%%%%%%%%%%%%%%%%%%%%%%%%%%%%%%%%%%%%%%%%%%%%%%%%%%%%%%%%%%%%%%%%%%%%
%
\section{Notations and preliminaries}
\label{sec:notations_preliminaries}
Throughout the whole paper $\O\subset \R^d$ will be a smooth bounded domain with boundary $\G:=\partial\O$.
The outer unit normal to $\pO$ is denoted by $n=n(x)$.
We consider $\G$ as a smooth submanifold of dimension $d-1$ without boundary, in particular no boundary terms will arise when integrating by parts on $\G$.
We will abuse notations and still write $\nabla=\nabla_\G$, $\dive=\diveg$ for the induced  gradient and divergence on the boundary.
(Subscripts will be used only when necessary depending on the context.)
For simplicity we shall often write
 $$
 \QO:=[0,1]\times\bO
 \qqtext{and}
 \QG:=[0,1]\times\G.
 $$
We collect below some definitions and notational conventions
\begin{itemize}
\item
If $(\mathcal X,d)$ is a Polish space we write $\M(\mathcal X)$, $\M^+(\mathcal X)$, and $\P(\mathcal X)$ for the space of Borel measures, nonnegative measures, and probability measures over $\mathcal X$, respectively.
\item
If $\mu\in\M(\mathcal X)$ and $\mathcal X'\subset \mathcal X$ we define the restriction $\nu=\mu\mrest\mathcal X'$ of $\mu$ to $\mathcal X'$ by $\nu(B'):=\mu(B')$ for all induced Borel sets $B'=B\cap \mathcal X'$.
\item
If $\mathcal X$ is a $d$-dimensional manifold we abuse notations and simply write $\M(\mathcal X)^d$ for vector-field measures (more rigorously, $1$-currents of order $0$).
If $Q_\mathcal X$ is a space-time domain of the form $[0,1]\times\mathcal X$ we also write $\M(Q_\mathcal X)^{d}$ for measures on $Q_\mathcal X$ taking values in the $d$-dimensional tangent plane $T_x\mathcal X$
\item 
The total variation of a measure $\mu\in \M(\mathcal X)$ is denoted by
$$
\TV{\mu}=\sup\left\{ \int_{\mathcal X }\varphi(x)\rd\mu(x):
\qquad \varphi\in \mathcal C_b(\mathcal X)\right\}
$$
with a similar definition for vector-valued measures.
\item
The variation of a measure $\mu$ is denoted by $|\mu|$, and we write $\mu\ll\nu$ when $\mu$ is absolutely continuous w.r.t. $\nu$ (i-e $|\nu|(B)=0\Rightarrow|\mu|(B)=0$).
We say that $\mu\in\M(\mathcal X)^k$ is supported on a set $S\subset \mathcal X$ if $|\mu|(B)=|\mu|(B\cap S)$ for all Borel sets $B\subset\mathcal X$.
\item
The narrow (weak-$\ast$) convergence of measures is defined by duality with bounded continuous functions,
$$
\mu_n\narrowcv\mu
\qqtext{iff.}
\int_{\mathcal X }\varphi(x)\rd\mu_n(x)
\to
\int_{\mathcal X }\varphi(x)\rd\mu(x),
\qquad \forall \varphi\in \mathcal C_b(\mathcal X)
$$
as $n\to+\infty$, with a similar definition for vector-valued measures.
\item
For a Borel-measurable map $T:\mathcal X\to\mathcal Y$ the \emph{pushforward} of a measure $\mu\in\M(\mathcal X)$ is the measure $\nu = T\pf\mu\in \M(\mathcal Y)$ defined by $\nu(B)=\mu(T^{-1}(B))$ for all Borel set $B\subset\mathcal Y$, or equivalently
$$
\int_{\mathcal Y}\phi(y)\rd(T\pf\mu)(y) =\int_{\mathcal X}\phi(T(x))\rd\mu(x)
$$
for all $\phi\in\C_b(\mathcal Y)$.
\item
If a time-space measure $\mu\in\M([0,1]\times\mathcal X)$ has a time marginal that is absolutely continuous w.r.t. the Lebesgue measure $\rd t$ on $[0,1]$ it can be disintegrated in time \cite[Theorem 5.3.1]{AGS}.
In that case we write $\mu_t\in\M(\mathcal X)$ for the $\rd t$-a.e. well-defined disintegration such that $\mu=\int_0^1(\delta_t\otimes\mu_t)$ and we abbreviate $\mu=\mu_t\rd t$.
\item
The bounded-Lipschitz distance between measures $\mu_0,\mu_1\in \M^+(\mathcal X)$ is
$$
d_{\mathrm{BL},\mathcal X}(\mu_0,\mu_1)=\sup
\left\{
\left|
\int_{\mathcal X}
\Phi\,\rd(\mu_1-\mu_0)
\right|
\qqtext{s.t.}
\|\Phi\|_\infty+ \operatorname{Lip}(\Phi)\leq 1
\right\}
$$
and is well known to metrize the narrow convergence of probability measures.
The space $\left(\P(\mathcal X),d_{\mathrm{BL},\mathcal X}\right)$ is complete \cite{dudley2018real}.
It is not difficult to prove that this extends to arbitrary positive Radon measures, and $(\mathcal M^+(\mathcal X),d_{\mathrm{BL},\mathcal X})$ is complete.
\item
A given measure $\g$ on $\G=\pO$ can always be extended to a measure $\bar\g\in\M(\bO)$ on $\bO$ through
$$
\int_{\bO}\phi(x)\rd\bar\g(x)
:=
\int_{\pO}\phi|_{\pO}(x)\rd\g(x),
\qquad
\forall\phi\in \C(\bO).
$$
Equivalently, $\bar\g$ is the unique measure on $\bO$ such that $\g=\bar\g\mrest\pO$ and supported on $\pO$.
In the sequel we will still write $\g$ for this extension with a slight abuse of notations and without further mention.
\item
We define
$$
\Pp(\bO):=
\Bigg\{
(\o,\g)\in\M^+(\bO)\times\M^+(\G)
\qqtext{s.t.}
\varrho:=\o+\g\in\P(\bO)
\Bigg\}
$$
\end{itemize}

Finally, let us state for the record a version of the Fenchel-Rockafellar duality theorem that will fit our purpose in section~\ref{sec:definition_existence}:
\begin{theo}[\cite{rockafellar1967duality}]
\label{theo:fenchel_rockafellar}
Let $E,F$ be normed vector spaces with topological duals $E^*,F^*$.
Take a continuous linear operator $L\in\mathcal L(E,F)$ with adjoint $L^*\in\mathcal L(F^*,E^*)$, and let $\mathcal F:E\to \R\cup\{-\infty\}$ and $\mathcal G:F\to \R\cup\{-\infty\}$ be two proper, concave, upper semi-continuous functions.
If there exists $x\in E$ such that $\mathcal F(x)$ is finite and $\mathcal G$ is continuous at $y=Lx\in F$ then
$$
\sup\limits_{x\in E} \left\{ \mathcal F(x)+\mathcal G(Lx)\right\}
=
\min\limits_{y^*\in F^*} \left\{ -\mathcal F^*(L^*y^*)-\mathcal G^*(y^*)\right\}.
$$
Moreover if there exists $y^* \in F^*$, $x\in E$ such that $L^*y^*\in\partial(-\mathcal F)(x)$ and $Lx\in\partial (-\mathcal G^*)(y^*)$ then $x$ achieves the $\sup$ and $y^*$ is a minimizer.
\end{theo}
Here $-\mathcal F^*,-\mathcal G^*$ are the Fenchel-Legendre conjugates of the convex functions $-\mathcal F,-\mathcal G$, and $\partial(-\mathcal F),\partial(-\mathcal G^*)$ are their subdifferentials.

%
%%%%%%%%%%%%%%%%%%%%%%%%%%%%%%%%%%%%%%%%%%%%%%%%%%%%%%%%%%%%%%%%%%%%%%%%%%%%%%%%%%%%%%%%%%%%
%%%%%%%%%%%%%%%%%%%%%%%%%%%%%%%%%%%%%%%%%%%%%%%%%%%%%%%%%%%%%%%%%%%%%%%%%%%%%%%%%%%%%%%%%%%%
%
\section{Existence and properties of minimizers}
\label{sec:definition_existence}
In this section $\rho_0=(\o_0,\g_0)$ and $\rho_1=(\o_1,\g_1)$ are given points of $\Pp(\bO)$.
As in the classical Benamou-Brenier setting \cite{BB}, our ring road distance $\Wrr(\rho_0,\rho_1)$ will be defined by minimizing an action functional among all possible pairs of solutions of the continuity equations interpolating between $\o_0,\o_1$ and $\g_0,\g_1$ --  see \eqref{eq:formal_def_Wrr}.
We stress that two continuity equations are needed, one for $\o$ and one for $\g$.
Neither are conservative, and both will have an associated action functional.
%
%
%%%%%%%%%%%%%%%%%%%%%%%%%%%%%%%%%%%%%%%%%%%%%%%%%%%%%%%%%%%%%%%%%%%%%%%%%%%%%%%%%%%%%%%%%%%%
%%%%%%%%%%%%%%%%%%%%%%%%%%%%%%%%%%%%%%%%%%%%%%%%%%%%%%%%%%%%%%%%%%%%%%%%%%%%%%%%%%%%%%%%%%%%
\subsection{Continuity equations and action functionals}
\label{subsec:CE}
The right setting is to use $\o,F,\g,G,f$ as independent variables in a measure-theoretic framework.
More precisely,
\begin{defi}[continuity equations with boundary interaction]
\label{def:CE}
For $\rho_0,\rho_1\in\Pp(\bO)$ we denote by $\mathcal{CE}(\rho_0,\rho_1)$ the set of tuples $\mu=(\mu_\O,\mu_\G)=(\o,F\ ;\ \g,G,f) \in \left(\M(\QO)\times \M(\QO)^d\right)\times\left(\M(\QG)\times \M(\QG)^{d-1}\times \M(\QG)\right)$ solving the continuity equations
$$
\left\{
\begin{array}{ll}
  \partial_t\o +\dive F =0  & \mbox{in }\O\\
  F\cdot n =f & \mbox{in }\partial\O
\end{array}
\right.
\qquad
\mbox{and}\qquad
\partial_t \g+\dive G =f
\mbox{ in }\G
$$
in the weak sense with initial/terminal data $\o_0,\o_1$ and $\g_0,\g_1$, respectively.
This is equivalent to
\begin{equation}
\label{eq:CE_omega}
\iint_\QO \partial_t\varphi \,\rd \o 
+
\iint_\QO \nabla\varphi\cdot \rd F
-
\iint_\QG \varphi\,\rd f
=
\int_\bO\varphi(1,.)\,\rd\o_1 - \int_\bO\varphi(0,.)\,\rd\o_0
\end{equation}
and
\begin{equation}
\label{eq:CE_gamma}
\iint_\QG \partial_t\psi \,\rd \g 
+ 
\iint_\QG \nabla\psi\cdot \rd G
+
\iint_\QG \psi \,\rd f
=
\int_\G\psi(1,.)\,\rd\g_1 - \int_\G\psi(0,.)\,\rd\g_0
\end{equation}
for all $\varphi\in\mathcal C^1(\QO)$ and all $\psi\in \mathcal C^1(\QG)$.
\end{defi}
Note that $\o$ must really be a measure on the whole $[0,1]\times\bO$ in order to allow for test functions to be $\C^1$ up to the boundary and encode the flux condition $F\cdot n=f$ on the boundary.
In particular, a solution $\o$ of \eqref{eq:CE_omega} is a priori allowed to (and in general does) charge $\pO$ for intermediary times, even when the endpoints $\o_0,\o_1$ do not.
Along the same lines, it is worth pointing out that $f$ should be thought of the normal flux of $F$ only if $\o,F$ are smooth enough, but this does not hold in our general measure-theoretic framework.
For example even for $F=0$, one can take for $\o\in\M^+([0,1]\times \bO)$ a singular measure supported only on the boundary, in which case our integral formulation \eqref{eq:CE_omega} simply means $\partial_t\o=-f$ in the sense of distributions in $(0,1)\times \pO$.
In this setting, and borrowing terminology from chemistry, the ``chemical component'' $\o$ can accumulate on the boundary while transforming into a $\g$ species according to the elementary stoichiometry $\o \xrightleftharpoons[-f]{+f}\g$.
In general $f$ can be thought of as the superposition of the normal flux $F\cdot n$ of $\o$ particles arriving from the interior $\O$ and hitting $\G$, combined with the effect of $\o$ particles already present on the boundary and being transformed into $\g$ species.
(One may think of two types of cars both located at a same toll area on $\G$, but still labeled $\o$ or $\g$ depending on which side of the toll gate they are currently driving.)
\\

As expected this formulation is automatically consistent with a global kinematics, i-e with a unique conservative continuity equation for the total density.
\begin{prop}
\label{prop:CE_rho}
Let $\rho_0,\rho_1\in\Pp(\bO)$ and $\mu=(\o,F\ ;\ \g,G,f)\in\CE(\rho_0,\rho_1)$, and let $\bar G\in \M(\QO)^{d}$ denote the extension of $G\in\M(\QG)^{d-1}$ first by zero in the normal direction $\bar G=(G,0)$ on $\G$ and then by zero on $\O$.
Then $\varrho:=\o+\g\in\M(\QO)$ and $H:= F+\bar G\in\M(\QO)^d$ solve
  \begin{equation}
  \label{eq:CE_rho}
\left\{
\begin{array}{ll}
  \partial_t\varrho +\dive H =0  & \mbox{in }\O\\
  H\cdot n =0 & \mbox{in }\partial\O
\end{array}
\right.
\end{equation}
in the weak sense with initial/terminal data $\varrho_0=\o_0+\g_0,\varrho_1=\o_1+\g_1$.
\end{prop}
\begin{proof}
Taking $\psi=\varphi|_\pO$ in \eqref{eq:CE_gamma}, the gradient $\nabla\psi=\nabla_\G\psi$ is nothing but the tangential gradient $\nabla_\G\left(\varphi|_\pO\right)=\nabla_\tau \varphi$, and by definition of $\bar G$ we can write $\nabla\psi\cdot \rd G=\nabla_\tau\varphi\cdot \rd G=(\nabla_\tau\varphi,\partial_n\varphi)\cdot (\rd G,0)=\nabla\varphi \cdot \rd \bar G$.
Summing the continuity equations \eqref{eq:CE_omega}\eqref{eq:CE_gamma} gives the weak formulation
$$
\iint_\QO \partial_t\varphi \,\rd (\o+\g)
+
\iint_\QO \nabla\varphi\cdot \rd (F+\bar G) 
% \\
=
\int_\bO\varphi(1,.)\,(\rd\o_1+\rd \g_1) - \int_\bO\varphi(0,.)\,(\rd\o_0+\rd \g_0) 
$$
for all $\varphi\in \C^1(\QO)$ as required.
\end{proof}
In order to measure kinetic energy let us first introduce the actions.
\begin{defi}[generalized Lagrangians]
\label{defi:A_O_A_G}
For $\mu_\O=(\o,F)\in \R\times \R^d$ and $\mu_\G=(\g,G,f)\in \R\times \R^{d-1}\times \R$ we let
$$
A_\O(\mu_\O):=
\left\{
\begin{array}{ll}
  \frac{|F|^2}{2\o} & \mbox{if }\o>0\\
  0 & \mbox{if }(\o,F)=(0,0)\\
  +\infty & \mbox{otherwise}
\end{array}
\right.
% $$
% and
% $$
\qqtext{and}
  A^\k_\G(\mu_\G):=
  \left\{
\begin{array}{ll}
\frac{|G|^2+\k^2 f^2}{2\g} & \mbox{if }\g>0\\
  0 & \mbox{if }(\g,G,f)=(0,0,0)\\
  +\infty & \mbox{otherwise}
\end{array}
\right.
$$
\end{defi}
It is worth pointing out that $A_\O$ is exactly the Lagrangian appearing in the definition \eqref{eq:def_formal_WO} of the Wasserstein distance, while $A^\k_\G$ is the Lagrangian in the definition \eqref{eq:def_WFR} of the Wasserstein-Fisher-Rao metrics.
In the sequel the quotients $\frac{|F|^2}{\o},\frac{|G|^2}{\g},\frac{f^2}{\g}$ should always be understood in this general sense.
Note that $A_\O,A^\k_\G$ are convex l.s.c. and $1$-homogeneous, allowing to define the corresponding functionals on the space of measures:
\begin{defi}[action functionals]
\label{defi:action_functional}
For $\mu=(\mu_\O,\mu_\G)=(\o,F\ ;\ \g,G,f)$ an element of $\Big(\M(\QO)\times \M(\QO)^d\Big) \times \Big( \M(\QG)\times \M(\QG)^{d-1}\times \M(\QG)\Big)$ we set
\begin{equation}
\label{eq:def_action_functional}
\mathcal A(\mu) :=
\iint_{\QO} A_\O \left(\frac{\rd \mu_\O}{\rd\lambda_\O}\right)\rd\lambda_\O
+
\iint_{\QG} A^\k_\G \left(\frac{\rd \mu_\G}{\rd\lambda_\G}\right)\rd\lambda_\G,
\end{equation}
where $(\lambda_\O,\lambda_\G)\in\M^+(\QO)\times\M^+(\QG)$ are any two nonnegative Borel measures such that $|\mu_\O|\ll\lambda_\O$ and $|\mu_\G|\ll\lambda_\G$.
Since $A_\O$ and $A^\k_\G$ are $1$-homogeneous this definition does not depend on the choice of $\lambda_\O,\lambda_\G$.
\end{defi}
Clearly $\A$ is convex, $1$-homogeneous, and standard results \cite[Theorem 3.3]{bouchitte1990new} show that $\A$ is moreover lower semicontinuous w.r.t. the (sequential) narrow convergence of measures.
As can be expected, solutions of the continuity equations enjoy some nice properties, particularly those with finite action:
\begin{prop}[properties of solutions of continuity equations]
\label{prop:disintegration_momentum_velocity}
Any $\mu=(\o,F\ ;\ \g,G,f)\in \CE (\rho_0,\rho_1)$ can be disintegrated in time as
$$
\rd \o(t,x)=\rd \o_t(x)\rd t
\qqtext{and}
\rd \g(t,x)=\rd \g_t(x)\rd t.
$$
If moreover $\mathcal A(\mu)<+\infty$ then
\begin{enumerate}[(i)]
\item
\label{item:og>0_and_FGf_asolute_continuity}
The measures $\o,\g$, and $\varrho=\o+\g$ are nonnegative, and
\begin{equation}
  \label{eq:|o_t|+|g_t|=|rho_t|=1}
  \TV{\o_t}+\TV{\g_t}=\TV{\varrho_t}=1
  \qquad\mbox{for a.e. }t\in[0,1].
\end{equation}
Moreover $|F|\ll \o$ and $|G|+|f|\ll\g$.
\item 
The Radon-Nikodym densities
$$
u_t(x):=\frac{\rd F}{\rd\o}(t,x)
,\qquad
v_t(x):=\frac{\rd G}{\rd\g}(t,x)
,\qquad
r_t(x):=\frac{\rd f}{\rd\g}(t,x),
$$
are well-defined $\rd\o,\rd\g$ a.e. and
\begin{multline}
  \label{eq:action_velocity_reaction}
  \mathcal A(\mu)=
  \frac 12\iint_{\QO} |u|^2\rd\o
+
\frac 12\iint_{\QG} \left(|v|^2 + \k^2 |r|^2\right)\rd\g
\\
=
\frac 12\int_0^1\int_{\bO} |u_t|^2\rd\o_t \rd t
+
\frac 12\int_0^1\int_\G (|v_t|^2+\k^2 r_t^2)\rd\g_t \rd t
\end{multline}
\item
The curves $t\mapsto\o_t\in\M(\bO)$ and $t\mapsto \g_t\in\M(\G)$ are narrowly continuous and satisfy the bounded-Lipschitz estimate
\begin{equation}
\label{eq:estimate_dBL}
d_{\mathrm{BL},\bO}(\o_s,\o_t) + d_{\mathrm{BL},\G}(\g_s,\g_t)
\leq C_\k\sqrt{\A(\mu)}|t-s|^{\frac 12}
\hspace{1cm}
s,t\in[0,1]
\end{equation}
with $C_\k=4\max(1,1/\k)$.
In particular the initial/terminal conditions hold in the narrow sense.
\end{enumerate}
\end{prop}
\begin{proof}
Regarding the disintegration, we only give the details for $\o$ since the argument is identical for $\g$.
Let $\pi(t,x)=t$ be the time projection.
In order to disintegrate $\o$ it suffices by \cite[Theorem 5.3.1]{AGS} to show that the time marginal $\bar\o:=\pi\pf\o$ is absolutely continuous with respect to the Lebesgue measure $\rd t$ on $[0,1]$.
Recall that, by definition of the pushforward, $\bar\o$ is defined by the identity
$$
\int_0^1 \xi(t)\rd\bar\o(t)
= \iint_{\QO}\xi(t)\rd\o(t,x),
\qquad\mbox{for all }\xi\in\C([0,1]).
$$
Fix an arbitrary $\xi\in\C([0,1])$ and let $\zeta(t):=\int_0^t\xi(s)\rd s\in \C^1([0,1])$.
Testing $\varphi(t,x)=\zeta(t)$ in the continuity equation gives
\begin{multline*}
\left| \int_0^1 \xi(t)\rd\bar\o(t) \right| 
= \left| \iint_{\QO}\xi(t) \rd\o(t,x)\right|
= \left| \iint_{\QO}\partial_t \varphi \rd\o \right|
\\
\overset{\eqref{eq:CE_omega}}{=} \Bigg|-\iint_{\QO}\cancel{\nabla\varphi}\cdot \rd F +  \int_\bO\varphi(1,.)\rd\o_1 - \int_\bO\cancel{\varphi(0,.)}\,\rd\o_0
+\iint_{\QG} \varphi\,\rd f  \Bigg|
\\
=\Bigg|\zeta(1)\int_\bO\rd\o_1  +\iint_{\QG} \zeta(t)\,\rd f(t,x)  \Bigg|
\\
\leq 
(\TV{\o_1}+\TV{f})\cdot\|\zeta\|_{L^\infty}
\leq
(\TV{\o_1}+\TV{f}) \cdot \|\xi\|_{L^1}.
\end{multline*}
Since $\o_1$ and $f$ have finite masses this shows that $\mathcal L(\xi):=\int_0^1 \xi(t)\rd\bar\o(t)$ is continuous for the $L^1(\rd t)$ norm, thus $\mathcal L$ can be extended from $\C([0,1])$ to the whole space $L^1(\rd t)$ as a continuous linear form.
This means that the measure $\bar\o\in\M([0,1])$ is in fact of the form $\rd\bar\o(t)=L(t)\rd t$ for some bounded function $L$ such that $\|L\|_\infty=\|\mathcal L\|_{(L^1)'}\leq \TV{\o_1}+\TV{f}$.
This is actually even stronger than what we need, and entails the disintegration part of our statement.
\\
Assume now that $\A(\mu)<+\infty$.
\\
\underline{\bf (i)}
We only give the details for $\g,G,f$, the argument is identical for $\o,F$.
Note that we can always choose the reference measure $\lambda_\G:=|\g|+|G|+|f|$ in \eqref{eq:def_action_functional}.
We write below $\tilde\g,\tilde G,\tilde f$ for the corresponding Radon-Nikodym densities.
Assume by contradiction that $\g$ is not nonnegative: Then there exists a Borel set $B\subset \QO$ such that $\lambda_\G(B)\geq |\g|(B)>0$, and $\tilde \g(x)<0$ for $\lambda_\G$-a.e. $x\in B$.
According to Definition~\ref{defi:A_O_A_G} this means $A^\k_\G\left(\frac{\rd \mu_\G}{\rd\lambda_\G}\right)=A^\k_\G(\tilde\g,\tilde G,\tilde f)=+\infty$ on $B$, thus $\mathcal A(\mu)\geq \mathcal A^\k_\G(\mu_\G)\geq \int_B A^\k_\G(\tilde\g,\tilde G,\tilde f)\rd\lambda_\G=+\infty$.\\
Now that we know $\g\geq 0$, assume by contradiction that $|G|$ is not absolutely continuous w.r.t $\g$.
Then there is a Borel set $B$ such that $\g(B)=0$ but $|G|(B)>0$, in particular $\lambda_\G(B)\geq |G|(B)>0$.
But then $\tilde \g(x)\equiv 0$ while $\tilde G(x)\not\equiv 0$ on $B$.
Choosing any subset $B'\subset B$ such that $\lambda_\G(B')>0$ and $\tilde G(x)\not=0$ on $B'$, we see that $A^\k_\G(\tilde \g,\tilde G,\tilde f)(x)=+\infty$ for $\lambda_\G$-a.e. $x\in B'$ and therefore $\A^\k_\G(\mu_\G)=+\infty$ as before.
The absolute continuity $|f|\ll\g$ is obtained similarly.\\
By the previous steps $\varrho=\o+\g\geq 0$ disintegrates in time, and by Proposition~\ref{prop:CE_rho} $\varrho$ also solves the conservative continuity equation \eqref{eq:CE_rho}.
This classically implies the mass conservation $\TV{\varrho_t}=\TV{\varrho_0}=\TV{\varrho_1}=1$, which gives of course $\TV{\o_t}+\TV{\g_t}=\TV{\o_t+\g_t}=\TV{\varrho_t}$ due to $\o_t,\g_t\geq 0$.
\\
\underline{\bf (ii)}
In order to get \eqref{eq:action_velocity_reaction}, the first step allows to define $u(t,x):=\frac{\rd F}{\rd\o}(t,x)$ and $v(t,x):=\frac{\rd G}{\rd\o}(t,x),r(t,x):=\frac{\rd f}{\rd\o}(t,x)$, but also allows to choose $\lambda_\O=\o$ and $\lambda_\G=\g$ as reference measures in \eqref{eq:def_action_functional}.
The corresponding Radon-Nikodym densities are then $\tilde \o:=\frac{\rd \o}{\rd \lambda_\O}=\frac{\rd \o}{\rd \o}=1$ and $\tilde F:=\frac{\rd F}{\rd \lambda_\O}=\frac{\rd F}{\rd \o}=u$, thus
$$
\A_\O(\mu_\O)
=\iint_{\QO}A_\O(\tilde\o,\tilde F)\rd \gamma_\O
=\iint_{\QO}A_\O(1,u)\rd \o 
=\iint_{\QO}\frac {|u|^2}{2}\rd \o .
$$
Similarly, $\tilde \g=1$ and $\tilde G=v,\tilde f=r$ in \eqref{eq:def_action_functional} gives
$$
\A^\k_\G(\mu_\G)
=\iint_{\QG}A^\k_\G(\tilde\g,\tilde G,\tilde f)\rd \gamma_\O
=\iint_{\QG}A^\k_\G(1,v,r)\rd \g 
=\iint_{\QG} \frac {|v|^2+\k^2 r^2}{2}\rd \g.
$$
Let us now address the second equality in \eqref{eq:action_velocity_reaction}.
Because $\o$ and $\g$ disintegrate in time, step (i) shows that $F,G,f$ do too. 
Clearly the corresponding $F_t,G_t,f_t$ must be absolutely continuous w.r.t. $\o_t,\g_t$ for a.e. time.
In other words we can write unambiguously $\frac{\rd F}{\rd \o}(t,x)=u(t,x)= u_t(x)=\frac{\rd F_t}{\rd\o_t}(x)$, with equality $\o=\o_t\rd t$ almost everywhere. (Ditto for $v=\frac{\rd G}{\rd\g},r=\frac{\rd f}{\rd \g}$.)
The second equality in \eqref{eq:action_velocity_reaction} follows.\\
\underline{\bf (iii)}
Because the bounded-Lipschitz distance metrizes the narrow convergence of measures it suffices to establish \eqref{eq:estimate_dBL}.
We only give the proof for $t\mapsto\o_t$, the argument is similar for $\g_t$.
For fixed $\Phi\in\C^1(\bO)$ we will estimate below the derivative of
$$
l(t):=\int_\bO\Phi(x)\rd\o_t(x).
$$
Note that, due to the disintegration $\int_0^1\TV{\o_t}\rd t=\TV{\o}<+\infty$, the function $l\in L^1(0,1)$ can legitimately be considered as a distribution $\mathcal D'(0,1)$.
To compute its distributional derivative $l'$, pick an arbitrary $h\in\C^\infty_c(0,1)$ and let $\varphi(t,x)=h(t)\Phi(x)$.
Then \eqref{eq:CE_omega} with $\rd\o(t,x)=\rd\o_t(x)\rd t$, $\rd F(t,x)=u_t(x)\rd\o_t(x)\rd t$, and $\rd f(t,x)=r_t(x) \rd\g_t(x)\rd t$ from the previous step, gives
\begin{multline*}
  \langle l',h\rangle_{\mathcal D',\mathcal D}=-\langle l,h'\rangle_{\mathcal D',\mathcal D}
  =
  - \int_0^1 \left(\int_\bO\Phi(x)\rd\o_t(x)\right) h'(t)\rd t
  =
  -\iint_{\QO}\partial_t\varphi \rd \o
  \\
  =
  \iint_{\QO}\nabla\varphi\cdot \rd F
  -
  \iint_{\QG}\varphi\, \rd f
  +
  \int_\bO \underbrace{\varphi(0,.)}_{=0}\,\rd\o_0 - \int_\bO \underbrace{\varphi(1,.)}_{=0}\,\rd\o_1
  \\
  =\int_0^1 h(t) \underbrace{\left(\int_\bO \nabla \Phi(x)\cdot u_t(x)\rd\o_t(x)   \right)}_{:=m_1(t)}\rd t
  -\int_0^1 h(t) \underbrace{\left(\int_\G \Phi(x) r_t(x)\rd\g_t(x)   \right)}_{:=m_2(t)}\rd t
\end{multline*}
and shows that $l'=m_1-m_2$.
Since $\TV{\o_t},\TV{\g_t}\leq \TV{\rho_t}=1$ from the previous step, we have $\|u_t\|_{L^1_{\o_t}}\leq \|u_t\|_{L^2_{\o_t}}$ and $\|r_t\|_{L^1_{\g_t}}\leq \|r_t\|_{L^2_{\g_t}}$.
Whence by \eqref{eq:action_velocity_reaction}
\begin{multline*}
|l'(t)|\leq  |m_1(t)|+|m_2(t)|
%  \\
\leq \|\nabla\Phi\|_{_\infty} \int_\bO |u_t|\rd\o_t
+\|\Phi\|_{_\infty} \int_\G |r_t|\rd\g_t
\\
\leq \sqrt{2}(\|\Phi\|_\infty+\|\nabla\Phi\|_\infty)
\left(
\int_\O|u_t|^2\rd\o_t
+
\int_\G|r_t|^2\rd\g_t
\right)^{\frac 12}\in L^2(0,1).
\end{multline*}
Thus $l$ is absolutely continuous, and by Cauchy-Schwartz inequality
\begin{multline*}
  \left|
\int_{\bO}
\Phi\,\rd(\o_t-\o_s)
\right|
=
|l(t)-l(s)|
\leq \int_s^t |l'(\tau)|\rd\tau
\leq \|l'\|_{L^2(0,1)}
\\
\leq 
\sqrt{2}(\|\Phi\|_\infty+\|\nabla\Phi\|_\infty) \left(
\int_0^1\int_\bO|u_t|^2\rd\o_t\rd t
+
\int_0^1\int_\G|r_t|^2\rd\g_t\rd t
\right)^{\frac{1}{2}}
\\
=2(\|\Phi\|_\infty+\|\nabla\Phi\|_\infty) \left(
\int_0^1\int_\bO\frac{|u_t|^2}{2}\rd\o_t\rd t
+
\frac 1{\k^2} \int_0^1\int_\G\frac{\k^2|r_t|^2}2 \rd\g_t\rd t
\right)^{\frac{1}{2}}
\\
\leq  2\max(1,1/\kappa)\sqrt{\A(\mu)}(\|\Phi\|_\infty+\|\nabla\Phi\|_\infty) .
\end{multline*}
This entails the first half of \eqref{eq:estimate_dBL} for the interior $\o$ component.
The estimate for the boundary component $\g$ is established similarly and we omit the details.
\end{proof}
The (squared) ring road distance is then
\begin{defi}
\label{defi:Wrr}
For $\rho_0,\rho_1\in\Pp\left(\bO\right)$ we set
\begin{equation}
\label{eq:def_Wrr}
\Wrr^2(\rho_0,\rho_1):=\inf\limits_{\mu\in\mathcal{CE}(\rho_0,\rho_1)}\mathcal A(\mu).
\end{equation}
\end{defi}
This is always well-defined
\begin{lem}
\label{lem:Wrr_finite}
The quantity $\Wrr(\rho_0,\rho_1)$ is always finite.
\end{lem}
\begin{proof}
  Pick any point $y\in\partial\O$.
We will show below that any $\rho_0\in\Pp(\bO)$ can be connected to $(0,\delta_y)\in\Pp(\bO)$ with finite cost: By symmetry $(0,\delta_y)$ can also be connected to any other $\rho_1$, thus connecting $\rho_0$ to $\rho_1$ with finite cost.
\\
  Note that scaling time $s= \tau t$ and $(F_s,G_s,f_s)\leftrightarrow \frac{1}{\tau}(F_t,G_t,f_t)$ gives an inverse scaling for the action $\mathcal A^\tau=\int_0^\tau(\dots)\rd s=\frac 1\tau\mathcal A$.
  (Moving slower in time $s\in[0,\tau],\tau>1$ takes lesser energy than in time $t\in[0,1]$.)
Therefore it is enough to show that $\rho_0$ can be connected to $(0,\delta_y)$ in a finite number of elementary steps, each occurring in time one with finite cost, and then ultimately scaling back to $t\in[0,1]$ will do.
\begin{enumerate}
  \item
  Pick first an interior Wasserstein geodesic $(\o,F)$ with zero-flux from $\o_0$ to $\tilde\o_0:=\TV{\o_0}\delta_y$, and a boundary Wasserstein geodesic $(\g,G)$ from $\g_0$ to $\tilde \g_0:=\TV{\g_0}\delta_y$.
  Setting $\mu:=(\o,F,\g,G,0)$ gives a solution of the generalized continuity equation \eqref{eq:CE_omega}\eqref{eq:CE_gamma} connecting $\rho_0=(\o_0,\g_0)$ to $\tilde\rho_0:=(\tilde\o_0,\tilde\g_0)$ with cost
  $$
  \Wrr^2(\rho_0,\tilde\rho_0)=\A_\O(\mu_\O)+\A^\k_\G(\mu_\G)= \W_\bO^2(\o_0,\tilde\o_0)+\W_\G^2(\g_0,\tilde\g_0)<+\infty.
  $$
  \item
  In order to connect now $(\tilde\o_0,\tilde\g_0)=(\TV{\o_0}\delta_y,\TV{\g_0}\delta_y)$ to $(0,\delta_y)$ we use a pure Fisher-Rao geodesic between $\tilde\g_0=\TV{\g_0}\delta_y$ and $\tilde \g_1=\delta_y$ in \eqref{eq:def_formal_FR}.
  The latter is given explicitly by $\g_t:=[(1-t)\sqrt{\TV{\g_0}}+t]^2\delta_y$ and $f_t:=\partial_t\g_t=2(1-\sqrt{\TV{\g_0}})\times [(1-t)\sqrt{\TV{\g_0}}+t]\delta_y$ -- see \cite[Proposition 4.2]{chizat2018interpolating}.
  In order to absorb the mass variation we simply enforce $\partial_t\o_t=-f_t$ on the boundary with no motion whatsoever, in other words we set $\o_t:=\tilde\o_0-\int_0^tf_s\rd s$ and $F_t=G_t:=0$.
  It is easy to check that $\o_t$ remains nonnegative due to the initial mass constraint $\TV{\o_0}+\TV{\g_0}=1$.
  The path $\mu:=(\o_t,0\ ;\ \g_t,0,f_t)\rd t$ connects now the desired endpoints with cost
  \begin{multline*}
  \Wrr^2((\TV{\o_0}\delta_y,\TV{\g_0}\delta_y),(0,\delta_y))
  \leq
  0 + \int_0^1\int_\G\frac{0+\k^2|f_t|^2}{2\g_t}\rd t
  \\
  =
  \FR^2(\TV{\g_0}\delta_y,\delta_y)
  =
  2\k^2\left(1-\sqrt{\TV{\g_0}}\right)^2
  <+\infty
  \end{multline*}
  and the proof is complete.
  \end{enumerate}
\end{proof}
% 
% 
%%
%%%%%%%%%%%%%%%%%%%%%%%%%%%%%%%%%%%%%%%%%%%%%%%%%%%%%%%%%%%%%%%%%%%%%%%%%%%%%
\subsection{Existence}
In this section we address the existence of minimizers $\mu$ in \eqref{eq:def_Wrr} and derive the equations for the geodesics.
This will involve infinite-dimensional convex analysis, and we start with preliminary material.
We define the ``subsolution'' sets
\begin{equation}
\label{eq:def_D_Omega}
S_\O:=\left\{
(\alpha,\beta)\in \R\times \R^d
:\qquad
\alpha+\frac{|\beta|^2}{2}\leq 0
\right\},
\end{equation}
\begin{equation}
\label{eq:def_D_Gamma}
S^\k_\G:=\left\{
(a,b,c,d)\in \R\times \R^{d-1}\times \R\times \R
:\qquad
a+\frac{|b|^2}{2}+\frac{|c-d|^2}{2\k^2}\leq 0
\right\}
\end{equation}
as well as the convex indicators
$$
\iota_{S_\O}(\alpha,\beta):=
\left\{
\begin{array}{ll}
0 & \mbox{if }(\alpha,\beta)\in S_\O
\\
  +\infty & \mbox{otherwise}
\end{array}
\right.
$$
and
$$
\iota_{S^\k_\G}(a,b,c,d):=
\left\{
\begin{array}{ll}
0 & \mbox{if }(a,b,c,d)\in S^\k_\G
\\
  +\infty & \mbox{otherwise}
\end{array}
\right.
.
$$
The variables $\alpha,\beta$ will be dual multipliers for $\o,F$, and $a,b$ will be dual to $\g,G$.
Due to the nonstandard bulk/interface coupling we shall actually need \emph{two} separate extra multipliers $c,d$ for the remaining boundary flux, and one should roughly think below of $c-d$ as being dual to $f$.
We will typically take $(\alpha,\beta)=(\partial_t\phi,\nabla\phi)$ and $(a,b,c,d)=(\partial_t\psi,\nabla\psi,\psi,\phi|_{\pO})$ for suitable test-functions $\phi\in\C^1([0,1]\times\bO),\psi\in \C^1([0,1]\times\G)$.
Accordingly, $(\partial_t\phi,\nabla\phi)\in S_\O$ and $(\partial_t\psi,\nabla\psi,\psi,\phi|_{\pO})\in S^\k_\G$ mean that $\phi,\psi$ are (smooth) subsolutions of the Hamilton-Jacobi system
$$
\partial_t\phi+\frac 12|\nabla\phi|^2 \leq 0
\qqtext{and}
\partial_t\psi+\frac 12 |\nabla\psi|^2+\frac 1{2\k^2}|\psi-\phi|^2 \leq 0.
$$
Note that this coupled system of Hamilton-Jacobi equations is invariant by addition of a common constant $\phi+k,\psi+k$ and that the convex closed set $S^\k_\G$ is thus invariant under diagonal shifts $c+k,d+k$.
\\
As in the Benamou-Brenier approach \cite{BB}, the key is to identify the actions $A_\O,A^\k_\G$ as the support-functions of $S_\O,S^\k_\G$.
More precisely,
\begin{lem}
\label{lem:computation_iotaG*}
For $(\g,G,f,\eta)\in\R\times\R^{d-1}\times\R\times\R$ the convex conjugate $\iota^*_{S^\k_\G}$ of $\iota_{S^\k_\G}$ is
$$
\iota^*_{S^\k_\G}(\g,G,f,\eta) =\left\{
\begin{array}{ll}
    \frac{|G|^2+\k^2 f^2}{2\g} & \mbox{if }\g>0 \mbox{ and }f +\eta=0
    \\
0 & \mbox{if }(\g,G,f,\eta)=(0,0,0,0)
\\
  +\infty & \mbox{otherwise}
\end{array}
\right.
$$
\end{lem}
A possible alternative formulation is $\iota^*_{S^\k_\G}(\g,G,f,\eta)=\bar A^\k_\G(\g,G,f,\eta)$, where the \emph{extended action} on $\G$ is
\begin{equation}
\label{eq:def_A_extended}
\bar A^\k_\G(\g,G,f,\eta):=A^\k_\G(\g,G,f) +
\left\{
\begin{array}{ll}
0 & \mbox{if }f+\eta=0\\
+\infty & \mbox{otherwise}
\end{array}
\right.,
\end{equation}
and $A^\k_\G$ is as in Definition~\ref{defi:A_O_A_G}.
Note that $\bar A^\k_\G$ is convex, l.s.c., and one-homogeneous (as a convex conjugate, it is a supremum of linear functions). 
The condition $f+\eta=0$ reflects by duality the invariance of $S^\k_\G$ under $c+k,d+k$ discussed earlier.
We have similarly
\begin{lem}
\label{lem:computation_iotaO*}
For $(\o,F)\in\R\times\R^{d}$ the convex conjugate $\iota^*_{S_\O}$ of $\iota_{S_\O}$ is
$$
\iota^*_{S_\O}(\o,F) =A_\O(\o,F)=\left\{
\begin{array}{ll}
    \frac{|F|^2}{2\o} & \mbox{if }\o>0\\
0 & \mbox{if }(\o,F)=(0,0)\\
%  +\infty & \mbox{if }p<0 \mbox{ or } r+s\neq 0\\
  +\infty & \mbox{otherwise}
\end{array}
\right..
$$
\end{lem}
The proof of these two results relies on elementary finite-dimensional convex analysis and we omit the details.\\

Let us write for brevity
$$
E:=\C^1(\QO)\times \C^1(\QG),
$$
and for $(\phi,\psi)\in E$ define the primal objective functional
\begin{multline}
\label{eq:def_J_Wrr}
\mathcal J^\k(\phi,\psi):=
  \int_\bO\phi(1,x)\rd\o_1(x)-\int_\bO\phi(0,x)\rd\o_0(x)  +  \int_\G \psi(1,x)\rd\g_1(x)  -  \int_\G\psi(0,x)\rd\g_0(x) 
  \\
  -
\iint_{\QO} \iota_{S_\O}(\partial_t\phi,\nabla\phi)\rd x\rd t
- \iint_{\QG} \iota_{S^\k_\G}(\partial_t\psi,\nabla\psi,\psi,\phi|_{\pO})\rd x \rd t.
\end{multline}
The main result in this section is
\begin{theo}
\label{theo:duality_exist_geodesics}
For fixed $\rho_0,\rho_1\in\Pp(\bO)$ we have duality
$$
  \Wrr^2(\rho_0,\rho_1)=\sup_{(\phi,\psi)\in E} \mathcal J^\k(\phi,\psi),
$$
and $\Wrr$-geodesics exist in the sense that $\inf=\min$ is attained in \eqref{eq:def_Wrr}.
\end{theo}
Note carefully that the objective functional $\mathcal J^\k$ only depends on $\k$ through the second indicator $\iota_{S^\k_\G}$ encoding the Hamilton-Jacobi constraint $\partial_t\psi +\frac 12 |\nabla\psi|^2+\frac{1}{2\k^2}|\psi-\phi|^2\leq 0$ on $\G$.
This will be important in section~\ref{sec:varying toll} when we take the limits $\k\to 0$ and $\k\to+\infty$.
\begin{proof}
We closely follow the lines of \cite[Theorem 2.1]{chizat2018interpolating}.
The strategy of proof consists in identifying the minimization $\Wrr^2=\inf(\dots)$ as the dual problem to the primal maximization $\sup \mathcal J^\k$, and applying the Fenchel-Rockafellar duality theorem~\ref{theo:fenchel_rockafellar}.

We first define the unfolding operator
\begin{equation*}
\begin{array}{lrcl}
  L: &  E & \rightarrow & F\\
  & (\phi,\psi) & \mapsto & (\partial_t\phi,\nabla\phi\ ;\ \partial_t\psi,\nabla\psi,\psi,\phi|_{\pO})
\end{array}
\end{equation*}
Note that $L$ is obviously continuous for the natural $\mathcal C^1 $ and $\mathcal C^0$ topologies on $E,F$, respectively.
The primal problem $\sup \mathcal J^\k$ reads
$$
\sup\limits_{(\phi,\psi)\in E}\,
\Big\{ 
\mathcal{F}(\phi,\psi) + \mathcal G(L(\phi,\psi))
\Big\}
$$
with
$$
\mathcal F(\phi,\psi) := \int_\bO\phi(1,.) \,\rd \o_1-\int_\bO\phi(0,.) \,\rd \o_0  +  \int_\G \psi(1,.) \,\rd \g_1-\int_\G\psi(0,.) \,\rd \g_0
$$
and
$$
\mathcal{G}(\alpha,\beta;a,b,c,d):=-\int_0^1\int_\O \iota_{S_\O}(\alpha,\beta)\rd x\rd t
-\int_0^1\int_\G \iota_{S^\k_\G}(a,b,c,d) \rd x\rd t.
$$

Note that, because $\mathcal F$ is linear continuous and since $\iota_{S_\O},\iota_{S^\k_\G}$ are convex l.s.c, both $\mathcal F$ and $\mathcal G$ are concave, proper, u.s.c. functionals.
It is not hard to find at least a pair $(\phi,\psi)$ such that $\mathcal G$ is continuous at $L(\phi,\psi)$ and $\F(\phi,\psi)<+\infty$. (Take for example $\phi(t,x)=-t$ and $\psi(t,x)=-t$, which are \emph{strict} subsolutions of the Hamilton-Jacobi equations.)
The Fenchel-Rockafellar theorem~\ref{theo:fenchel_rockafellar} therefore guarantees that
\begin{equation}
\label{eq:J_fenchel_rockafellar}
\sup\limits_{(\phi,\psi)\in E} \mathcal J^\k=
\inf\limits_{\nu\in F^*}
\,\Big\{
-\mathcal F^*(-L^*\nu) -\mathcal G^*(\nu)
\Big\}
\end{equation}
where $-\mathcal F^*=(-\F)^*,-\mathcal G^*=(-\mathcal G)^*$ are the Fenchel-Legendre (convex) conjugates of the convex functions $-\mathcal F,-\mathcal G$, respectively.
Here $L^*:F^*\to E^*$ is the adjoint of $L$, and the target dual space identifies to
$$
F^*= \Big( \M(\QO)\times \M(\QO)^d \Big) \times 
  \Big(
  \M(\QG)\times \M(\QG)^{d-1}\times \M(\QG) \times \M(\QG)
  \Big)
$$
with elements denoted by
$$
\nu=(\nu_\O,\nu_\G)=(\o,F\ ;\ \g,G,f,\eta)\in F^*.
$$
(We use the notation $\nu$ instead of the previous $\mu=(\mu_\O,\mu_\G)=(\o,F\ ;\ \g,G,f)$ to emphasize the augmented scalar variable $\eta$.)
Let us compute separately the two conjugates in \eqref{eq:J_fenchel_rockafellar}.
\begin{itemize}
\item 
By definition of the Legendre-Fenchel transform we have
\begin{multline*}
-\mathcal F^*(-L^*\nu)=\sup\limits_{(\phi,\psi)\in E}\left\{
\langle -L^*\nu,(\phi,\psi)\rangle_{E^*,E} - (
-\mathcal F)(\phi,\psi)  
\right\}\\
=\sup\limits_{(\phi,\psi)\in E}\left\{
\mathcal F(\phi,\psi) - \langle \nu,L(\phi,\psi)\rangle_{F^*,F}
\right\}
\\
=\sup\limits_{(\phi,\psi)\in E}
\Bigg\{
\int_\bO\phi(1,.)\,\rd\o_1-\int_\bO\phi(0,.)\,\rd\o_0  +  \int_\G \psi(1,.)\,\rd\g_1-\int_\G\psi(0,.)\,\rd\g_0\\
  -\left(\iint_{\QO} \partial_t\phi\, \rd \o
  +\iint_{\QO} \nabla\phi \cdot \rd F\right)\\
  -\left(\iint_{\QG} \partial_t \psi \, \rd \g
  +\iint_{\QG} \nabla \psi \cdot \rd G
  + \iint_{\QG} \psi \, \rd f
  + \iint_{\QG} \phi|_{\pO}\, \rd \eta\right)
\Bigg\}
\end{multline*}
We recognize at once the convex indicator of the continuity equations with endpoints $\o_i,\g_i$ and boundary flux $-\eta$, in other words
\begin{equation}
\label{eq:computation_F*-L}
  -\mathcal F^*(-L^*\nu)
=\left\{
\begin{array}{ll}
  0 & \mbox{if }
  \left\{
  \begin{array}{l}
  \partial_t\o +\dive F =0  \mbox{ with }F\cdot n|_\pO=-\eta \mbox{ and }\o|_{t=0,1}=\o_{0,1}\\
%      \mbox{and}\\
    \partial_t\g +\dive G =f \mbox{ with }\g|_{t=0,1}=\g_{0,1}
    \end{array}
    \right.
    \\
    \\
    \vspace{.2cm}
    +\infty & \mbox{otherwise}
\end{array}
\right.
\end{equation}
Here the equations and initial-terminal/boundary conditions should be understood in the integral sense as in Definition~\ref{def:CE}.
\item
For the second conjugate in \eqref{eq:J_fenchel_rockafellar} we denote by $\xi=(\alpha,\beta\ ;\ a,b,c,d)$ a generic element in $F$, and  by $\nu=(\nu_\O,\nu_\G)=(\o,F\ ;\ \g,G,f,\eta)$ the dual elements of $F^*$.
We compute then
\begin{multline*}
-\mathcal G^*(\nu) = \sup\limits_{\xi\in F}
\Big\{
<\nu,\xi>_{F*,F} + \mathcal G(\xi)
\Big\}\\
=\sup\limits_{(\alpha,\beta\ ;\ a,b,c,d)\in F}
\Bigg\{
\iint_{\QO}\alpha\,\rd \o
+\iint_{\QO}\beta\cdot \rd F\\
+ \iint_{\QG}a\, \rd \g
+ \iint_{\QG}b\cdot  \rd G
+ \iint_{\QG} c\, \rd f
+ \iint_{\QG}d\, \rd \eta\\
-\iint_{\QO} \iota_{S_\O}(\alpha,\beta)\rd x\,\rd t
-\iint_{\QG} \iota_{S^\k_\G}(a,b,c,d)\rd x\,\rd t
\Bigg\},
\end{multline*}
and this clearly uncouples as
\begin{multline*}
-\mathcal G^*(\nu) = 
\sup\limits_{(\alpha,\beta)}
\Bigg\{
\iint_{\QO}\alpha\,\rd \o
+\iint_{\QO}\beta\cdot \rd F
-\iint_{\QO} \iota_{S_\O}(\alpha,\beta)\rd x\,\rd t
\Bigg\}\\
+
\sup\limits_{(a,b,c,d)}
\Bigg\{
  \iint_{\QG}a\, \rd \g
+ \iint_{\QG}b\cdot  \rd G
+ \iint_{\QG} c\, \rd f
+ \iint_{\QG}d\, \rd \eta
-\iint_{\QG} \iota_{S^\k_\G}(a,b,c,d)\rd x\,\rd t
\Bigg\}.
\end{multline*}
Applying \cite[Theorem 5]{rockafellar1971integrals} allows to ``take the convex conjugation under the integral sign'', and exploiting lemmas~\ref{lem:computation_iotaG*}\ref{lem:computation_iotaO*} to identify $\iota^*_{S_\O}=A_\O, \iota^*_{S^\k_\G}=\bar A^\k_\G$ leads to
\begin{multline*}
-\mathcal G^*(\nu) = 
\Bigg(\iint_{\QO}
A_\O\left(\frac{\rd \nu_\O}{\rd\mathcal L}\right)\rd \mathcal L
+
\iint_{\QO}
A_\O^\infty\left(\frac{\rd \nu_\O }{\rd \nu_\O^S}\right)\rd \nu_\O^S\Bigg)
\\
+
\Bigg(
\iint_{\QG}
\bar A^\k_\G\left(\frac{\rd \nu_\G}{\rd\mathcal L}\right)\rd \mathcal L
+
\iint_{\QG}
\bar A_\G^{\k\infty}\left(\frac{\rd \nu_\G }{\rd \nu_\G^S}\right)\rd \nu_\G^S
\Bigg)
\end{multline*}
Here $\mathcal L=\rd x\rd t$ denotes indistinctly the Lebesgue measure on $\QO$ or $\QG$, respectively, $\nu_\O^S,\nu_\G^S$ are any nonnegative measures dominating the singular parts of $|\nu_\O|,|\nu_\G|$, and $A_\O^\infty,\bar A_\G^{\k\infty}$ denote the recession functions of $A_\O, \bar A^\k_\G$.
Since $A_\O, \bar A^\k_\G$ are $1$-homogeneous their recession functions $A_\O^\infty=A_\O$ and $\bar A_\G^{\k\infty}=\bar A_\G^\k$, thus we can rewrite
$$
-\mathcal G^*(\nu) 
= \iint_{\QO}A_\O\left(\frac{\rd \nu_\O}{\rd\lambda_\O}\right)\rd \lambda_\O
+ \iint_{\QG}\bar A^\k_\G\left(\frac{\rd \nu_\G}{\rd\lambda_\G}\right)\rd \lambda_\G
$$
for any dominating measures $\lambda_\O\gg |\nu_\O|$ and $\lambda_\G\gg |\nu_\G|$.\\
So far we were writing $\nu=(\nu_\O,\nu_\G)=(\o,F\ ;\ \g,G,f,\eta)$, but let us now rather write
$$
\nu =(\mu_\O,\mu_\G,\eta):=(\o,F\ ;\ g,G,f\ ;\ \eta)
$$
in order to relate to the action functional \eqref{eq:def_action_functional}.
With this choice, and by definition \eqref{eq:def_A_extended} of the extended action $\bar A^\k_\G$, we can write
\begin{multline}
\label{eq:computation_G*}
-\mathcal G^*(\nu) 
= \iint_{\QO}A_\O\left(\frac{\rd \mu_\O}{\rd\lambda_\O}\right)\rd \lambda_\O
\ +\ \iint_{\QG}A^\k_\G\left(\frac{\rd \mu_\G}{\rd\lambda_\G}\right)\rd \lambda_\G
\ + \ \left\{
\begin{array}{ll}
0 & \mbox{if }f+\eta=0\\
+\infty & \mbox{otherwise}
\end{array}
\right.\\
=\mathcal A(\mu) 
\ +\  \left\{
\begin{array}{ll}
0 & \mbox{if }\eta=-f\\
+\infty & \mbox{otherwise}
\end{array}
\right.
\end{multline}
\end{itemize}
Gathering \eqref{eq:J_fenchel_rockafellar}\eqref{eq:computation_F*-L}\eqref{eq:computation_G*}, with now the correct flux condition $F\cdot n=f$ (since $F\cdot n =-\eta$ from $-\F^*(-L^*\nu)<+\infty$ and $\eta=-f$ from $-\mathcal G^*(\nu)<+\infty$), we end up with the claimed duality
$$
\sup\limits_{(\phi,\psi)\in E}\mathcal J^\k
=\inf \Big\{
\mathcal A(\mu)
\quad \mbox{ s.t. }\quad 
\mu\in\CE(\rho_0,\rho_1)
\Big\}=\Wrr^2(\rho_0,\rho_1)
$$
Finally, recall from Lemma~\ref{lem:Wrr_finite} that $\Wrr^2(\rho_0,\rho_1)=\inf\{\dots\}$ is always finite: the Fenchel-Rockafellar theorem further guarantees in that case the attainment $\inf=\min$ of the dual problem in $\sup \mathcal J^\k=\inf (\dots)$, and the proof is complete.
\end{proof}

As expected, we have
\begin{prop}
$\Wrr$ is a distance on $\Pp\left(\bO\right)$.
\end{prop}
\begin{proof}
The symmetry $\Wrr(\rho_0,\rho_1)=\Wrr(\rho_1,\rho_0)$ is obvious, since the action is even in the $F,G,f$ variables and therefore the problem is completely time-symmetric.
\\
For the indiscernibles, consider $\Wrr(\rho_0,\rho_1)=0$.
By Theorem~\ref{theo:duality_exist_geodesics} there exists a minimizer $\mu\in\CE(\rho_0,\rho_1)$.
Owing to \eqref{eq:dBL_leq_Wrr} we see that $d_{\mathrm{BL},\bO}(\o_0,\o_1)+d_{\mathrm{BL},\G}(\g_0,\g_1)\leq C_\k\sqrt{\A(\mu)}=0$, thus $\o_1=\o_0$ and $\g_1=\g_0$ as required.
The converse is immediate: if $\rho_0=\rho_1$ then $F=G=f=0$ gives an admissible $\mu$ with cost zero.\\
For the triangular inequality, choose any $\rho_0,\rho_1,\rho_2\in\Pp(\bO)$.
By the previous step we can assume that they are pairwise distinct.
By theorem~\ref{theo:duality_exist_geodesics} there exist a minimizer $\mu_{01}=(\o_{01},F_{01}\ ;\ \g_{01},G_{01},f_{01})$ from $\rho_0$ to $\rho_1$ and a minimizer $\mu_{12}=(\o_{12},F_{12}\ ;\ \g_{12},G_{12},f_{12})$ from $\rho_1$ to $\rho_2$, both in time $t\in[01]$.
For any fixed $\theta\in (0,1)$ one can easily rescale $ \mu_{01} \leadsto  \mu_{01}^\theta$ in time $t\in[0,\theta]$ and $\mu_{12} \leadsto  \mu_{12}^\theta$ in time $t\in[0,1-\theta]$.
Concatenating $\mu_{01}^\theta$ and $\mu_{12}^\theta$ in times $[0,\theta]\cup[\theta,1]$ gives an admissible competitor $\mu^\theta_{02}$ connecting $\rho_0,\rho_2$ in time $t\in[0,1]$.
The resulting cost is
$$
\Wrr^2(\rho_0,\rho_2)
\leq \A(\mu^\theta_{02})=
\frac{1}{\theta}\mathcal A(\mu_{01})
+\frac{1}{1-\theta}\mathcal A(\mu_{12})
=\frac{1}{\theta}\Wrr^2(\rho_0,\rho_1)
+\frac{1}{1-\theta}\Wrr^2(\rho_1,\rho_2),
$$
and choosing $\theta:=\frac{\Wrr(\rho_0,\rho_1)}{\Wrr(\rho_0,\rho_1)+\Wrr(\rho_1,\rho_2)}$ finally gives
$$
\Wrr^2(\rho_0,\rho_2) \leq \Big(\Wrr(\rho_0,\rho_1)+\Wrr(\rho_1,\rho_2)\Big)^2.
$$
\end{proof}
%
%%%%%%%%%%%%%%%%%%%%%%%%%%%%%%%%%
\subsection{Characterization and properties of geodesics}
Our next result gives a sufficient condition for $\mu\in\CE(\rho_0,\rho_1)$ to be a minimizer, and provides the geodesic equations at least formally.
\begin{theo}
  \label{theo:characterization_geodesics}
Fix $\rho_0,\rho_1\in\Pp(\bO)$, and assume that $\mu=(\o,F\ ;\ \g,G,f)\in\CE(\rho_0,\rho_1)$ is such that $\mathcal A(\mu)<\infty$ and, for some $(\phi,\psi)\in \C^1(\QO)\times\C^1(\QG)$,
\begin{equation}
\label{eq:FGf_geodesics_potential_phi_psi}
F=\o \nabla \phi,
\qquad 
G=\g \nabla\psi , 
\qquad
f =  \g \frac{\psi-\phi|_{\pO}}{\k^2}
\end{equation}
as well as
\begin{equation}
\label{eq:assumption_phi_subsol}
\partial_t\phi + \frac{1}{2}|\nabla\phi|^2 \leq 0 \mbox{ everywhere in }\QO, \mbox{ with equality }\o-\mbox{a.e.},
\end{equation}
\begin{equation}
 \label{eq:assumption_psi_subsol}
 \partial_t\psi + \frac{1}{2}|\nabla\psi|^2 +\frac 1{2\k^2} \left|\psi-\phi\right|^2\leq 0 \mbox{ everywhere in }\QG, \mbox{ with equality }\g-\mbox{a.e.}.
\end{equation}
Then $\mu$ minimizes $\Wrr^2(\rho_0,\rho_1) =\mathcal A(\mu)$.
\end{theo}
We expect these Hamilton-Jacobi conditions to be also necessary, thus fully characterizing all geodesics.
However the strong $\C^1$ regularity required above for $\phi,\psi$ should not be expected in all generality (see section~\ref{sec:explicit_geodesics_dirac_masses} and in particular the $\frac 1t$ loss of time regularity in \eqref{eq:interpolant_1d_o}), hence we shall be content with the ``sufficient'' part as in our statement.
Note that the condition $\mu\in\CE(\rho_0,\rho_1)$ implicitly prescribes the boundary condition $\o\nabla\phi\cdot n=F\cdot n=f=\g\frac{\psi-\phi}{\k^2}$ for $\phi$ on $\pO$.
\begin{proof}
The argument is adapted from \cite[Theorem 2.3]{chizat2018interpolating}.
With the same notations as in the proof of Theorem~\ref{theo:duality_exist_geodesics}, the Fenchel-Rockafellar duality (Theorem~\ref{theo:fenchel_rockafellar}) guarantees that a pair $(\phi,\psi)\in E$ is a maximizer of the primal problem $\sup\mathcal J^\k$ as soon as there exists some $\nu\in F^*$ such that
\begin{enumerate}[(i)]
  \item 
  $L^*\nu\in \partial(-\mathcal F) (\phi,\psi)$,
  \item
  $L(\phi,\psi)\in\partial (-\mathcal G^*)(\nu)$,
\end{enumerate}
in which case $\nu$ is necessarily a minimizer in \eqref{eq:J_fenchel_rockafellar}.
Recalling that we only use the extra variable $\eta$ to eliminate the invariance $\phi+k,\psi+k$, such a $\nu=(\mu,\eta)=(\o,F\ ;\ \g,G,f\ ;\ \eta)$ automatically gives a minimizer $\mu=(\o,F\ ;\ \g,G,f)\in\CE(\rho_0,\rho_1)$ in\eqref{eq:def_Wrr}.\\ 
Thus is suffices to check that (i) and (ii) hold with $\nu,\phi,\psi$ as in our statement, upon setting
$$
\nu:=(\nu_\O,\nu_\G)=(\o,F\ ;\ \g,G,f,-f) 
\hspace{1.5cm}(\mbox{i.e. taking }\eta:=-f).
$$
Condition (i) is automatically satisfied since $-\F$ is linear and $\mu$ solves the continuity equations: The subdifferential $\partial(-\F)(\phi,\psi)$ can be identified by computing, for arbitrary $(\phi',\psi')\in E$,
\begin{multline*}
-\F(\phi',\psi')+\F(\phi,\psi)
= 
\int_\bO[\phi'-\phi](1,.) \,\rd \o_1-\int_\bO[\phi'-\phi](0,.) \,\rd \o_0 
\\
+
\int_\G [\psi'-\psi](1,.) \,\rd \g_1-\int_\G[\psi'-\psi](0,.) \,\rd \g_0
\\
\overset{\eqref{eq:CE_omega}\eqref{eq:CE_gamma}}{=}
\iint_\QO \partial_t[\phi'-\phi]\,\rd \o
+ 
\iint_\QO \nabla[\phi'-\phi]\cdot \rd F
-
\iint_\QG [\phi'-\phi]\,\underbrace{\rd f}_{=-\rd\eta}
\\
+
\iint_\QG \partial_t[\psi'-\psi] \,\rd \g 
+ 
\iint_\QG \nabla [\psi'-\psi] \cdot \rd G
+
\iint_\QG [\psi'-\psi] \,\rd f
\\
=
\langle L(\phi'-\phi,\psi'-\psi),\nu\rangle_{F,F^*}
=
\langle (\phi'-\phi,\psi'-\psi),L^*\nu\rangle_{E,E^*}.
\end{multline*}
This means indeed that $L^*\nu\in\partial(-\F)(\phi,\psi)$.\\
For (ii) we first recall from the proof of Theorem~\ref{theo:duality_exist_geodesics} and \eqref{eq:computation_G*} that, writing $\nu=(\mu,\eta)=(\o,F,\g,G,f\ ;\ \eta)$, we already computed $-\mathcal G^*(\nu)=\A(\mu)+\iota_{[f=-\eta]}(\nu)$.
Exploiting the definition \eqref{eq:def_A_extended} of the extended action $\bar A^\k_\G$, the constraint $f=-\eta$ can be encoded into the boundary contribution to write
$$
-\mathcal G^*(\nu)
=
\iint_{\QO}A_\O\left(\frac{\rd \nu_\O}{\rd\lambda_\O}\right)\rd\lambda_\O
+
\iint_{\QG}\bar A^\k_\G \left(\frac{\rd \nu_\G}{\rd\lambda_\G}\right)\rd\lambda_\G.
$$
Here we denote again $\nu_\O=(\o,F)$ and $\nu_\G=(\g,G,f,\eta)$, and choose any reference measures $\lambda_\O,\lambda_\G$ such that $|\nu_\O|\ll\lambda_\O$ and $|\nu_\G|\ll\lambda_\G$.
We want to prove that $L(\phi,\psi)\in\partial (-\mathcal G^*)(\nu)$, which amounts to showing that
$$
-\mathcal G^*(\nu') + \mathcal G^*(\nu)\geq \langle L(\phi,\psi),\nu'-\nu\rangle _{F,F^*}
$$
for any $\nu'\in F^*$.
To this end we first write
\begin{multline*}
-\mathcal G^*(\nu') + \mathcal G^*(\nu)
=\iint_{\QO}\left(A_\O\left(\frac{\rd \nu'_\O}{\rd\lambda_\O}\right)-A_\O\left(\frac{\rd \nu_\O}{\rd\lambda_\O}\right)\right)\rd\lambda_\O\\
+ \iint_{\QG}\left(\bar A^\k_\G\left(\frac{\rd \nu'_\G}{\rd\lambda_\G}\right)-\bar A^\k_\G\left(\frac{\rd \nu_\G}{\rd\lambda_\G}\right)\right)\rd\lambda_\G
\end{multline*}
where the $\lambda$'s can be chosen to dominate simultaneously $|\nu_\O|+|\nu'_\O|\ll\lambda_\O$ and $|\nu_\G|+|\nu'_\G|\ll\lambda_\G$.
By Lebesgue decomposition such reference measures can always be taken of the form
$$
\lambda_\O=\o+\o^\perp,
\qquad
\lambda_\G=\g+\g^\perp,
$$
with $\o,\o^\perp$ and $\g,\g^\perp$ respectively mutually singular.
Therefore
\begin{multline}
\label{eq:computation_Gnu'-Gnu}
-\mathcal G^*(\nu') + \mathcal G^*(\nu)
=\Bigg[
\iint_{\QO}\left(A_\O\left(\frac{\rd \nu'_\O}{\rd\lambda_\O}\right)-A_\O\left(\frac{\rd \nu_\O}{\rd\lambda_\O}\right)\right)\rd\o\\
\hspace{2cm} + \iint_{\QO} \left (A_\O\left(\frac{\rd \nu'_\O}{\rd\lambda_\O}\right)-A_\O\left(\frac{\rd \nu_\O}{\rd\lambda_\O}\right)\right)\rd\o^\perp \Bigg]\\
+ \Bigg[
\iint_{\QG}\left(\bar A^\k_\G\left(\frac{\rd \nu'_\G}{\rd\lambda_\G}\right)-\bar A^\k_\G\left(\frac{\rd \nu_\G}{\rd\lambda_\G}\right)\right)\rd\g\\
\hspace{2cm} + \iint_{\QG} \left (\bar A^\k_\G\left(\frac{\rd \nu'_\G}{\rd\lambda_\G}\right)-\bar A^\k_\G\left(\frac{\rd \nu_\G}{\rd\lambda_\G}\right)\right)\rd\g^\perp\Bigg]\\
=
\iint_{\QO}\underbrace{\left(A_\O\left(\frac{\rd \nu'_\O}{\rd\lambda_\O}\right)-A_\O\left(\frac{\rd \nu_\O}{\rd\lambda_\O}\right)\right)}_{:=A(t,x)}\rd\o
+ \iint_{\QO} \underbrace{A_\O\left(\frac{\rd \nu'_\O}{\rd\lambda_\O}\right)}_{:=B(t,x)}\rd\o^\perp\\
+
\iint_{\QG}\underbrace{\left(\bar A^\k_\G\left(\frac{\rd \nu'_\G}{\rd\lambda_\G}\right)-\bar A^\k_\G\left(\frac{\rd \nu_\G}{\rd\lambda_\G}\right)\right)}_{:=C(t,x)}\rd\g
+ \iint_{\QG}\underbrace{ \bar A^\k_\G\left(\frac{\rd \nu'_\G}{\rd\lambda_\G}\right)}_{:=D(t,x)}\rd\g^\perp,
\end{multline}
because $A_\O\left(\frac{\rd \nu_\O}{\rd\lambda_\O}\right)=0$ for $\rd\o^\perp$ a.e. $(t,x)$ and $A^\k_\G\left(\frac{\rd \nu_\G}{\rd\lambda_\G}\right)=0$ for $\rd\g^\perp$ a.e. $(t,x)$ due to $\o\perp\o^\perp$ and $\g\perp\g^\perp$.
Moreover, straightforward (finite-dimensional) convex analysis shows that the subdifferentials
$$
\partial A_\O(\o,F)=
\left\{
\begin{array}{ll}
\left\{
\left(-\frac{|F|^2}{2\o^2},\frac F\o\right)
\right\} & \mbox{if }\o>0 
\\
S_\O & \mbox{if }(\o,F)=(0,0)
\\
\emptyset & \mbox{otherwise}
\end{array}
\right.
$$
and
\begin{multline*}
\partial \bar A^\k_\G(\g,G,f,\eta)\overset{\eqref{eq:def_A_extended}}{=}
\Bigg\{
(a,b,c,d):
\qquad
(a,b,c-d)\in\partial A^\k_\G(\g,G,f)
\Bigg\}
\\
=
\left\{
\begin{array}{ll}
\left\{
\left(-\frac{|G|^2+\k^2f^2}{2\g^2},\frac G\g\right)
\right\}
\times \left\{
(c,d)\mbox{ s.t. } c-d=\k^2\frac{f}{\g}
\right\}
& \mbox{if }\g>0 \mbox{ and }f+\eta=0
\\
S^\k_\G
& \mbox{if }(\g,G,f,\eta)=(0,0,0,0)
\\
\emptyset & \mbox{otherwise}
\end{array}
\right.
.
\end{multline*}
We recall that $S_\O,S^\k_\G$ were defined in \eqref{eq:def_D_Omega}\eqref{eq:def_D_Gamma}.
\begin{rmk}
  The quantity $c-d$ appearing in the computation of $\partial\bar A^\k_\G$ is of course dual (orthogonal) to $f+\eta$ appearing in \eqref{eq:def_A_extended}, and $\partial\bar A^\k_\G$ is accordingly invariant under diagonal shifts $c+k,d+k$.
\end{rmk}
\noindent
Our assumption \eqref{eq:assumption_phi_subsol} on $\phi$ precisely means
$$
  \begin{array}{cl}
  (\partial_t\phi,\nabla\phi)(t,x)\in \partial A_\O\left(\frac{\rd \nu_\O}{\rd\lambda_\O}(t,x)\right)
  &
  \mbox{for }\o-\mbox{a.e. }t,x
  \\
  (\partial_t\phi,\nabla\phi)(t,x)\in \partial A_\O(0)
  & 
  \mbox{for all }t,x,
  \end{array}
$$
hence
\begin{align}
  \label{eq:A(t,x)_geq}
  &A(t,x)\geq (\partial_t\phi,\nabla\phi)\cdot
  \left( \frac{\rd\nu'_\O}{\rd\lambda_\O} -\frac{\rd\nu_\O}{\rd\lambda_\O}\right)(t,x)
  &\mbox{for }\o-\mbox{a.e. }t,x
   \\
   \label{eq:B(t,x)_geq}
   &B(t,x)\geq 0+(\partial_t\phi,\nabla\phi)\cdot
  \left( \frac{\rd\nu'_\O}{\rd\lambda_\O} -\frac{\rd\nu_\O}{\rd\lambda_\O}\right)
  (t,x)
  &\mbox{for all }t,x.
\end{align}
Similarly, our assumption \eqref{eq:assumption_psi_subsol} on $\phi,\psi$ means
$$
  \begin{array}{cl}
    (\partial_t\psi,\nabla\psi,\psi,\phi)(t,x)\in \partial \bar A^\k_\G\left(\frac{\rd \nu_\G}{\rd\lambda_\G}(t,x)\right)
    &
    \mbox{for }\g-\mbox{a.e. }t,x
    \\
    (\partial_t\psi,\nabla\psi,\psi,\phi)(t,x) \in \partial \bar A^\k_\G(0)
    &
    \mbox{for all }t,x,
  \end{array}
$$
hence
\begin{align}
\label{eq:C(t,x)_geq}
& C(t,x)\geq (\partial_t\psi,\nabla\psi,\psi,\phi)\cdot
\left( \frac{\rd\nu'_\G}{\rd\lambda_\G} -\frac{\rd\nu_\G}{\rd\lambda_\G}\right)
(t,x)
&
\mbox{for }\g-\mbox{a.e. }t,x
\\
\label{eq:D(t,x)_geq}
& D(t,x)\geq 0+(\partial_t\psi,\nabla\psi,\psi,\phi)\cdot
\left( \frac{\rd\nu'_\G}{\rd\lambda_\G} -\frac{\rd\nu_\G}{\rd\lambda_\G}\right)
(t,x)
% \qquad
&\mbox{for all }t,x.
\end{align}
Injecting \eqref{eq:A(t,x)_geq}\eqref{eq:B(t,x)_geq}\eqref{eq:C(t,x)_geq}\eqref{eq:D(t,x)_geq} into \eqref{eq:computation_Gnu'-Gnu} gives
\begin{multline*}
-\mathcal G^*(\nu') + \mathcal G^*(\nu)
\geq
\iint_{\QO}(\partial_t\phi,\nabla\phi)\cdot
\left( \frac{\rd\nu'_\O}{\rd\lambda_\O} -\frac{\rd\nu_\O}{\rd\lambda_\O}\right)
\rd\o
+ \iint_{\QO} (\partial_t\phi,\nabla\phi)\cdot
\left( \frac{\rd\nu'_\O}{\rd\lambda_\O} -\frac{\rd\nu_\O}{\rd\lambda_\O}\right)
\rd\o^\perp\\
+
\iint_{\QG}(\partial_t\psi,\nabla\psi,\psi,\phi)\cdot
\left( \frac{\rd\nu'_\G}{\rd\lambda_\G} -\frac{\rd\nu_\G}{\rd\lambda_\G}\right)
\rd\g
+ \iint_{\QG}(\partial_t\psi,\nabla\psi,\psi,\phi)\cdot
\left( \frac{\rd\nu'_\G}{\rd\lambda_\G} -\frac{\rd\nu_\G}{\rd\lambda_\G}\right)
\rd\g^\perp\\
=
\iint_{\QO}(\partial_t\phi,\nabla\phi)\cdot \rd (\nu'_\O-\nu_\O)
+
\iint_{\QG}(\partial_t\psi,\nabla\psi,\psi,\phi)\cdot\rd(\nu'_\G-\nu_\G)\\
=\langle L(\phi,\psi),\nu'-\nu\rangle)_{F,F^*},
\end{multline*}
where the middle equality stems from our choice $\lambda_\O=\o+\o^\perp$ and $\lambda_\G=\g+\g^\perp$.
This finally entails $L(\phi,\psi)\in\partial (-\mathcal G^*)(\nu)$ and achieves the proof.
\end{proof}
We address next the natural question of constant-speed interpolations.
For any fixed $\rho_0,\rho_1\in\Pp(\bO)$ Theorem~\ref{theo:duality_exist_geodesics} always gives at least one minimizer $\mu=(\o,F\ ;\ \g,G,f)$ with action $\A(\mu)=\Wrr^2(\rho_0,\rho_1)<+\infty$.
By Proposition~\ref{prop:disintegration_momentum_velocity} the interior and boundary densities disintegrate in time $\o=\o_t\rd t$ and $\g=\g_t\rd t$, thus $\rho=\rho_t\rd t$ as well with $\rho_t:=(\o_t,\g_t)\in\Pp(\bO)$ for all $t\in[0,1]$.
Moreover $t\mapsto\rho_t$ is doubly narrowly continuous due to \eqref{eq:estimate_dBL}.
The measure $\rho_t\in \Pp(\bO)$ can thus be evaluated unambiguously at any time, and yields a natural interpolant $(\rho_t)_{t\in[0,1]}$ between the two endpoints.
As can be expected, this interpolant is consistent with the metric notion of constant-speed geodesics:
\begin{prop}
\label{prop:cst_speed}
 Take $\rho_0,\rho_1\in\Pp(\bO)$, let $(\o_t,F_t,\g_t,G_t,f_t)\rd t$ be any geodesic, and let $\rho_t=(\o_t,\g_t)\in\Pp(\bO)$ be the corresponding interpolant, $t\in[0,1]$.
 Then
 $$
 \Wrr(\rho_s,\rho_t)=|t-s|\Wrr(\rho_0,\rho_1),
 \qquad \forall\,s,t\in[0,1].
 $$
\end{prop}
\begin{proof}
The argument is fairly standard and we only sketch the proof.
Since solutions of the generalized continuity equation (Definition~\ref{def:CE}) can be concatenated in time, and because $\rho_t$ is doubly-narrowly continuous, it is easy to see that $\mu=(\o_\tau,F_\tau,\g_\tau,G_\tau,f_\tau)\rd\tau$ must be optimal in any subinterval $\tau\in[s,t]\subset [0,1]$ with fixed endpoints $\rho_s,\rho_t$, for $s\leq t$.
 A standard arc-length reparametrization (see e.g. \cite[Lemma 5.3]{kondratyev2017new} or the proof of \cite[Theorem 5.4]{dolbeault2009new}) then shows that the action
 $A(\mu_\tau)=\frac{|F_\tau|^2}{2\o_\tau}+\frac{|G_\tau|^2+\k^2 f^2_\tau}{2\g_\tau}=cst=\Wrr^2(\rho_0,\rho_1)$ is constant in time, and our statement immediately follows as
 $$
\Wrr^2(\rho_s,\rho_t)=|t-s|\int_s^t A(\mu_\tau)\rd\tau=|t-s|^2\Wrr^2(\rho_0,\rho_1).
 $$
\end{proof}

%%%%%%%%%%%%%%%%%%%%%%%%%%%%%%%%%%%%%%%%%%%%%%%%%%%%%%%%%%
\section{Explicit geodesics for point-masses}
\label{sec:explicit_geodesics_dirac_masses}
In this section we compute the distance and geodesics for two Dirac masses: One in the interior $\rho_0=(\delta_{x_0},0)$, the other on the boundary $\rho_1=(0,\delta_{x_R})$ at a distance $R>0$, and with supporting segment $I=[x_0,x_R)\subset \O$ lying in the interior as depicted in Figure~\ref{fig:I}.
We make no assumption on the contact angle between $I$ and $\pO$, in particular the segment $[x_0,x_R]$ may very well be tangent to $\pO$ at $x_R$.

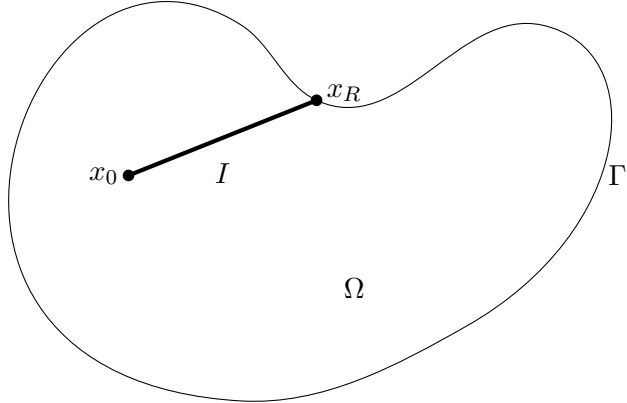
\begin{figure}[h!]
\begin{center}
\begin{tikzpicture}[use Hobby shortcut]
% \draw[help lines,step=1] (-4,0) grid (7,6);
%   \filldraw (0,0) circle (2pt);
  \path
  (0,0) coordinate (z0)
  (3,1) coordinate (z1)
  (4,5) coordinate (z2)
  (1,4) coordinate (z3)
  (0,5) coordinate (z4);
  \draw[closed] (z0) .. (z1) .. (z2) .. (z3) .. (z4);

  \filldraw (1,4) circle (2pt);
  \node [above, right] at (1,4.1) {$x_R$};
  
  \filldraw (-1.5,3) circle (2pt);
  \node [left] at (-1.5,3) {$x_0$};

  \draw [ultra thick](1,4)--(-1.5,3);
  \node [below] at (-.25,3.3) {$I$};
  \node  at (1.5,1.5) {$\O$};
  \node at (5,3) {$\G$};
\end{tikzpicture}
\end{center}
\caption{1D geodesic along $I=[x_0,x_R]$}
\label{fig:I}
\end{figure}
As tempting as it might be, the interpolating measure simply cannot be a traveling Dirac mass:
Being quadratic, the flux cost $|f|^2/\g$ indeed prevents an instantaneous jump of mass from $\g_t=0$ for $t<1$ to $\g_1=\delta_{x_R}$ at time $t=1$.
Therefore we need a more clever ansatz.
All the computations below will remain formal as a first step but will allow to compute explicitly the solution.
In order to make the analysis rigorous we will then use the verification Theorem~\ref{theo:characterization_geodesics} to check a posteriori that the interpolant computed formally is really a geodesic.

Two effects will be competing in the total action $\A_\bO+\A^\k_\G$:
On the one hand, since the flux penalization is exactly the Fisher-Rao Lagrangian $|f|^2/2\g$, and because no motion should be involved on the boundary, the sought $\Wrr$ geodesic has a strong incentive to conform as much as possible to a Fisher-Rao geodesic at least for the boundary mass $\g$.
The latter is known to be quadratic in time, $\g_t\approx t^2$ and $f_t=\partial_t\g_t\approx 2t$.
On the other hand for our coupled model such a growth is only possible if a nontrivial influx $f=F\cdot n$ arises from the interior.
In the absence of coupling the optimal motion in the interior would be given by Wasserstein displacement, and particles would tend to move with constant \emph{velocity} from $x_0$ to $x_R$.
The previous Fisher-Rao behavior $f_t\approx 2t$ rather corresponds to particle arriving at the boundary with constant \emph{acceleration}.
The two separate boundary/interior optimizers are thus incompatible with each other, and therefore a delicate transition occurs between constant speed and constant acceleration.
Let us now try to put this heuristic discussion on more solid ground.
\\

Since we are clearly in a one-dimensional framework we choose to parametrize the segment $I=[x_0,x_R]$ by arc-length $r\in[0,R]$, and we set the origin $r=0$ at $x_0\in\O$ with $r=R$ at $x_R\in\pO$.
We argue below as if the whole problem were set in the one dimensional segment $\bO= I=[0,R]$, and we will compute explicitly the geodesics in the variables $(t,r)$.

Since the interior density can only penetrate the boundary gradually in order to keep a finite flux cost, it seems clear that $\o_t$ must somehow scatter along the line $r\in[0,R]$ and that the mass initially concentrated at $r=0$ must split.
We thus choose to consider the interior density $\o_t$ as a continuous superposition of Lagrangian particles initially labeled by $y\in [0,1]$, all starting at $r=0$ with infinitesimal mass $\rd y$.
We denote by
\begin{center}
$X^y_t=$ position at time $t$ of a particle with label $y$, 
\end{center}
thus satisfying at time $t=0$
$$
X^y_0=0
\qquad\mbox{for all }y\in[0,1].
$$
Moreover since some mass of $\o_t$ must contribute to growing $\g_t$ via the outflux at $r=R$, some particles must eventually reach the boundary before $t=1$ and should accordingly be discarded afterwards (for if not, all the particles would reach the boundary simultaneously at $t=1$ and the flux cost would be infinite).
We denote by $[0,Y_t]$ the labels of particles that have not reached the boundary by time $t$, for some $Y_t\in [0,1]$  still to be determined and satisfying $Y_0=1$ and $Y_1=0$ (all the interior mass should vanish at $t=1$).
The remaining particles $y\in[Y_t,1]$ at time $t$ have already been absorbed by the boundary and transformed into $\g$ particles, and should not contribute to the interior density.
Whence our ansatz:
\begin{equation}
  \label{eq:ansatz_geodesics_superposition}
\o_t := \int_0^{Y_t} \delta_{X^y_t}\rd y
=X_t\pf\left(\operatorname{Leb}^1_{[0,Y_t]}\right)
\qqtext{and}
\g_t:=m_t\delta_R
\end{equation}
as depicted in Figure~\ref{fig:ansatz_1D}, where $\operatorname{Leb}^1_{[0,Y_t]}$ denotes the $1$-dimensional Lebesgue measure restricted to $[0,Y_t]\subset \R^1$.
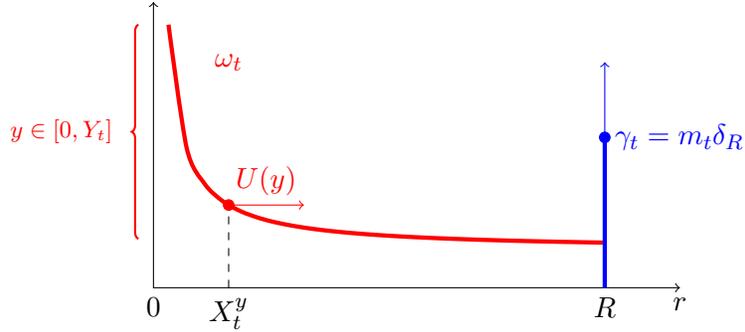
\begin{figure}[h!]
  \begin{center}
  \begin{tikzpicture}
%   \draw[help lines,step=1] (0,0) grid (7,4);
%   \filldraw (0,0) circle (2pt);
  \draw [<->] (0,3.8) -- (0,0) -- (7,0);
  \node [below] at (0,0) {$0$};
  \node [below] at (6,0) {$R$};
  \node [below] at (7,0) {$r$};
  
  \draw [domain=.2:6,smooth,variable=\x,red,ultra thick] plot({\x},{.5+.6/\x});
  \draw [dashed] (1,0)--(1,1.1);
  \filldraw [red] (1,1.1) circle (2pt);
  \node [below] at (1,0) {$X_t^y$};
  \draw [->,red] (1,1.1) -- (2,1.1);
  \node [red,above] at (1.5,1.1) {$U(y)$};
  \node [red] at (1,3) {$\o_t$};

  \draw [ultra thick, blue] (6,0) -- (6,2);
  \filldraw [blue] (6,2) circle (2pt);
  \node [right,blue] at (6,2) {$\g_t=m_t\delta_R$};
  \draw [->,blue] (6,2)--(6,3);
  
  \draw [decorate,thick,red,decoration={brace}] (-.2,.55+.6/6) -- (-.2,3.5) node [red,left,midway,xshift=-.2cm]{\footnotesize $y\in[0,Y_t]$};
%   \draw [decorate,red,decoration={brace}] (-.2,0) -- (-.2,.45+.6/6) node [red,left,midway,xshift=-.2cm]{\footnotesize $y\in[Y_t,1]$};
%   \draw [dashed,red] (0,.5+.6/6) -- (6,.5+.6/6);
%   \draw [decorate,decoration={brace,amplitude=10pt},xshift=-4pt,yshift=0pt]
% (0.5,0.5) -- (0.5,5.0) node [black,midway,xshift=-0.6cm] 
% {\footnotesize $P_1$};
  
  \end{tikzpicture}
  \end{center}
  \caption{The one-dimensional superposition}
  \label{fig:ansatz_1D}
  \end{figure}
Here $m_t=\TV{\g_t}$ is the total mass on the boundary and should satisfy the mass conservation
$$
\TV{\o_t}+\TV{\g_t}=\int_0^{Y_t}\TV{\delta_{X^y_t}}\rd y + \TV{\g_t}= Y_t + m_t =1.
$$
In order to find the geodesic we proceed by alternate optimization:
We first minimize the interior flow of the particles for a given boundary mass profile $t\mapsto m_t$, compute the optimal cost as a functional of $m$, and then we minimize the resulting cost w.r.t. all admissible $m$'s.
We will then reconstruct a posteriori the Eulerian fields  $\o_t(r),u_t(r),\g_t,f_t$.
\begin{enumerate}
 \item 
 {\bf (optimization for fixed $m$)}
 Intuitively it seems obvious that an optimal $m_t$ should increase in time from $m_0=0$ to $m_1=1$, and we thus assume that $\dot m_t>0$.
 The key point in this first step is that, just like in classical optimal transport, particles should have zero Lagrangian acceleration and move with constant velocity as prescribed by the Hamilton-Jacobi equation $\partial_t\phi+\frac 12 |\nabla\phi|^2=0$.
 In other words, the velocity $u_t(X_t^y)$ should not depend on time for a given particle $X^y_t$, and we thus set
 $$
 U(y):=\frac{d}{dt}X_t^y =\frac{X^y_t-0}{t-0}.
 $$
 (The constant speed of a particle starting from $r=0$ at time $0$ and located at $r=X^y_t$ at time $t$.)
 We will eventually determine the function $U(y)$ later on, but for now we write the characteristics as
 $$
 X^y_t= t U(y).
 $$
 For a given label $y$ the particle reaches the boundary $r=R$ in time exactly
\begin{equation}
\label{eq:lifespan_1D_ansatz}
 t=\tau(y):=\frac{R}{U(y)}
 \qquad\Leftrightarrow\qquad 
 X^y_{\tau(y)}=R,
 \end{equation}
 and the infinitesimal kinetic energy carried by this particle during its whole lifespan (before reaching the boundary) is
 $$
 \rd e_\O(y)=\int_0^{\tau(y)} \left\{\frac{1}{2}\rd y|U(y)|^2\right\}\,\rd t=\frac{1}{2} R U(y)\rd y.
 $$
 The overall kinetic energy is simply
 \begin{equation}
   \label{eq:total_kinetic_m_fixed}
  E_\O[m]:=\int_0^1\rd e_\O(y)=\frac R2 \int_0^1 U(y)\rd y.
 \end{equation}
 In order to make this more explicit as a functional of the given profile $t\mapsto m_t$, note that a particle $X^y_t$ has not reached the boundary by time $t$ if and only if $t\leq\tau(y)$, i-e $X^y_t=tU(y)\leq R$. 
 Another crucial feature of classical optimal transport is that particles should not cross:
 It is therefore natural to assume that $U$ is a nondecreasing function of $y$, the position $X_t^y=tU(y)$ is nondecreasing in $y\in[0,1]$, and the set of particles $y\in[0,Y_t]$ of particles still within the domain $[0,R]$ at time $t$ must therefore be given by
$$
  Y_t=\max\{y:\, X^y_t\leq R\}
  \qquad \Leftrightarrow \qquad
  tU(Y_t)=R.
$$
By mass conservation $Y_t+m_t=\TV{\o_t}+\TV{\g_t}=1$ this also reads
\begin{equation}
  \label{eq:U_1-mt}
  U(1-m_t)=U(Y_t)=\frac Rt.
\end{equation}
Recalling that $\dot m_t>0$, we can change variables $y=1-m_t$ with $\rd y=-\dot m_t\rd t$ in \eqref{eq:total_kinetic_m_fixed} and compute the interior kinetic cost
$$
  E_\O[m]
  =\frac R2 \int_0^1 U(y)\rd y
  =
  \frac{R}{2} \int_0^1U(1-m_t)\dot m_t\rd t
  \overset{\eqref{eq:U_1-mt}}{=}
  \frac{R^2}{2} \int_0^1\frac{\dot m_t}t \rd t.
$$
 \item
 {\bf (minimization with respect to $m$)}
 Now we want to minimize the total cost ``interior kinetic + boundary flux'' over all admissible mass profiles $t\mapsto m_t$.
 The kinetic cost for fixed $m$ has just been computed in the previous step, and the flux cost is explicit in terms of $m_t$ since $\partial_t \g_t+0=f_t$ simply means $f_t^2/\g_t=(\dot m_t)^2/m_t$.
 We are thus trying to solve
 $$
 \min\limits_m \left\{
 \frac{R^2}{2}\int_0^1\frac{\dot m_t}t \rd t
 +
 \frac{\k^2}{2}\int_0^1 \frac{|\dot m_t|^2}{m_t}\rd t
 \qqtext{s.t.}
 m_0=0,m_1=1
 \right\}.
 $$
 For this Lagrangian
 $$
 L(m,\dot m,t):= \frac{\k^2 |\dot m|^2}{2 m} + \frac{R^2\dot m}{2 t}
 $$
 the usual Euler-Lagrange equation reads here
 $$
 \frac{d}{dt}\left(\frac{\partial L}{\partial \dot m}\right) =\frac{\partial L}{\partial m}
 \qquad
 \Leftrightarrow
 \qquad
 \frac{d}{dt}\left(\k^2\frac{\dot m}{m} +\frac{R^2}{2t}\right) = -\frac{\k^2 }{2}\left|\frac{\dot m}{m}\right|^2.
 $$
 This only depends on the logarithmic derivative $\dot m/m$ and it is therefore natural to look for power law solutions
 $$
 m_t=t^\alpha,
 \hspace{1.5cm}(\alpha>0).
 $$
 In this simple setting the previous Euler-Lagrange equation is satisfied if and only if
 \begin{equation*}
   \label{eq:def_alpha_optimal_polynom}
 \k^2\alpha +\frac{R^2}{2}=\frac{\k^2}{2}\alpha^2
 \qquad \Leftrightarrow\qquad
 \alpha^2 -2\alpha -\frac{R^2}{\k^2}=0.
\end{equation*}
 This quadratic polynomial in $\alpha$ has two real roots:
 The first is always negative and should be discarded, and the second reads explicitly
 \begin{equation}
   \label{eq:def_alpha_optimal}
 \alpha=1 +\sqrt{1+\frac{R^2}{\k^2}}.
\end{equation}
 \item
 {\bf (reconstruction of the Eulerian fields)}
The power law $m_t=t^\alpha$ can obviously be inverted as $t=(m_t)^{1/\alpha}$.
From \eqref{eq:U_1-mt} we have therefore,
$$
U(1-m_t)
=\frac{R}{t}
=\frac{R}{(m_t)^{1/\alpha}}
=\frac{R}{[1-(1-m_t)]^{1/\alpha}}.
$$
Since $\dot m>0$ we can use $y=1-m_t$ as an independent variable, whence
$$
U(y)=\frac{R}{(1-y)^{1/\alpha}}
\qquad 
\forall\,y\in[0,1].
$$
The lifespan of an arbitrary particle can then be computed from \eqref{eq:lifespan_1D_ansatz} as $\tau(y)=\frac{R}{U(y)}=(1-y)^{1/\alpha}\in[0,1]$, and the characteristics $X^y_t=U(y)t$ are therefore
\begin{equation}
  \label{eq:characteristic_Xyt-Lagrangian}
X^y_t=\frac{R}{(1-y)^{1/\alpha}}t
\qqtext{for}
t\leq \tau(y)=(1-y)^{1/\alpha}.
\end{equation}
The upper bound $Y_t$ in our ansatz \eqref{eq:ansatz_geodesics_superposition} can be computed by solving explicitly $\tau(Y_t)=t$, leading to $Y_t=1-t^\alpha$.
By definition \eqref{eq:ansatz_geodesics_superposition} of $\o_t$ we have, for any $\phi\in \C([0,R])$
\begin{equation*}
\int_0^R\phi(r)\rd \o_t(r)
=
\int_0^{Y_t} \phi(X^y_t)\rd y
=
\int_0^{1-t^\alpha} \phi\left(\frac{R}{(1-y)^{1/\alpha}}t\right)\rd y.
\end{equation*}
Changing variables $ r =\frac{R}{(1-y)^{1/\alpha}}t\Leftrightarrow y=1-(Rt/r)^\alpha$ with $\rd y=\alpha(Rt/r)^\alpha\frac{\rd r}{r}$ in this last integral gives
$$
\int_0^R\phi(r)\rd \o_t(r)
=
\int_{Rt}^R\phi(r)\alpha\left(\frac{Rt}{r}\right)^\alpha\frac{\rd r}{r}
$$
and therefore identifies
$$
\o_t=\alpha\left(\frac{Rt}{r}\right)^\alpha\frac{1}{r}\chi_{[Rt,R]}(r)\rd r.
$$
By definition the velocity field $u_t(r)$ in Eulerian coordinates is the velocity of the Lagrangian particle $X^y_t$ sitting at position $r$ at time $t$.
Since particles do not cross and $U(y)$ is nondecreasing there is a unique label $y_t$ such that $X^{y_t}_t=r$ for given $t,r$, and
$$
u_t(r): = \left[\frac{d}{dt}X^y_t\right]_{y=y_t}
\qqtext{with}y_t
\mbox{ s.t. }X^{y_t}_t =r.
$$
The explicit expression \eqref{eq:characteristic_Xyt-Lagrangian} of the characteristics gives $y_t=1-(Rt/r)^\alpha$, whence
$$
u_t(r)=\left.\frac{R}{(1-y)^{1/\alpha}}\right|_{y=y_t}=\frac{r}{t}.
$$
\end{enumerate}
We are now in position to exploit these formal computations rigorously:
\begin{theo}
\label{theo:explicit_1D_geodesics}
For $R>0$ let $x_0\in\O$ and $x_R\in\pO$ be two points at distance $R$ such that the segment $[x_0,x_R)$ lies in $\O$, and let $\alpha_\k:=1+\sqrt{1+\frac{R^2}{\k^2}}$.
Then
\begin{equation}
  \label{eq:explicit_Wrr_1D_diracs}
  \Wrr^2\Big((\delta_{x_0},0),(0,\delta_{x_R})\Big)=\frac{1}{2}(R^2 +\k^2\alpha_\k)\frac{\alpha_\k}{\alpha_\k-1}
\end{equation}
and a geodesic is given by $\mu^\k=\mu^\k_t\rd t=(\o^\k_t,F^\k_t,\g^\k_t,0,f^\k_t)\rd t$ with
\begin{equation}
  \label{eq:interpolant_1d_o}
  \o^\k_t:=\alpha_\k\left(\frac{Rt}{r}\right)^{\alpha_\k}\frac{1}{r}\chi_{[Rt,R]}(r)\rd r,
  \qquad 
  u^\k_t(r):=\frac{r}{t},
  \qquad
  F^\k_t:=u^\k_t \o^\k_t
\end{equation}
\begin{equation}
  \label{eq:interpolant_1d_g}
  \g^\k_t:=t^{\alpha_\k}\delta_R,
  \qquad
  G^\k_t:=0,
  \qquad
  f^\k_t:=\alpha_\k t^{\alpha_\k-1}\delta_R.
\end{equation}
\end{theo}
\noindent
Before proceeding with the proof let us point out several interesting facts here:
\begin{enumerate}
  \item
  The explicit cost \eqref{eq:explicit_Wrr_1D_diracs} is of the form ``transport + toll'', $\mathcal O(R^2)+\mathcal O(\k^2)$.
  This illustrates the idea that our model is essentially classical optimal transport in the interior combined with a non-reducible toll.
  \item 
  For fixed $\k$ we see that taking $R\to 0$ gives $\alpha_\k\to 2$, in which case we recover the quadratic Fisher-Rao ansatz $\g_t=t^2$.
  (We would then be transferring an $\o$-point mass to a $\g$-point mass, both located at the same site $x_R\in\pO$ and of course no mass displacement is involved in that task).
  \item
  Letting $\k\to 0$ for fixed $R>0$ gives $\alpha_\k\sim\frac R{\k}\to+\infty$, $\k^2\alpha_\k\sim R\k\to 0$, and
  \begin{equation*}
  \label{eq:Wrr_diracs_to_w2_k_to_0}
  \Wrr^2((\delta_{x_0},0),(0,\delta_{x_R}))\xrightarrow[\k\to 0]{} \frac{1}{2}R^2=\W^2_\bO(\delta_{x_0},\delta_{x_R}).
  \end{equation*}
  Moreover, leveraging its fully explicit expression, it is easy to check that the interpolant \eqref{eq:interpolant_1d_o}\eqref{eq:interpolant_1d_g} converges narrowly to the $\W_\bO$-Wasserstein geodesic between $\varrho_0=\delta_{x_0},\varrho_1=\delta_{x_R}$ in the sense that $\varrho^\k:=\o^\k+\g^\k\narrowcv \varrho:=\delta_{x_t}\rd t$ and $F^\k+G^\k\narrowcv H:=\dot x_t\delta_{x_t}\rd t$ with $x_t:=(1-t)x_0+tx_R$.
  \item
  As $\k\to+\infty$ we have $\alpha_\k\to 2$, hence from \eqref{eq:explicit_Wrr_1D_diracs} $\Wrr^2((\delta_{x_0},0),(0,\delta_{x_R}))\sim 2\k^2\to+\infty$.
  This should be expected:
  In that case we are trying to connect measures having very different masses in the interior and on the boundary, therefore the necessary flux is heavily penalized by
  the expensive toll $\k\gg 1$.
\end{enumerate}
All of this will be generalized later in Section~\ref{sec:varying toll} when we consider the large and small toll limits $\kappa\to+\infty,\k\to 0$ for arbitrary measures.
\begin{proof}
In order to alleviate the notations we drop the $\k$ subscripts in the whole proof, and write $\rho_0=(\delta_{x_0},0)$ and $\rho_1=(0,\delta_{x_R})$.
Let us first focus on the purely one-dimensional case $d=1$.
It is not difficult to check that the interpolants \eqref{eq:interpolant_1d_o}\eqref{eq:interpolant_1d_g} solve the continuity equations 
$$
\left\{
\begin{array}{ll}
\partial_t\o_t +\dive_r(\o_t u_t) =0 & \mbox{for }(t,r)\in(0,1)\times (0,R)\\
\o_t u_t|_{r=R}= f_t & \mbox{for }t\in(0,1)
\end{array}
\right.
\qqtext{and}
\partial_t\g_t +0 =f_t
$$
in the sense of Definition~\ref{def:CE}.
In order to check that $\mu_t=(\o_t,F_t,\g_t,G_t,f_t)$ in statement is really a geodesic we can appeal to Theorem~\ref{theo:characterization_geodesics} and try to find two functions $\phi,\psi$ such that $F_t=\o_t u_t=\o_t\nabla\phi$, $G_t=0=\g_t\nabla\psi$, $f_t=\g_t\frac{\psi-\phi}{\k^2}$, and solving the two Hamilton-Jacobi equations \eqref{eq:assumption_phi_subsol}\eqref{eq:assumption_phi_subsol}.
With the explicit expressions \eqref{eq:interpolant_1d_o}\eqref{eq:interpolant_1d_g} now at hand this becomes an easy task: Writing as before $m_t=\TV{\g_t}=t^\alpha$ and letting
$$
\phi(t,r):=\frac{r^2}{2t},
\qquad
\psi(t):=\phi(t,R)+\k^2\frac{\dot m_t}{m_t}=\frac{R^2}{2t}+\frac{\k^2\alpha}{t},
$$
we have automatically
$$
\partial_r\phi(t,r)=\frac{r}{t}=u_t(r)
\qqtext{and}
\g_t\frac{\psi(t)-\phi(t,R)}{\k^2}=\frac{\dot m_t}{m_t}=\partial_t\g_t.
$$
The first Hamilton-Jacobi equation is satisfied as
$$
\partial_t\phi +\frac{1}{2}|\nabla\phi|^2
=
-\frac{r^2}{2t^2}+\frac{1}{2}\left|\frac{r}{t}\right|^2=0.
$$
For the equation in $\psi$ we have
\begin{multline*}
\partial_t\psi +\frac{1}{2}|\nabla\psi|^2+\frac{1}{2\k^2}|\psi-\phi|^2
=
\partial_t\left(\frac{R^2}{2t}+\frac{\k^2\alpha}{t}\right)+ 0 + \frac{1}{2\k^2}\left|\frac{\k^2\alpha}{t}\right|^2
\\
=-\left(\frac{R^2}{2}+\k^2\alpha\right)\frac{1}{t^2} +0+ \frac{\alpha^2\k^2}{2}\frac{1}{t^2}
= \frac{1}{t^2}\left(\frac{\k^2}{2}\alpha^2-\k^2\alpha -\frac{R^2}{2}\right)
=0
\end{multline*}
because the optimal value \eqref{eq:def_alpha_optimal} of $\alpha=\alpha_\k$ was derived precisely by canceling this last polynomial.

Unfortunately this application of Theorem~\ref{theo:characterization_geodesics} is not fully rigorous because of the singular $\frac 1t$ factor, corresponding to the unavoidable mass splitting.
However, $(\phi,\psi)$ have the required regularity in any subinterval $t\in [\eps,1]$, hence our interpolant is really a geodesic from $\rho_\eps$ to $\rho_1$ for all $\eps>0$.
Rescaling time thus gives
$$
\Wrr^2(\rho_\eps,\rho_1)=(1-\eps) \int_\eps^1 A(\mu_t)\rd t.
$$
It is not hard to check from the previous explicit expressions that $A(\mu_t)$ is of course integrable in time, hence the latter quantity converges as
$$
\Wrr^2(\rho_\eps,\rho_1)\xrightarrow[\eps\to 0]{}\int_0^1 A(\mu_t)\rd t.
$$
On the other hand by construction $\rho_\eps$ is obtained by following for small times an admissible path $\mu=(\o,F,\g,G,f)$ with finite cost starting from $\rho_0$, hence scaling again in time
$$
\Wrr^2(\rho_0,\rho_\eps)
\leq
\eps \int_0^1 A(\mu_t)\rd t
\xrightarrow[\eps\to 0]{}0
$$
and by triangular inequality $|\Wrr(\rho_0,\rho_1)-\Wrr(\rho_\eps,\rho_1)|\leq \Wrr(\rho_0,\rho_\eps) \to 0$.
This implies
$$
\Wrr^2(\rho_0,\rho_1)
=
\lim\limits_{\eps\to 0} \Wrr^2(\rho_\eps,\rho_1)
=
\lim\limits_{\eps\to 0} (1-\eps)\int_\eps^1 A(\mu_t)\rd t
=
\int_0^1 A(\mu_t)\rd t
$$
and shows that the interpolant \eqref{eq:interpolant_1d_o}\eqref{eq:interpolant_1d_g} is indeed a geodesic as expected.

In order to evaluate the latter integral we recall from the formal computations in the beginning of the section that, by construction, the kinetic cost in the interior is exactly $\frac{R^2}{2}\int_0^1\frac{\dot m_t}t \rd t$.
(This can also be checked rigorously by direct evaluation of $\frac 12\int_0^1\int_0^R|u_t|^2\rd\o_t\rd t$.)
 Putting everything together with $m_t=t^\alpha$, the final cost is
\begin{multline*}
\Wrr^2(\rho_0,\rho_1)
=
\frac{R^2}{2}\int_0^1\frac{\dot m_t}t \rd t
 +
\frac{\k^2}{2}  \int_0^1 \frac{|\dot m_t|^2}{m_t}\rd t
 \\
= \frac{R^2}{2}\int_0^1 \frac{\alpha t^{\alpha-1}}{t}\rd t
+ \frac{\k^2}{2} \int_0^1 \frac{|\alpha t^{\alpha-1}|^2}{t^\alpha}\rd t
\\
=\frac{1}{2}(\alpha R^2 +\alpha^2\k^2)\int_0^1 t^{\alpha-2}\rd t
% \\
=\frac{1}{2}(R^2 +\k^2\alpha)\frac{\alpha}{\alpha-1}
\end{multline*}
and the argument is complete.

Now if $d>1$ we simply adapt the above one-dimensional scenario as follows:
Using the uniform one-dimensional Hausdorff measure $\mathcal H^1_I$ supported on $I$ it is a simple exercise to turn the above one-dimensional interpolant $(\o,F,\g,G,f)$ into a full $d$-dimensional solution of the continuity equations as in Definition~\ref{def:CE}.
Without loss of generality we can assume that $x_0=0_{\R^d}$, and the previous Hamilton-Jacobi solutions can be extended to the whole domain by setting $\bar\phi(t,x):=\phi(t,|x|)$ and $\bar\psi(t,x):=\psi(t)+\bar\phi(t,x)-\bar\phi(t,x_R)$.
A straightforward computation then shows that $\bar\phi,\bar\psi$ still satisfy \eqref{eq:assumption_phi_subsol}\eqref{eq:assumption_psi_subsol} and are therefore optimal locally in time away from $t=0^+$.
The rest of the argument applies verbatim:
Optimality in any time interval $[\eps,1]$ gives optimality in the whole $[0,1]$, and the proof is finally complete.
\end{proof}

%
%%%%%%%%%%%%%%%%%%%%%%%%%%%%%%%%%%%%%%%%%%%%%%%%%%%%%%%%%%%%%%%%%%%%%%%%%%%%%%%%%%%%%%%%
\section{Geometrical and topological properties}
\label{sec:geometric_topological_properties}
\subsection{Comparison with other distances}
Here we compare our ring road distance $\Wrr(\rho_0,\rho_1)$ with the distances naturally involved in the construction, namely $\W_\bO(\varrho_0,\varrho_1)$, $\W_{\bO}(\o,\o_1)$, $\W_{\G}(\g_0,\g_1)$, $\WFR(\g_0,\g_1)$.
We will also need to compare it to bounded-Lipschitz distances for technical reasons.
\begin{prop}
\label{prop:comparison_other_distances}
For any $\rho_0,\rho_1\in\Pp(\bO)$ there holds
\begin{enumerate}[(i)]
  \item
  \label{item:WFR_leq_Wrr}
  \begin{equation}
  \label{eq:WFR_leq_Wrr}
\WFR^2(\g_0,\g_1)\leq   \Wrr^2(\rho_0,\rho_1)
\end{equation}
  \item
  \label{item:prop_comparison_|o0|=|01|_|g0|=|g1|}
  \begin{equation}
  \label{eq:Wrr_leq_W+W}
  \Wrr^2(\rho_0,\rho_1)\leq \W^2_{\bO}(\o_0,\o_1) + \W^2_\G(\g_0,\g_1),
  \end{equation}
%   with
  \item
  \label{item:prop_comparison_W2_leq_Wrr}
  Writing $\varrho_0=\o_0+\g_0,\varrho_1=\o_1+\g_1$,
  \begin{equation}
    \label{eq:Wrr_geq_W}
   \W^2_{\bO}(\varrho_0,\varrho_1)\leq \Wrr^2(\rho_0,\rho_1)
 \end{equation}
  \item
  \begin{equation}
  \label{eq:dBL_leq_Wrr}
  d_{\mathrm{BL},\bO}(\o_0,\o_1) + d_{\mathrm{BL},\G}(\g_0,\g_1)
\leq C_\k\Wrr(\rho_0,\rho_1)
  \end{equation}
  with $C_\k=4 \max(1,1/\k)$.
\end{enumerate}

\end{prop}
\noindent
Let us emphasize that the bounds (\ref{item:WFR_leq_Wrr})(\ref{item:prop_comparison_|o0|=|01|_|g0|=|g1|})(\ref{item:prop_comparison_W2_leq_Wrr}) are optimal, as we shall see later on.
In \eqref{eq:Wrr_leq_W+W} one should implicitly read $\W^2_{\bO}(\o_0,\o_1)= \W^2_\G(\g_0,\g_1)=+\infty$ for incompatible masses $\TV{\o_0}=1-\TV{\g_0}\neq 1-\TV{\g_1}=\TV{\o_1}$, in which case the statement is vacuous.
At first sight (\ref{item:prop_comparison_|o0|=|01|_|g0|=|g1|}) and (\ref{item:prop_comparison_W2_leq_Wrr}) may seem contradictory.
This is fortunately not the case since, even with mass compatibility $\TV{\o_0}=\TV{\o_1}$, we have generically $\W^2_{\bO}(\o_0+\g_0,\o_1+\g_1)< \W^2_{\bO}(\o_0,\o_1)+\W_\G^2(\g_0,\g_1)$.
(For example take $\bO$ a ball with $\rho_0=\left(\frac{1}{2}\delta_O,\frac{1}{2}\delta_N\right)$ and $\rho_1=\left(\frac{1}{2}\delta_O,\frac{1}{2}\delta_S\right)$, where the points $N,S,O$ are set at the North/South poles and at the origin.)

\begin{proof}
\begin{enumerate}[(i)]
  \item 
  Pick from Theorem~\ref{theo:duality_exist_geodesics} a $\Wrr$ geodesic $\mu=(\mu_\O,\mu_\G)=(\o,F\ ;\ \g,G,f)$ from $\rho_0$ to $\rho_1$.
  Note of course that $\mu_\G=(\g,G,f)$ solves the continuity equation $\partial_t\g +\dive G =f$ and connects $\g_0,\g_1$.
  By definition \eqref{eq:def_WFR} of the Wasserstein-Fisher-Rao metrics we get
  \begin{equation*}
  \WFR^2(\g_0,\g_1)=\min\limits_{\mu'_\G} \mathcal A^\k_\G(\mu'_\G)
  \leq \A^\k_\G(\mu_\G) \leq \mathcal A(\mu)=\Wrr^2(\rho_0,\rho_1).
  \end{equation*}
\item
As already mentioned if the masses are incompatible our statement is vacuous, hence we assume $\TV{\o_0}=\TV{\o_1}$ and $\TV{\g_0}=\TV{\g_1}$.
Pick an interior Wasserstein geodesic $(\o,F)$ from $\o_0$ to $\o_1$ in $\bO$ (with zero flux), and independently a boundary Wasserstein geodesic $(\g,G)$ from $\g_0$ to $\g_1$ in $\G$.
Since $f=F\cdot n=0$ the two conservative continuity equations together immediately yield a solution of the generalized continuity equation in the sense of Definition~\ref{def:CE}.
As a consequence $\mu:=(\o,F\ ;\ \g,G,0)=(\mu_\O,\mu_\G)$ is an admissible competitor in \eqref{eq:def_Wrr} and
$$
\Wrr^2(\rho_0,\rho_1)\leq \mathcal A(\mu)
=
\mathcal A_\O(\mu_\O)+\mathcal A^\k_\G(\mu_\G)
=
\W^2_\bO(\o_0,\o_1) + \W^2_\G(\g_0,\g_1).
$$
\item
We use the dual characterization in Theorem~\ref{theo:duality_exist_geodesics}, which for convenience we write here as
\begin{multline*}
\Wrr^2(\rho_0,\rho_1)
\\
=\sup_{(\phi,\psi)\in E}
\Bigg\{
\left(\int_\bO \phi (1,.)\rd\o_1 +\int_\G \psi (1,.)\rd\g_1\right)
- \left(\int_\bO \phi (0,.)\rd\o_0 +\int_\G \psi (0,.)\rd\g_0\right)\\
\mbox{s.t. }
\begin{array}{l}
\partial_t\phi +\frac 12 |\nabla\phi|^2\leq 0
\\
\partial_t\psi +\frac 1{2} |\nabla\psi|^2+\frac{1}{2\k^2}|\psi-\phi|^2\leq 0
\end{array}
\Bigg\}
\end{multline*}
with $E=\C^1([0,1]\times\bO)\times\C^1([0,1]\times \G)$.
Putting $\psi=\phi|_{\pO}$, the gradient on the boundary $\nabla\psi=\nabla_\G\psi$ is simply the tangential gradient $\nabla_\tau\left( \phi|_\pO\right)$ and the second Hamilton-Jacobi inequality
$$
\partial_t\psi +\frac 12 |\nabla\psi|^2+\frac{1}{2\k^2}|\psi-\phi|^2
=\partial_t\phi +\frac 12 |\nabla_\tau\phi|^2+0
\leq \partial_t\phi +\frac 12 |\nabla\phi|^2\leq 0
$$
is automatically satisfied as soon as $\phi$ is a subsolution.
As a consequence the supremum over all $(\phi,\psi)\in E$ is clearly larger than the supremum over the smaller set $\{\psi=\phi|_{\pO}\}\subsetneq E$, thus
\begin{multline*}
\Wrr^2(\rho_0,\rho_1)
\geq
\sup_{\phi\in \C^1(\QO)}
\Bigg\{
\left(\int_\bO \phi (1,.)\rd\o_1 +\int_\G \phi|_{\pO} (1,.)\rd\g_1\right)\\
- \left(\int_\bO \phi (0,.)\rd\o_0 +\int_\G \phi|_{\pO} (0,.)\rd\g_0\right)
\qquad\mbox{s.t. }
\partial_t\phi +\frac 12 |\nabla\phi|^2\leq 0
\Bigg\}\\
=
\sup_{\phi\in \C^1(\QO)}
\Bigg\{
\int_{\bO} \phi (1,.)\rd\varrho_1
-\int_{\bO} \phi (0,.)\rd\varrho_0
\qquad\mbox{s.t. }
\partial_t\phi +\frac 12 |\nabla\phi|^2\leq 0
\Bigg\}
 \\
=\W^2_{\bO}(\varrho_0,\varrho_1).
\end{multline*}
(The last equality is the well-known Kantorovich duality \cite{villani_BIG,OTAM} for the standard Wasserstein distance on $\bO$ between $\varrho_0=\o_0+\g_0$ and $\varrho_1=\o_1+\g_1$.)
\item
Pick a geodesic $\mu\in\CE(\rho_0,\rho_1) $ from Theorem~\ref{theo:duality_exist_geodesics}.
By Proposition~\ref{prop:disintegration_momentum_velocity} and \eqref{eq:estimate_dBL} we see that
$$
d_{\mathrm{BL},\bO}(\o_0,\o_1) + d_{\mathrm{BL},\G}(\g_0,\g_1)
\leq 4\max(1,1/\k)\sqrt{\A(\mu)}=C_\k \Wrr(\rho_0,\rho_1)
$$
and the proof is complete.
\end{enumerate}  
\end{proof}
We now turn to the more specific case of measures with either the interior or boundary densities being fixed equal for both endpoints.
In that case one natural question to ask is whether our distance can be expressed in terms of distances involving only the complementary densities.
\begin{prop}[fixed interior/boundary densities]
\label{prop:comparison_other_distances_fixed}
~
\begin{enumerate}[(i)]
\item
\label{item:Wrr_neq_W2_o_fixed}
In general $\o_0=\o_1$ does \emph{not} imply $\Wrr^2(\rho_0,\rho_1)=\W_\G^2(\g_0,\g_1)$.
\item
\label{item:Wrr_neq_WFR_g_fixed}
In general $\o_0=\o_1$ does \emph{not} imply $\Wrr^2(\rho_0,\rho_1)=\WFR^2(\g_0,\g_1)$.
\item 
\label{item:Wrr=W2_go_boundary_fixed}
If $\O$ is convex then
\begin{equation}
\label{eq:Wrr=W2_g_fixed}
\left.
\begin{array}{c}
\g_0=\g_1
\\
\o_0\mrest\pO=\o_1\mrest\pO
\end{array}
\right\}
\qquad\Rightarrow\qquad
\Wrr^2(\rho_0,\rho_1)=\W^2_\bO(\o_0,\o_1).
\end{equation}
If on the contrary $\O$ is not convex, it may happen that $\Wrr^2(\rho_0,\rho_1)<\W^2_\bO(\o_0,\o_1)$.
\end{enumerate}
\end{prop}
\begin{proof}
We stress that for (\ref{item:Wrr=W2_go_boundary_fixed}) it is really necessary that both boundary masses $\o\mrest\pO,\g$ match at the endpoints $t=0,1$, otherwise the statement fails even if the total boundary masses match $\varrho_0\mrest\pO=\varrho_1\mrest\pO$ as one may have hoped for.
\begin{enumerate}[(i)]
  \item
  If $\pO$ is very curved, traveling along the boundary may turn out to be more expensive than first paying the toll to enter $\O$, moving next in the interior over a much shorter distance, and finally paying again the toll to reenter the ring road upon arrival at the target destination.
  \\
  For an explicit counterexample, take $\O$ a very flat ellipse with minor axis of fixed length $R$ but very large major axis, and pick two opposite points $x_0,x_1$ on the minor axis $I=[x_0,x_1]$ as in Figure~\ref{fig:ellipse}.
  We choose $  \rho_0:=\left(0,\delta_{x_0}\right)$ and $\rho_1:=\left(0 , \delta_{x_1}\right)$, hence
  $\o_0=\o_1=0$ while $\g_0=\delta_{x_0}$ and $\g_1=\delta_{x_1}$.
  On the one hand, choosing the major axis large enough, the distance $\W^2_\G(\g_0,\g_1)=\W^2_\G(\delta_{x_0},\delta_{x_1})=\frac{1}{2}d^2_\G(x_0,x_1)$ can clearly be made arbitrarily large.
  On the other hand, using twice the very same Fisher-Rao scenario as in the proof of Lemma~\ref{lem:Wrr_finite}, it is easy to construct an admissible 
  path connecting first $\rho_0=(0,\delta_{x_0})$ to $\rho_{1/3}:=(\delta_{x_0},0)$ in time $t\in[0,1/3]$, moving $\rho_{1/3}$ to $\rho_{2/3}:=(\delta_{x_1},0)$ following an interior Wasserstein geodesic $\W_\bO(\delta_{x_0},\delta_{x_1})$ in time $t\in[1/3,2/3]$ along $I=[x_0,x_1]$, and then transferring back $\rho_{2/3}=(\delta_{x_1},0)$ to $\rho_1=(0,\delta_{x_1})$ by pure reaction in time $[2/3,1]$.
  Taking into account the scaling in time gives a cost $3(1/2+R^2/2+1/2)$, which is clearly smaller than $\W^2_\G(\delta_{x_0},\delta_{x_1})=\frac 12 d^2_\G(x_0,x_1)$ if the major axis is sufficiently large.
  \begin{figure}
    \begin{center}
    \begin{tikzpicture}
  %     \draw[help lines,step=1] (-6,-2) grid (6,2);
      \draw (0,0) ellipse (6cm and 1cm);
      \draw [ultra thick] (0,-1)--(0,1);
      \draw [ultra thick,->] (0,-1)--(0,0);
      \filldraw (0,1) circle (2pt);
      \filldraw (0,-1) circle (2pt);
      \node [above] at (0,1) {$x_R$};
      \node [below] at (0,-1) {$x_0$};
      \node [right] at (.2,0) {$I$};
    \end{tikzpicture}
    \end{center}
  \caption{The flat ellipse and the bridge}
  \label{fig:ellipse}
  \end{figure}
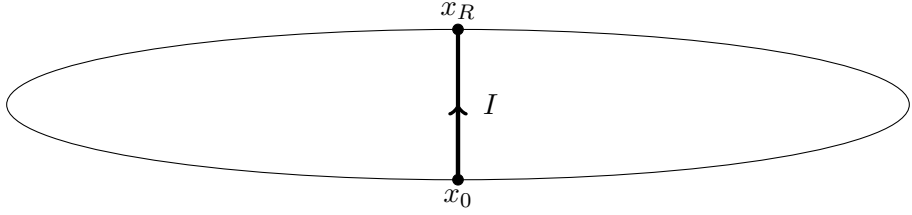

  %%%%%%%%%%%%%%%%%%%%%%%%%%%%%%%%%%%%%%%%%%%%%%%%%%%%%%%%%%%%%
  \item
  Recall from \eqref{eq:WFR_leq_Wrr} that there always holds $\Wrr^2(\rho_0,\rho_1)\geq \WFR^2(\g_0,\g_1)$, thus a counterexample can only come from strict inequality.
  The heuristic explanation is then as follows: It is known  \cite{chizat2018interpolating,kondratyev2016new} that a Wasserstein-Fisher-Rao geodesic $(\g^*,G^*,f^*)$ always has $f^*\neq 0$, unless $\g_0=\g_1$ (roughly speaking because $f^*=\g^*\frac {\psi^*}{\k^2}$ and $G^*=\g^*\nabla\psi^*$ for some scalar potential $\psi^*$, thus $f^*=0$ would imply $\psi^*=0$ and $G^*=0$ too and therefore $\g_0=\g_1$).
  This forces $F^*\neq 0$ through the flux condition $F^*\cdot n=f^*$, and in turn imposes a nontrivial motion and strictly positive kinetic action inside $\O$.
  It is precisely this interior kinetic action that forces a gap $\Wrr^2 >\WFR^2$.
  However this rationale does not take into account the fact that $\o$ may charge the boundary and act as a reservoir storing or releasing mass according to $+\partial_t\g=f=-\partial_t(\o\mrest\pO)$, and some caution must be taken.
  \\
  More rigorously, take from Theorem~\ref{theo:duality_exist_geodesics} a geodesic $\mu=(\o,F\ ;\ \g,G,f)$ from $\rho_0$ to $\rho_1$, and take a Wasserstein-Fisher-Rao geodesic $\mu^*_\G=(\g^*,G^*,f^*)$ from $\g_0$ to $\g_1$ \cite[Theorem 2.1]{chizat2018interpolating}.
  By definition of $\Wrr$ and $\WFR$ we always have
  \begin{multline}
    \label{eq:gap_Wrr_WFR_o0=01}
  \Wrr^2(\rho_0,\rho_1)
  = \A(\mu)
  \\
  =\iint_{\QO}\frac{|F|^2}{2\o} + \iint_{\QG}\frac{|G|^2+\k^2|f|^2}{2\g}
  \geq 
  \iint_{\QO}\frac{|F|^2}{2\o} + \inf\limits_{\g',G',f'}\iint_{\QG}\frac{|G'|^2+\k^2|f'|^2}{2\g'}
  \\
  =
  \iint_{\QO}\frac{|F|^2}{2\o} + \iint_{\QG}\frac{|G^*|^2+\k^2|f^*|^2}{2\g^*}
  =
  \iint_{\QO}\frac{|F|^2}{2\o} + \WFR^2(\g_0,\g_1).
  \end{multline}
  In the middle inequality the infimum is taken along solutions of $\partial_t\g'+\dive G'=f'$ connecting $\g_0,\g_1$ as in the definition \eqref{eq:def_WFR} of $\WFR$.
  
  Hence, in order to produce a strict inequality it suffices to exhibit a pair $\rho_0,\rho_1\in\Pp(\bO)$ such that $F\neq 0$.
  To this end take two points $x_0,x_1\in \G$ far away from each other for the intrinsic distance $d_\G$ on $\G$.
  We claim that any geodesic $(\o,F,\g,G,f)$ between $\rho_0:=\left(0,\delta_{x_0}\right)$ and $\rho_1:=\left(0,\delta_{x_1}\right)$ has $F\neq 0$.
  For if not, the integral formulation \eqref{eq:CE_omega} with $F=0$ easily shows that $\partial_t(\o\mrest\O)=0$ and $\partial_t(\o\mrest\pO)=-f$.
  In other words no real flux arises from the interior, $f$ only consists in a pure source term, and $\o\mrest\pO$ must act as a reservoir for whatever mass must be carried to--or discharged from--the boundary $\g$-species.
  With our choice of measures $\g_0=\delta_{x_0},\g_1= \delta_{x_1}$, and if $d_\G(x_0,x_1)>\pi \k$, it is known \cite[Theorem 4.1]{chizat2018interpolating} that the Wasserstein-Fisher-Rao geodesic is of pure Fisher-Rao reaction type, namely $\g_t= (1-t)^2\delta_{x_0} + t^2\delta_{x_1}$ with $G\equiv 0$.
  This prescribes $f_t=\partial_t\g_t=-2(1-t)\delta_{x_0}+2t\delta_{x_1}$, and the condition $\partial_t(\o_t\mrest \pO)=-f_t$ gives by direct integration $(\o_t\mrest\pO)=(\o_0\mrest\pO)-\int_0^tf_s\rd s=(1-t)^2\delta_{x_0}-t^2\delta_{x_1}$.
  This contradicts the positivity at $x_1$.
  As a consequence either $F\neq 0$ or the middle inequality in \eqref{eq:gap_Wrr_WFR_o0=01} is strict, and in any case we obtain $\Wrr^2(\rho_0,\rho_1)>\WFR(\g_0,\g_1)$ as desired.
  \end{enumerate}
  \begin{rmk}
   The opposite line of thoughts shows that the lower bound \eqref{eq:WFR_leq_Wrr} is optimal:
   As an example, take $\rho_0=\left(\frac 12 \delta_{x_1},\frac 12 \delta_{x_0}\right)$ and $\rho_1=\left(\frac 12 \delta_{x_0},\frac 12 \delta_{x_1}\right)$ for two points $x_0,x_1\in\pO$.
   In other words, put some initial $\o_0$-mass at $x_1$ where $\g_1=\frac{1}{2}\delta_{x_1}$ needs to be created, and don't put any $\o_0$-mass at $x_0$ where $\g_0=\frac{1}{2}\delta_{x_0}$ needs to discharge.
   Clearly $\rho_0,\rho_1$ can be connected by a pure ``reaction'' path with $F=G=0$, the optimal way to do this is precisely given by the Fisher-Rao geodesics between $\g_0,\g_1$, and thus $\Wrr^2(\rho_0,\rho_1)=\WFR^2(\g_0,\g_1)$ since no kinetic action is involved along the interpolation.
   It is interesting to note that the total density remains constant along the process, i-e $\varrho_0=\varrho_t=\varrho_1=\frac{1}{2}\delta_{x_0}+\frac{1}{2}\delta_{x_1}$.
  \end{rmk}
\begin{enumerate}[(i)]
\setcounter{enumi}{2}
  \item 
  From Proposition~\ref{prop:comparison_other_distances}(\ref{item:prop_comparison_|o0|=|01|_|g0|=|g1|})(\ref{item:prop_comparison_W2_leq_Wrr}) with $\g_0=\g_1$ we already know that
  \begin{equation}
  \label{eq:W2_leq_Wrr_leq_W2int}
  \W^2_\bO(\varrho_0,\varrho_1)\leq \Wrr^2(\rho_0,\rho_1) \leq \W^2_{\bO}(\o_0,\o_1) + 0.
  \end{equation}
  With our assumption that $\O$ is convex and because $\varrho_0\mrest\pO=(\o_0\mrest \pO)+\g_0=(\o_1\mrest \pO)+\g_1=\varrho_1\mrest\pO$, standard arguments from classical optimal transport guarantee that the whole boundary $\pO$ is fixated in the Monge-Kantorovich problem defining $\W^2_\bO(\varrho_0,\varrho_1)$.
  As a consequence $\W^2_\bO(\varrho_0,\varrho_1)=\W^2_\bO(\varrho_0\mrest\O,\varrho_1\mrest\O)=\W^2_\bO(\o_0\mrest\O,\o_1\mrest\O)$.
  The very same convexity argument with now $\o_0\mrest\pO=\o_1\mrest\pO$ also guarantees that $\W^2_\bO(\o_0,\o_1)=\W^2_\bO(\o_0\mrest\O,\o_1\mrest\O)$, hence $\W^2_\bO(\varrho_0,\varrho_1)=\W^2_\bO(\o_0,\o_1)$ in \eqref{eq:W2_leq_Wrr_leq_W2int} and \eqref{eq:Wrr=W2_g_fixed} follows.
  \\
  To see that the convexity of $\O$ is really required, choose a non-convex domain $\O$ and some $\rho_0,\rho_1$ with $\g_0=\g_1$ but $\o_0,\o_1$ supported in the interior such that $\W_\bO^2(\varrho_0,\varrho_1)<\W_\bO^2(\o_0,\o_1)$.
  (Take e.g. $\O$ banana-shaped, with $\varrho_0=\frac 12\delta_{x_0}+\frac 12 \delta_y$ and $\varrho_1=\frac{1}{2}\delta_{x_1}+\frac 12 \delta_y$ for two points $x_0,x_1\in\O$ such that the segment $[x_0,x_1]$ is tangent to $\pO$ at $y\in\pO$.)
  Anticipating that there always holds $\lim\limits_{\k\to 0}\Wrr^2(\rho_0,\rho_1)=\W^2_\bO(\varrho_0,\varrho_1)$, see Theorem~\ref{theo:small_toll} later on, the result follows from $\Wrr^2(\rho_0,\rho_1)\sim \W^2_\bO(\varrho_0,\varrho_1) < \W_\bO^2(\o_0,\o_1)$ for small $\k$ and the proof is complete.
  (Of course the proof of Theorem~\ref{theo:small_toll} will not rely on the present statement and there is no circular reasoning here.)
\end{enumerate}
\end{proof}
Finally, let us record for completeness an easy consequence of the previous Proposition~\ref{prop:comparison_other_distances}:
\begin{prop}
\label{prop:Wrr=W_no_boundary_mass}
For any $\rho_0=(\o_0,0)\in\Pp(\bO)$ and $\rho_1=(\o_1,0)\in\Pp(\bO)$ there holds
$$
\Wrr^2(\rho_0,\rho_1)=\W^2_{\bO}(\o_0,\o_1)
=\W^2_{\bO}(\varrho_0,\varrho_1).
$$
Moreover a geodesic is given by $(\o,F,\g,G,f)=(\o^*,F^*,0,0,0)$ for any interior Wasserstein geodesic $(\o^*,F^*)$ between $\o_0$ and $\o_1$ in $\bO$.
\end{prop}
We stress that this holds regardless of any convexity assumption on $\O$, and is a desired feature of our model: In the absence of $\g$-mass on the boundary the dynamics should be governed by classical optimal transport.
This also shows that the bounds \eqref{eq:Wrr_leq_W+W}\eqref{eq:Wrr_geq_W} are sharp.
\begin{rmk}
When including a parameter $\delta>0$ in the boundary kinetic cost $\delta^2\frac{|G|^2}{2\g}$, Proposition~\ref{prop:Wrr=W_no_boundary_mass} may completely fail:
In the opposite spirit to Proposition~\ref{prop:comparison_other_distances_fixed}(\ref{item:Wrr_neq_W2_o_fixed}) and Figure~\ref{fig:ellipse}, if $\delta,\k\ll 1$ are sufficiently small then it may turn out to be much more efficient to first pay the toll, move along the fast ring road at a very cheap price $\mathcal O(\delta^2)$, and then pay again the toll to exit the ring road for a total cost $\mathcal O(\k^2+\delta^2)$, rather than avoiding the toll but then having to remain in the city and traveling at a more expensive cost $\mathcal O(1)$.
\end{rmk}
\begin{proof}
Equality of the distances immediately follows by \eqref{eq:Wrr_leq_W+W}\eqref{eq:Wrr_geq_W} with here $\W^2_{\bO}(\o_0,\o_1) + \W^2_\G(\g_0,\g_1)=\W^2_{\bO}(\varrho_0,\varrho_1)+ 0$ since $\g_0=\g_1=0$.
For the second part of the statement, pick any Wasserstein geodesic $(\o^*,F^*)$ between $(\o_0,\o_1)$.
Since Wasserstein geodesics have by definition zero-flux $f=0$ on the boundary, it is easy to check that $\mu=(\o^*,F^*;0,0,0)\in\CE(\rho_0,\rho_1)$ in the sense of Definition~\ref{def:CE}.
Moreover by the first step $\A(\mu)=\A_\O(\o^*,F^*)+\A_\G(0,0,0)=\W^2_{\bO}(\o_0,\o_1)+0=\Wrr^2(\rho_0,\rho_1)$, hence $\mu$ is a geodesic.
\end{proof}

%%%%%%%%%%%%%%%%%%%%%%%%%%%%%%%%%%%%%%%%%%%%%%%%%%%%%%%%%%%%%%%%%%%%%%%%%%%%%%%%%%%%%%%%%%%%%%%%%%%%%%%%%%%%%%%%%%%%%%%%%%%%%%%%%%%%%%%%%%%%%%%%%%%%%%%%%%
%%%%%%%%%%%%%%%%%%%%%%%%%%%%%%%%%%%%%%%%%%%%%%%%%%%%%%%%%%%%%%%%%%%%%%%%%%%%%%%%%%%%%%%%%%%%%%%%%%%%%%%%%%%%%%%%%%%%%%%%%%%%%%%%%%%%%%%%%%%%%%%%%%%%%%%%%%
%%%%%%%%%%%%%%%%%%%%%%%%%%%%%%%%%%%%%%%%%%%%%%%%%%%%%%%%%%%%%%%%%%%%%%%%%%%%%%%%%%%%%%%%%%%%%%%%%%%%%%%%%%%%%%%%%%%%%%%%%%%%%%%%%%%%%%%%%%%%%%%%%%%%%%%%%%

\subsection{Topological properties}
\label{subsec:topological_properties}
Most--if not all--distances usually involved in optimal transportation share the property that they metrize the narrow convergence, whether it be the Wasserstein, Wasserstein-Fisher-Rao, bounded-Lipschitz distances, etc.
By construction however, our $\Wrr$ metric clearly distinguishes the boundary and the interior via the non-reducible toll.
On the other hand the narrow convergence on $\bO$ does not see any particular distinction between the interior and the boundary, thus one could expect that our distance induces a stronger topology:
\begin{theo}
\label{theo:double_narrow_cv}
The distance $\Wrr$ metrizes the ``double'' narrow convergence, i-e $\Wrr(\rho_n,\rho)\to 0$ if and only if $\o_n\narrowcv \o$ in $\bO$ and $\g_n\narrowcv  \g$ in $\G$.
Moreover the space $\left(\Pp(\bO),\Wrr\right)$ is complete.
\end{theo}
Note that the double narrow convergence is strictly stronger than ``total'' convergence $\varrho_n=\o_n+\g_n\narrowcv\o+\g=\varrho $ in $\bO$ of the overall densities.
The typical example of a sequence of totally -- but not doubly -- converging sequence is $\rho_n=(\delta_{x_n},0)$ and $\rho=(0,\delta_x)$ for a sequence $x_n\in\O$ converging to some $x\in\pO$.
This sequence abruptly jumps from the interior ($\g_n=0$ for all $n$) to the boundary ($\g =\delta_x$) in the limit:
Due to the non-reducible toll this has a fixed positive cost $\Wrr(\delta_{x_n},\delta)\geq \mathcal O(\k)>0$ and therefore the sequence cannot converge for the $\Wrr$ topology.
Also, this double narrow convergence is completely equivalent to narrow convergence on $\P(\bO\cup\G)$.
\begin{proof}
Let $\rho_n=(\o_n,\g_n)$ be a sequence converging to $\rho=(\o,\g)$, i-e $\Wrr(\rho_n,\rho)\to 0$.
Owing to \eqref{eq:dBL_leq_Wrr} we see that $d_{\mathrm{BL},\bO}(\o_n,\o)\to 0$ and $d_{\mathrm{BL},\G}(\g_n,\g)\to 0$.
Since the bounded-Lipschitz distance metrizes the narrow convergence we see that $\o_n\narrowcv\o$ and $\g_n\narrowcv\g$ in $\bO,\G$, respectively.
\\
Conversely, assume that $\o_n\narrowcv\o$ and $\g_n\narrowcv\g$.
If $\TV{\o_n}=\TV{\o}$ for all $n$ (thus $\TV{\g_n}=\TV{\g}$ as well) then we would be done: since the (classical, conservative) Wasserstein distances on $\bO,\G$ metrize the corresponding narrow convergences \cite[Theorem 6.9]{villani_BIG} we would immediately get by Proposition~\ref{prop:comparison_other_distances}(\ref{item:prop_comparison_|o0|=|01|_|g0|=|g1|}) $\Wrr^2(\rho_n,\rho)\leq \W^2_\bO(\o_n,\o)+\W^2_\G(\g_n,\g)\to 0$.

The rest of the proof below will consist in reducing to this case of fixed masses, up to paying a negligible price.
More precisely, we will construct a sequence $\tilde\rho_n=(\tilde\o_n,\tilde\g_n)\in\Pp(\bO)$ such that $\Wrr^2(\rho_n,\tilde\rho_n)\to 0$, $\TV{\tilde\o_n}=\TV{\o}$ and $\TV{\tilde\g_n}=\TV{\g}$, as well as $\tilde\o_n\narrowcv\o$ and $\tilde\g_n\narrowcv\g$.
The previous discussion will guarantee $\Wrr(\tilde\rho_n, \rho)\to 0$, and by triangular inequality we will conclude that $\Wrr(\rho_n,\rho)\leq \Wrr(\rho_n,\tilde\rho_n)+\Wrr(\tilde\rho_n, \rho)\to 0$. 

In order to make this rigorous, we use
$$
\eps_n:=\TV{\g_n}-\TV{\g}
$$
as a control parameter.
If $\eps_n=0$ then $\TV{\g_n}=\TV{\g}$ and $\TV{\o_n}=\TV{\o}$, hence $\rho_n$ needs not be modified.
Consider first the case of an excess of mass on the boundary, $\eps_n>0$.
By narrow convergence we have of course $\eps_n\to 0$.
In order to construct $\tilde\rho_n$ the idea is to first create an annular gap around $\pO$ at small cost, and then infiltrate the small excess of mass $\eps_n>0$ from $\pO$ into the small gap -- again for a small price -- using the geodesics between point-masses from Theorem~\ref{theo:explicit_1D_geodesics}.
The whole process will be accomplished in three successive steps $\rho_n\leadsto\hat\rho_n\leadsto\check\rho_n\leadsto\tilde\rho_n$.
Each new measure will remain close to the previous one in the $\Wrr$ distance and in the narrow topology.
\begin{enumerate}
 \item 
  The first step will not modify $\g_n$.
  Pick a smooth, constant-in-time velocity field $v=v(x)$ pointing normally inward with unit norm on a fixed but sufficiently small tubular neighborhood of $\pO$, and satisfying moreover $\|v\|_{L^\infty(\bO)}=1$.
  Let $\hat \o_n:=(\Phi^{v}_{3\tau_n})\pf\o_n$ be the measure obtained by following the $v$-flow starting from $\o_n$ for a time $3\tau_n$ with $\tau_n$ sufficiently small.
  Since $v$ points inward this flow is mass conservative, $\TV{\hat\o_n}=\TV{\o_n}$.
  Choosing $\tau_n$ small enough, each Lagrangian particle is moving over a distance at most $3\tau_n\|v\|_\infty$, hence $\W_\bO^2(\hat\o_n,\o_n)=\mathcal O(\tau_n^2)=o(1)$.
  In particular $\hat\o_n-\o_n\narrowcv 0$, and according to \eqref{eq:Wrr_leq_W+W}  $\hat\rho_n:=(\hat\o_n,\g_n)$ satisfies $\Wrr^2(\hat\rho_n,\rho_n)\leq \W^2_\bO(\hat\o_n,\o_n)+0\to 0$.
  Moreover since we decided to follow the inward unit velocity field $v$ for time $3\tau_n$ we have now a tubular gap around $\pO$ of size at least $2\tau_n$, i-e $\dist(\supp \hat\o_n,\pO)\geq 2\tau_n$.
  \item
  The second step will leave now $\hat\o_n$ from the previous step unchanged.
  Fix an arbitrary point $y\in\pO$ and choose a small $r_n>0$.
  Using only mass displacement along the boundary (i-e $\partial_t\g+\dive G=0$) it is easy to first open up a hole of size $r_n$ around $y$ and then bring back a small $\eps_n$ mass at the center $y$ -- see Figure~\ref{fig:interior_transfer_neighborhood}.
  This newly defined measure $\check\g_n$ was obtained by moving first some mass (possibly of order one) over a distance at most $r_n$, and then moving a mass $\eps_n$ over a distance at most $\operatorname{diam}(\G)$.
  As a result $\W^2_\G(\check\g_n,\hat\g_n)=\mathcal O(r_n^2+\eps_n)=o(1)$.
  Moreover by construction $\check\g_n-\hat\g_n\narrowcv 0$ as required, and from \eqref{eq:Wrr_leq_W+W} we see that $\check \rho_n:=(\hat \o_n,\check\g_n)$ satisfies $\Wrr^2(\check\rho_n,\hat\rho_n)\leq 0+\W^2_\G(\check\g_n,\hat\g_n)\to 0$.
  \item
  The final step will transfer the $\eps_n$-excess of mass from $\G$ to $\O$ and pay the corresponding toll charge, which is expected to be small since this mass is small.
  After the previous steps we have now an interior safety cylinder of length at least $\tau_n$ and radius $r_n$ around $y$, containing no mass except at $y$--see again Figure~\ref{fig:interior_transfer_neighborhood}.
  Let us put inside this cylinder a one-dimensional segment $I_n=[x_n,y]$ of length $\tau_n\ll 1$, for some $x_n\in\O$ close to $y$.
  Using the geodesic between point-particles from Theorem~\ref{theo:explicit_1D_geodesics} and leaving everything outside of the safety cylinder untouched, it is easy to construct an admissible path between $(0,\eps_n\delta_y)$ and $(\eps_n\delta_{x_n},0)$ by simply multiplying \eqref{eq:interpolant_1d_o}\eqref{eq:interpolant_1d_g} by $\eps_n$.
  The resulting cost is simply $\eps_n$ times \eqref{eq:explicit_Wrr_1D_diracs} with $R=|x_n-y|=\tau_n\ll 1$, and therefore the final measure $\tilde\rho_n:= (\check\o_n+\eps_n\delta_{x_n}, \check\g_n-\eps_n\delta_y)$ satisfies $\Wrr^2(\check\rho_n,\tilde\rho_n)\leq \eps_n\frac{1}{2}(\tau_n^2 +\k^2\alpha_\k)\frac{\alpha_\k}{\alpha_\k-1}\to 0$.
  Moreover since only a small fraction of mass $\eps_n>0$ was involved in this last step we have of course $\tilde\o_n-\check\o_n\narrowcv 0$ and $\tilde\g_n-\check\g_n\narrowcv 0$.
\end{enumerate}
\noindent
This deals with the case $\eps_n=\TV{\g_n}-\TV{\g}>0$.
  \begin{figure}[h!]
  \begin{center}
  \begin{tikzpicture}[use Hobby shortcut]
%     \draw[help lines,step=.5] (-6,-2) grid (1,2);
    \draw [domain=-70:70] plot ({-3+3*cos(\x)}, {0+3*sin(\x)});
    \draw [domain=-70:70,dotted,thick] plot ({-3+cos(\x)}, {0+sin(\x)});
    \draw [<->] (0,-1.5)--(-2,-1.5);
    \node at (-1,-1.8) {$2\tau_n$};
    \draw [red, ultra thick,domain=10:70] plot ({-3+3*cos(\x)}, {0+3*sin(\x)});
    \draw [red, ultra thick,domain=-10:-70] plot ({-3+3*cos(\x)}, {0+3*sin(\x)});
%     \draw [domain=105:255,dotted, thick] plot ({1*cos(\x)}, {1*sin(\x)});
    \draw [dotted, thick] (-.1,.5) -- (-1.4,.5) -- (-1.4,-.5) -- (-.1,-.5);
    \draw [<->] (0,-.9) -- (-1,-.9);
    \node at (-.5,-1.1) {$\tau_n$};
    \draw [<->] (1,0) -- (1,-.5);
    \node [right] at (1,-.25) {$r_n$};
    \draw [thick] (-1,0)--(0,0);
    \draw [thick,->] (0,0)--(-.5,0);
    \node [above] at (-1,0) {$x_n$};
    \filldraw (-1,0) circle (2pt);
    \filldraw [red] (0,0) circle (2pt);
    \node [right] at (0,0) {$y$};
    \node [below] at (-.5,0) {$I_n$};
    \node [red] at (.5,1) {$\check\g_n$};
    \node  at (-2,2) {$\O$};
    \node  at (-.5,2.5) {$\G$};
      \path
  (-2.5,0) coordinate (z0)
  (-4,1) coordinate (z1)
  (-5,0) coordinate (z2)
  (-5,-.5) coordinate (z3)
  (-4,-1) coordinate (z4);
  \fill[closed, color=blue!10] (z0) .. (z1) .. (z2) .. (z3) .. (z4);
    \node [blue] at (-4,0) {$\supp\check\o_n$};
  \end{tikzpicture}
\end{center}
\caption{Discharge of the $\eps_n$-excess of mass from $y$ to $x_n$ within the small safety cylinder}
\label{fig:interior_transfer_neighborhood}
\end{figure}
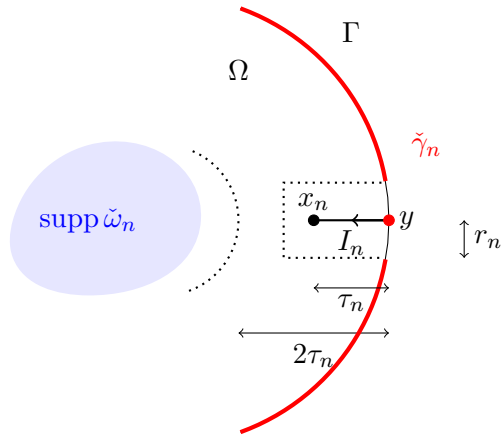

If $\eps_n<0$ we simply use the same three elementary steps in a different order: first modify $\o_n$ so as to confine almost all the interior mass outside of a $2\tau_n$ tubular neighborhood of $\pO$.
Then bring back a small mass $\eps_n$ to create an atomic measure $\eps_n\delta_{x_n}$ inside a small cylinder based at $y\in\pO$.
Modify next $\g_n$ so as to puncture a small $r_n$-neighborhood around $y$ on the boundary.
Finally, transfer the $\eps_n$-mass from the interior point $x_n$ to $y$ using the (suitably rescaled) geodesic from Theorem~\ref{theo:duality_exist_geodesics}.
This settles the case $\eps_n<0$ and establishes our first statement.

Regarding now the completeness, let $\{\rho_n\}_n\subset\Pp(\bO)$ be a Cauchy sequence, $\Wrr(\rho_p,\rho_q)\to 0$ as $p,q\to+\infty$.
Owing to the bounded-Lipschitz estimate \eqref{eq:dBL_leq_Wrr} we see that $\o_n,\g_n$ are Cauchy for the $d_{\mathrm{BL},\bO},d_{\mathrm{BL},\G}$ distances, respectively.
Since $(\M^+,d_{\mathrm{BL}})$ is complete there is a pair $\rho=(\o,\g)$ such that $\o_n\to\o$ and $\g_n\to\g$ for the respective bounded-Lipschitz distances.
Because the latter metrize the respective narrow convergences we have $\o_n\narrowcv\o$ in $\bO$ and $\g_n\narrowcv\g$ in $\G$, hence in particular $\rho=(\o,\g)\in\Pp(\bO)$.
By the first step this double narrow convergence characterizes $\Wrr(\rho_n,\rho)\to 0$ and the proof is complete.
\end{proof}

%%%%%%%%%%%%%%%%%%%%%%%%%%%%%%%%%%%%%%%%%%%%%%%%%%%%%%%%%%%%%%%%%%%%%%%%%%%%%%%%%%%%%%%%%%%%%%%%%%%%%%%%%%%%%%
%%%%%%%%%%%%%%%%%%%%%%%%%%%%%%%%%%%%%%%%%%%%%%%%%%%%%%%%%%%%%%%%%%%%%%%%%%%%%%%%%%%%%%%%%%%%%%%%%%%%%%%%%%%%%%
\section{Varying the toll}
\label{sec:varying toll}
Up to now the parameter $\k>0$ was fixed.
In this section we investigate the behavior of the distance and of the geodesics in the large and small toll limits $\k\to+\infty,\k\to0$.
\\

We first recall from Proposition~\ref{prop:comparison_other_distances}(\ref{item:prop_comparison_|o0|=|01|_|g0|=|g1|})(\ref{item:prop_comparison_W2_leq_Wrr}) that our ring road distance is sandwiched between Wasserstein distances as
$$
\W_\bO^2(\varrho_0,\varrho_1)
\leq
\Wrr^2(\rho_0,\rho_1)
\leq 
\W^2_{\bO}(\o_0,\o_1) + \W^2_{\G}(\g_0,\g_1),
$$
both bounds being sharp.
The upper bound should be understood here in the general sense, i-e $\W^2_{\bO}(\o_0,\o_1),\W^2_{\G}(\g_0,\g_1)=+\infty$ in case of mass incompatibility $\TV{\o_0}\neq \TV{\o_1},\TV{\g_0}\neq \TV{\g_1}$.
Both bounds are sharp from Proposition~\ref{prop:Wrr=W_no_boundary_mass}, and we have moreover
\begin{prop}
  \label{prop:monotonicity_k}
 For fixed $\rho_0,\rho_1\in\Pp(\bO)$ the map $\k\mapsto \Wrr(\rho_0,\rho_1)$ is nondecreasing.
\end{prop}
This strongly suggests that $\Wrr$ should converge as
$$
\W_\bO^2(\varrho_0,\varrho_1)
\xleftarrow[\k\to 0]{} 
\Wrr^2(\rho_0,\rho_1)
\xrightarrow[\k\to +\infty]{} 
\W^2_{\bO}(\o_0,\o_1) + \W^2_{\G}(\g_0,\g_1)
$$
when $\k$ varies.
This is indeed the case as we shall see below, and both limits will be established separately.
\begin{proof}
Note that the set of smooth subsolutions
$$
\mathcal S^\k:=
\Bigg\{
(\phi,\psi)\in \C^1(\QO)\times\C^1(\QG):
\qquad
\partial_t\phi+\frac{|\nabla\phi|^2}{2}\leq 0
\mbox{ and }
\partial_t\psi+\frac{|\nabla\psi|^2}{2}+\frac{|\psi-\phi|^2}{2\k^2}\leq 0
\Bigg\}
$$
is nondecreasing in $\k$.
The monotonicity immediately follows from the duality
$$
\Wrr^2(\rho_0,\rho_1)=\sup\limits_{(\phi,\psi)\in\mathcal S^\k}
\left\{
  \int_\bO\phi(1,.)\rd\o_1-\int_\bO\phi(0,.)\rd\o_0  +  \int_\G \psi(1,.)\rd\g_1  -  \int_\G\psi(0,.)\rd\g_0
\right\}
$$
in Theorem~\ref{theo:duality_exist_geodesics}.
\end{proof}

%%%%%%%%%%%%%%%%%%%%%%%%%%%%%%%%%%%%%%%%%%%%%%%%%%%%%%%%%%%%%%%%%%%%%%%%%%%%%%%%%%%%%%%%%%%%%%%%%%%%%%%%%%%%%%
\subsection{The large toll limit $\k\to+\infty$}
When $\k\to+\infty$ the mass flux $f$ between $\O$ and $\pO$ is penalized more and more heavily, and one should expect that in the limit no such flux can persist.
If $\rho_0,\rho_1$ have different masses on the interior and boundary then some flux is really needed in order to connect them, thus we expect that large tolls should lead to
\begin{equation}
  \label{eq:Wrr_k_to_infty_different_mass}
 \Wrr(\rho_0,\rho_1)\xrightarrow[\k\to+\infty]{}+\infty
\hspace{1cm}\mbox{for } \TV{\o_0}\neq \TV{\o_1}
\mbox{ or }\TV{\g_0}\neq \TV{\g_1}.
\end{equation}
On the other hand for compatible masses $\TV{\o_0}=\TV{\o_1},\TV{\g_0}=\TV{\g_1}$ we have from Proposition~\ref{prop:comparison_other_distances} the upper bound
$$
 \Wrr^2(\rho_0,\rho_1) \leq \W^2_{\bO}(\o_0,\o_1) + \W^2_{\G}(\g_0,\g_1).
$$
As $\k\to+\infty$ we expect that no flux should be allowed $f=0$, thus the two continuity equations in Definition~\ref{def:CE} should uncouple and we are left with two independent continuity equations, $\partial_t\o+\dive F=0$ in $\bO$ (with no-flux boundary conditions on $\pO$) and $\partial_t\g+\dive G=0$ in $\G$.
Since the flux $f=0$ the total action becomes $\frac{|F|^2}{2\o} + \frac{|G|^2}{2\g}+0$, and the minimization is now uncoupled in $(\o,F)$ and $(\g,G)$.
Thus one expects to retrieve $\Wrr^2(\rho_0,\rho_1) \to \W^2_{\bO}(\o_0,\o_1) + \W^2_{\G}(\g_0,\g_1)$.
\\

In order to make this more rigorous, we first have
\begin{lem}
  \label{lem:Wrr>f_TV>|g1-g0|}
 For any $\rho_0,\rho_1\in\Pp(\bO)$ and any geodesic $(\o,F,\g,G,f)$ there holds
 $$
 \Wrr^2(\rho_0,\rho_1)
 \geq \frac{\k^2}{2}\TV{f}^2
 \geq \frac{\k^2}{2} \Big|\TV{\g_1}-\TV{\g_0}  \Big|^2.
 $$
\end{lem}
This will immediately prove \eqref{eq:Wrr_k_to_infty_different_mass}.
\begin{proof}
  The argument is reminiscent from the proof of Proposition~\ref{prop:comparison_other_distances}(\ref{item:WFR_leq_Wrr}).
  Pick from Theorem~\ref{theo:duality_exist_geodesics} a geodesic $\mu=(\o,F\ ;\ \g,G,f)$ from $\rho_0$ to $\rho_1$.
  By definition (and with a slight abuse of notation) we see that
  $$
    \Wrr^2(\rho_0,\rho_1)=\A(\mu)
    \geq
    \frac {\k^2}{2} \iint_\QG \frac{f^2}{\g}
    =
    \frac {\k^2}{2} \iint_\QG r^2\, \rd\g,
  $$
  where we exploited Proposition~\ref{prop:disintegration_momentum_velocity} to express the flux cost in terms of $r=\frac{\rd f}{\rd \g}$.
  Owing to \eqref{eq:|o_t|+|g_t|=|rho_t|=1} we have moreover $\TV{\g}\leq \TV{\varrho}=1$, thus by Jensen's inequality
  $$
  \iint_\QG r^2\, \rd\g
  \geq \frac{1}{\TV{\g}}\left(\iint_\QG |r|\rd\g\right)^2
  \geq \left(\iint_\QG |r|\rd\g\right)^2
  = \TV{f}^2
  $$
  because $\frac{\rd|f|}{\rd \g}=|r|$.
  The continuity equation $\partial_t\g +\dive G=f$ finally controls the mass difference as
  $$
  \TV{f}
  \geq
  \left| \iint_\QG 1\rd f\right|
  \overset{\eqref{eq:CE_gamma}}{=}
  \left|\int_\G 1\rd\g_1  -\int_\G 1\rd\g_1 \right|
  =
  \Big| \TV{\g_1}-\TV{\g_0}\Big|.
  $$
\end{proof}
This settles the case of incompatible masses.
For the general case we have
\begin{theo}
  \label{theo:CV_Wrr_large_toll}
 For fixed $\rho_0,\rho_1\in\Pp(\bO)$ there holds
 \begin{equation}
    \label{eq:Wrr_cv_to_W+W}
  \Wrr^2(\rho_0,\rho_1)
  \xrightarrow[\k\to+\infty]{}
  \W^2_\bO(\o_0,\o_1)+\W^2_\G(\g_0,\g_1).
  \end{equation}
 Let moreover $\mu^\k=(\o^\k,F^\k\ ;\ \g^\k,G^\k,f^\k)$ be any geodesic for $\Wrr^2(\rho_0,\rho_1)$.
 If the mass compatibility $\TV{\o_0}=\TV{\o_1}$ holds (hence $\TV{\g_0}=\TV{\g_1}$ too) then up to a subsequence
 \begin{equation}
   \label{eq:geodesics_cvnarrow_k_to_infty}
   (\o^\k,F^\k) \narrowcv (\o,F),
   \qquad
   (\g^\k,G^\k) \narrowcv (\g,G),
   \qtext{and}
   \TV{f^\k}\to 0
\end{equation}
  for two geodesics $(\o,F)$ and $(\g,G)$ in $\W^2_\bO(\o_0,\o_1)$ and $\W^2_\G(\g_0,\g_1)$, respectively.
\end{theo}
Note that this does not say anything about the asymptotic behavior of geodesics for incompatible masses.
Note also that uniqueness of the $\W_\bO,\W_\G$ limit geodesics would allow to dispense from subsequences, in which case the whole sequence of geodesics would actually converge.
This might be useful numerically speaking, although uniqueness should of course not be expected in general without further assumptions on the geometry of $\O$ or on $\o_0,\o_1$ and $\g_0,\g_1$.
\begin{proof}
  The convergence \eqref{eq:Wrr_cv_to_W+W} was already proved in \eqref{eq:Wrr_k_to_infty_different_mass} for incompatible masses, thus in the rest of the proof we only consider $\TV{\o_0}=\TV{\o_1}$, $\TV{\g_0}=\TV{\g_1}$ and therefore $\W^2_\bO(\o_0,\o_1)+ \W_\G^2(\g_0,\g_1)<+\infty$.
  \\
  We first control the flux term $\TV{f^k}$.
 Since we are in the case of compatible masses we can appeal to \eqref{eq:Wrr_leq_W+W}, and Lemma~\ref{lem:Wrr>f_TV>|g1-g0|} gives
 $$
 \TV{f^\k}^2
 \leq \frac{2}{\k^2} \Wrr^2(\rho_0,\rho_1)
 \leq \frac{2}{\k^2}\left(
\W^2_\bO(\o_0,\o_1)+ \W_\G^2(\g_0,\g_1) \right)\to 0.
$$
We retrieve next some compactness on $\o^\k,F^\k,\g^\k,G^\k$.
To this end, recall first from \eqref{eq:|o_t|+|g_t|=|rho_t|=1} that we have the mass conservation
$
\TV{\o^\k}+\TV{\g^\k}=1,
$
thus $\TV{\o^\k},\TV{\g^\k}\leq 1$ uniformly.
For the momenta $F^\k,G^\k$ observe from \eqref{eq:Wrr_leq_W+W} that any geodesic satisfies
$$
\frac{1}{2}\iint_\QO \frac{|F^\k|^2}{\o^\k} +\frac{1}{2}\iint_\QG \frac{|G^\k|^2}{\g^\k}
  \leq 
  \Wrr^2(\rho_0,\rho_1)\leq \W^2_\bO(\o_0,\o_1)+ \W_\G^2(\g_0,\g_1)
$$
uniformly in $\k>0$.
Using the exact same Jensen inequality as in the proof of Lemma~\ref{lem:Wrr>f_TV>|g1-g0|} we see that, with $F^\k=u^\k \o^\k,G^\k=v^\k\g^\k$,
\begin{multline*}
  \TV{F^\k}^2 + \TV{G^\k}^2
  =
  \left(\iint_\QO |u^\k|\rd \o^\k\right)^2
+
 \left(\iint_\QG |v^\k| \rd\g^\k\right)^2
 \\
 \leq 
 \iint_\QO |u^\k|^2\rd \o^\k +\iint_\QG |v^\k|^2 \rd\g^\k
 =
\iint_\QO \frac{|F^\k|^2}{\o^\k} +\iint_\QG \frac{|G^\k|^2}{\g^\k}
\end{multline*}
and therefore we have a uniform total variation bound $\TV{F^\k} + \TV{G^\k} \leq C$.
Prokhorov's theorem guarantees the weak-$\ast$ compactness $(\o^\k,F^\k,\g^\k,G^\k)\narrowcv (\o,F,\g,G)$ up to subsequences, and we only have to prove that the limits $(\o,F)$ and $(\g,G)$ are necessarily separate Wasserstein geodesics.\\
This will be ensured by the following elementary result for weighted optimization problems (we omit the proof for brevity):
\begin{lem}
  \label{lem:weighted_optimization}
  Let $K$ be a compact set, take $\mathfrak f,\mathfrak g:K\to \R^+\cup\{+\infty\}$ two proper, lower semi-continuous functions, and consider
  $$
  \mathfrak h_\k(x):= \mathfrak f(x) +\k^2 \mathfrak g(x),
  \qquad \k>0.
  $$
  Assume that for all $\k$ there is a minimizer $x_\k\in K$ of $\mathfrak h_\k$.
  Then as $\k\to+\infty$ any cluster point $x_*$ of $\{x_\k\}$ minimizes $\mathfrak f$ in $\argmin \mathfrak g$.
\end{lem}
Here we choose $K\subset\CE(\rho_0,\rho_1)$ to be the set of all geodesics for all values of $\k\geq 1$, which is narrowly compact by the previous discussion and because the linear continuity equations \eqref{eq:CE_omega}\eqref{eq:CE_gamma} are stable under narrow limits.
The functions
$$
\mathfrak f(\o,F,\g,G,f):= \frac{1}{2}\iint _\QO \frac{|F|^2}{\o} + \frac 12 \iint_\QG \frac{|G|^2}{\g}
\qqtext{and}
\mathfrak g(\o,F,\g,G,f):= \frac 12 \iint_\QG \frac{f^2}{\g}
$$
are convex, proper, and lower semicontinuous w.r.t. the narrow (weak-$\ast$) convergence of measures \cite[Theorem 3.3]{bouchitte1990new}, and geodesics are of course minimizers of $\mathcal A=\mathfrak f+\k^2\mathfrak g$.
By definition of $\frac{f^2}{\g}$ (in the extended sense) the minimizers of $\mathfrak g$ are nothing but solutions of the continuity equations \eqref{eq:CE_omega}\eqref{eq:CE_gamma} of the form $\mu=(\o,F,\g,G,0)$, and any minimizer simply assigns the value $\iint\frac{f^2}{\g}=0$ to $\mathfrak g$.
This set of solutions $\CE(\rho_0,\rho_1)$ with $f=0$ obviously identifies with the whole set of pairs $(\o,F)$ and $(\g,G)$ of independent solutions of $\partial_t\o+\dive F=0$ with no-flux conditions, and $\partial_t\g+\dive G=0$, respectively.
Clearly minimizers of the sum $\frac{1}{2}\iint _\QO \frac{|F|^2}{\o} + \frac 12 \iint_\QG \frac{|G|^2}{\g} $ over all such pairs are given by minimizers $(\o,F)$ of $\frac{1}{2}\iint _\QO \frac{|F|^2}{\o}$ on the one hand, and minimizers $(\g,G)$ of $\frac{1}{2}\iint _\QG \frac{|G|^2}{\g}$ on the other hand.
This shows that the limits $(\o,F)$ and $(\g,G)$ are indeed Wasserstein geodesics for $\W_\bO^2(\o_0,\o_1)$ and $\W^2_\G(\g_0,\g_1)$, respectively.
 \\
Let us finally address the convergence in distance \eqref{eq:Wrr_cv_to_W+W}.
From \eqref{eq:Wrr_leq_W+W}, and by lower semi-continuity of the actions with $(\o^\k,F^\k) \narrowcv (\o,F)$ and $(\g^\k,G^\k)\narrowcv(\g,G)$, we see that
\begin{multline*}
  \limsup\limits_{\k\to+\infty} \Wrr^2(\rho_0,\rho_1)
  \leq
  \W^2_{\bO}(\o_0,\o_1) + \W^2_{\G}(\g_0,\g_1)
%   \\
  =
 \iint _\QO \frac{|F|^2}{2\o}
  + \iint_\QG \frac{|G|^2}{2\g}
  \\
  \leq 
  \liminf\limits_{\k\to+\infty} \iint _\QO \frac{|F^\k|^2}{2\o^\k}
  + \liminf\limits_{\k\to+\infty}   \iint_\QG \frac{|G^\k|^2}{2\g^\k}
%   \\
  \leq \liminf\limits_{\k\to+\infty}\left(\iint _\QO \frac{|F^\k|^2}{2\o^\k}
  +  \iint_\QG \frac{|G^\k|^2}{2\g^\k}\right)
  \\
  \leq \liminf\limits_{\k\to+\infty}\left(\iint _\QO \frac{|F^\k|^2}{2\o^\k}
  +\iint_\QG \frac{|G^\k|^2+\kappa^2 |f^\k|^2}{2\g^\k}\right)
%   \\
  = \liminf\limits_{\k\to+\infty}\Wrr^2(\rho_0,\rho_1),
\end{multline*}
where we used that $(\o^\k,F^\k,\g^\k,G^\k,f^\k)$ is a geodesic in the last equality.
This implies that $\liminf=\limsup=\lim$ in this chain of inequalities, thus
$$
\lim\limits_{\k\to+\infty}\Wrr^2(\rho_0,\rho_1)
=
\iint _\QO \frac{|F|^2}{2\o}
  + \iint_\QG \frac{|G|^2}{2\g}
  =
  \W^2_\bO(\o_0,\o_1)+\W^2_\G(\g_0,\g_1)
$$
and the proof is complete.
\end{proof}
%%%%%%%%%%%%%%%%%%%%%%%%%%%%%%%%%%%%%%%%%%%%%%%%%%%%%%%%%%%%%%%%%%%%%%%%%%%%%%%%%%%%%%%%%%%%%%%%%%%%%%%%%%%%%%
\subsection{The small toll limit $\k\to 0$}
When $\k\to 0$ the mass flux $f$ between $\O$ and $\pO$ is barely penalized, and mass can thus flow almost freely between the interior and the boundary.
The discrimination between $\o$ and $\g$ types of cars on the ring road becomes weaker and weaker, on $\pO=\G$ only the total density $(\o\mrest\pO)+\g$ is retained in the end, and one therefore expects to recover the classical optimal transport problem for the total densities $\varrho=\o+\g$ on $\bO$:
\begin{theo}
  \label{theo:small_toll}
 For fixed $\rho_0,\rho_1\in\Pp(\bO)$ there holds
 \begin{equation}
    \label{eq:Wrr_cv_to_W}
  \Wrr^2(\rho_0,\rho_1)
  \xrightarrow[\k\to0]{}
  \W^2_\bO(\varrho_0,\varrho_1)
  \end{equation}
  with $\varrho_0=\o_0+\g_0$ and $\varrho_1=\o_1+\g_1$.
 Let moreover $\mu^\k=(\o^\k,F^\k\ ;\ \g^\k,G^\k,f^\k)$ be any geodesic for $\Wrr^2(\rho_0,\rho_1)$, let $\bar G^\k$ be the extension of $G^\k$ (by zero in the normal direction as well as inside $\O$), and set $\varrho^\k:=\o^\k+\g^\k$ and $H^\k:=F^\k+\bar G^\k$.
 Then, up to a subsequence
 \begin{equation}
   \label{eq:geodesics_cvnarrow_k_to_0}
   (\varrho^\k,H^\k) \narrowcv (\varrho,H)
\end{equation}
  for a Wasserstein geodesic $(\varrho,H)$ minimizing $\W^2_\bO(\varrho_0,\varrho_1)$.
\end{theo}
Note that this does not say anything about the convergence of the fluxes $f^\k$.
This is a delicate issue because the small cost $\k\ll 1$ allows the time regularity to degenerate.
As an example, for the explicit one-dimensional geodesic $\Wrr^2((\delta_{x_0},0),(0,\delta_{x_R}))$ from Section~\ref{sec:explicit_geodesics_dirac_masses} it is easy to check that $f^\k\narrowcv f= \delta_{1}\otimes\delta_{x_R}$ in $\M([0,1]\times\G)$ and that in the limit the boundary density becomes discontinuous in time and jumps from $\g_t=0$ for all $t<1$ to $\g_1=\delta_{x_R}$.
This explains the terminal time-impulse $\delta_1$ in $f=\partial_t\g$.
For $\k>0$ this transition of mass is spread over time, but as $\k\to 0$ the transfer is delayed as much as possible, concentrates in shorter and shorter time intervals, and in the limit all the mass jumps instantaneously.
Much of the proof below will actually consist in checking that the blow-up $\iint_\G\frac{|f^\k|^2}{2\g^\k}\to+\infty$ remains mild enough so that the effective cost $\k^2\iint_\G\frac{|f^\k|^2}{2\g^\k}\to 0$.

\begin{proof}
  We will first establish \eqref{eq:Wrr_cv_to_W}, and then deduce convergence of the geodesics.
  To this end we first recall the lower bound $\W^2_\bO(\varrho_0,\varrho_1) \leq \Wrr^2(\rho_0,\rho_1)$ from Proposition~\ref{prop:comparison_other_distances}.
  In particular
  $$
  \W^2_\bO(\varrho_0,\varrho_1) \leq \liminf\limits_{\k\to 0}\Wrr^2(\rho_0,\rho_1)
  $$
  and therefore it suffices to prove that $\limsup\limits_{\k\to 0} \Wrr^2(\rho_0,\rho_1)\leq \W^2_\bO(\varrho_0,\varrho_1)$.
  In order to establish this upper bound we will construct two measures $\rho^\k_0=(\o^\k_0,0),\rho^\k_1=(\o^\k_1,0)\in\Pp(\bO)$ supported away from the boundary such that
  \begin{equation}
    \label{eq:Wrr_approx_W_perturb_endpoints_interiors}
  \Wrr(\rho_0,\rho^\k_0)=o(1),
  \qquad
  \Wrr(\rho^\k_0,\rho^\k_1)\sim \W_\bO(\varrho_0,\varrho_1)
  ,
  \qquad
  \Wrr(\rho^\k_1,\rho_1)=o(1)
\end{equation}
  as $\k\to 0$.
  This will require in particular transferring all the boundary mass of $\g_0,\g_1$ to the interior at a small cost.
  By triangular inequality \eqref{eq:Wrr_approx_W_perturb_endpoints_interiors} will give the desired upper bound 
  $$
  \limsup\limits_{\k\to0} \Wrr(\rho_0,\rho_1)
  \leq 
  \limsup\limits_{\k\to0} \Big(\Wrr(\rho_0,\rho^\k_0) +\Wrr(\rho^\k_0,\rho^\k_1) +  \Wrr(\rho^\k_1,\rho_1)\Big)
  = 
  0 + \W_\bO(\varrho_0,\varrho_1) +0.
  $$
  We only discuss the construction for the perturbation $\rho_0^\k$ of $\rho_0$, the other endpoint is handled in the exact same fashion.
  The perturbation will be constructed in two steps $\rho_0\leadsto\tilde\rho^\k_0\leadsto\rho^\k_0$, first creating an annular gap around the boundary and then infiltrating the boundary mass inside the gap as in the proof of Theorem~\ref{theo:double_narrow_cv}.
  Each step will remain $o(1)$-close to the previous one in the $\Wrr$ distance as $\k\to0$.
  \begin{enumerate}
   \item 
   Pick $\eps=\eps_\k$ small enough to be determined later on.
   For $r>0$ we write $\G^{r}=\{x:\quad \operatorname{dist}(x,\pO)\leq r\}$ for the closed interior $r$-neighborhood of $\G=\pO$.
   Let $v_\eps(x)$ be a smooth velocity field with $\|v_\eps\|_{L^\infty(\bO)}\leq 1$, pointing inward and perpendicular to $\pO$, with unit norm on $\G^{2\eps}$, and vanishing outside of $\G^{3\eps}$.
   Let $\tilde \o^\k_0:=\Phi^{v_\eps}_{2\eps}\pf\o_0$ be the measure obtained by following the $v$-flow for times $2\eps$ starting from $\o_0$, keep $\tilde \g^\k_0:=\g_0$ unchanged, and let $\tilde\rho_0^\k:=\tilde\o^\k_0+\tilde\g^\k_0$.
   The corresponding time-flow $\tilde \o^\eps_t:=\Phi^{v_\eps}_t\pf\o_0$ gives an admissible path $\mu_t:=(\tilde \o^\eps_t,v_\eps \tilde \o^\eps_t,\g_0,0,0)$ connecting $\rho_0$ to $\tilde\rho_0^\k$ in time $t\in[0,2\eps]$, with cost
   \begin{equation}
     \label{eq:Wrr_rho_tilde_rho_leq_eps}
   \Wrr^2(\rho_0,\tilde\rho^\k_0) 
   \leq
   \frac {2\eps}2\int_0^{2\eps}\int_\bO |v_\eps|^2\rd\tilde\o^\k_t \rd t\leq 2\eps^2 \|v_\eps\|^2_\infty\|\tilde\o^\k\|\leq 2\eps^2 .
    \end{equation}
   Here we used an appropriate scaling in time in the middle integral.
   Note that the new interior density $\tilde\o^\k_0$ is now supported outside of $\G^{2\eps}$ at a distance at least $2\eps$ from the boundary.
   This will allow below to safely perturb the remaining measures only within $\G^\eps$, without modifying at all the just-constructed $\tilde\o^\k_0$.
  \item
  The second step will transfer the yet untouched boundary mass to the interior at distance exactly $\eps$ and for a small cost, while leaving the previous interior density $\tilde\o^\k_0$ untouched outside of $\G^{2\eps}$.
  For $y\in\pO$ the normal map $N_\eps(y):= y-\eps n(y)$ takes values in $\O$ if $\eps$ is small enough.
  Abbreviating $y_\eps:=N_{\eps}(y)\in\O$, we define
  $$
  \hat\o^\k_0:=N_{\eps_\k} \pf \tilde\g^\k_0=N_{\eps_\k} \pf \g_0
  ,\qquad
  \o^\k_0:=\tilde\o^\k_0+\hat\o^\k_0
  ,\qquad
  \rho^\k_0:=(\o^\k_0,0).
  $$
  The whole idea will consist below in connecting
  $$
  \tilde \g^\k_0=\g_0=\int_\G \delta_y\rd \g_0(y)
  \qqtext{and}
  \hat\o^\k_0=\int_\G \delta_{y_\eps}\rd \g_0(y)
  $$
  using a superposition of geodesics between point-masses from section~\ref{sec:explicit_geodesics_dirac_masses}, each starting at $y\in\O$ and ending at $y_\eps$ with infinitesimal mass $\rd \g_0(y)$.
  The dynamics of the resulting path will take place entirely inside $\G^\eps$, everything else will remain fixed outside, and it will be enough to estimate the cost of this path inside $\G^\eps$ to control $\Wrr^2(\tilde\rho^\k_0,\rho^\k_0)$.
  
  For $\g_0$-a.e. all $y\in\G$ let $\rho^{\k y}=(\o^{\k y},\g^{\k y})$ be the $\Wrr$ geodesic from $(0,\delta_y)$ to $(\delta_{y_\eps},0)$ constructed explicitly in Section~\ref{sec:explicit_geodesics_dirac_masses}, see Theorem~\ref{theo:explicit_1D_geodesics} up to time reversal.
  (This geodesic is well-defined since for all $y$ the segment $I^y=[y_\eps,y)$ remains included in $\O$ if $\eps$ is small enough.)
  This geodesic required no motion on the boundary, $\mu^{\k y}=(\o^{\k y},F^{\k y},\g^{\k y},0,f^{\k y})$, and was supported on $[0,1]\times I^y$.
  By linear superposition it is easy to check that
  $$
  \o^\k:=\tilde\o_0^\k\rd t + \int_\G \o^{\k y}\rd \g_0(y),
  \qquad 
  \g^\k:=\int_\G \g^{\k y}\rd \g_0(y)
 $$
 $$
  F^\k:=\int_\G F^{\k y}\rd \g_0(y),
  \qquad 
  G^\k:=0,
  \qquad 
  f^\k:=\int_\G f^{\k y}\rd \g_0(y)
  $$
  solve $\partial_t\o^\k +\dive F^\k =0$ with flux $f^\k$ as well as $\partial_t\g^\k + \dive G^\k =f^\k$, simply because $\o^{\k y},F^{\k y}$ and $\g^{\k y},G^{\k y},f^{\k y}$ do so for all $y$.
  Writing
  \begin{align*}
  & \o^\k|_{t=0}
  =\tilde\o_0^\k + \int_\G \o^{\k y}|_{t=0}\rd \g_0(y)
  =\tilde\o_0^\k + \int_\G 0
  =\tilde\o_0^\k,
  \\
  &\o^\k|_{t=1}
  =\tilde\o_0^\k + \int_\G \o^{\k y}|_{t=1}\rd \g_0(y)
  =\tilde\o_0^\k +  \int_\G \delta_{y_\eps}\rd \g_0(y)
  =\o^\k_0+\hat\o^\k_0=\o^\k_0
  \end{align*}
and
 \begin{align*}
  & \g^\k|_{t=0}
  = \int_\G \g^{\k y}|_{t=0}\rd \g_0(y)
  =\int_\G \delta_y\rd \g_0(y)
  =\g_0=\tilde\g^\k_0
  \\
  & \g^\k|_{t=1}
  = \int_\G \g^{\k y}|_{t=1}\rd \g_0(y)
  =\int_\G 0
  =0
  \end{align*}
  we see that this curve interpolates between $\tilde\rho^\k_0=(\tilde\o^\k_0,\tilde\g^k_0)$ and $\rho^\k_0=(\o^\k_0,0)$, and therefore $\mu^\k=(\o^\k,F^\k,\g^\k,0,f^\k)\in\CE(\tilde\rho^\k_0,\rho^\k_0)$.
  This gives an admissible competitor for the minimization problem defining $\Wrr^2(\tilde\rho^\k_0,\rho^\k_0)$, and because the action $\A$ is convex and $1$-homogeneous we have
  $$
  \Wrr^2(\tilde \rho^\k_0,\rho^\k_0)
  \leq
  \A(\mu^\k)
  =\A\left(\int_\G \mu^{\k y}\rd\g_0(y)\right)
  \leq 
  \int_\G\A(\mu^{\k y})\rd \g_0(y).
$$
Since $(\o^{\k y},F^{\k y},\g^{\k y},0,f^{\k y})$ is a geodesic between $\delta_y$ and $\delta_{y_\eps}$ we can apply Theorem \ref{theo:explicit_1D_geodesics} and \eqref{eq:explicit_Wrr_1D_diracs} with $R=|y_\eps-y|=\eps=\eps_\k$ to compute explicitly $\A(\mu^{\k y})=\Wrr^2((0,\delta_y),(\delta_{y_{\eps_\k}},0))$, resulting in
 \begin{equation}
     \label{eq:Wrr_rho_tilde_rho_leq_k}
\Wrr^2(\tilde \rho^\k_0,\rho^\k_0)
\leq \int_\G\left(
\frac{1}{2}(\eps_\k^2 +\k^2\alpha_\k)\frac{\alpha_\k}{\alpha_\k-1}
\right)\rd \g_0(y)
\leq 
\frac{1}{2}(\eps_\k^2 +\k^2\alpha_\k)\frac{\alpha_\k}{\alpha_\k-1}
\end{equation}
with $\alpha_\k=1+\sqrt{1+\frac{\eps_\k}{\k}}$.
  \end{enumerate}
  Taking $\eps_\k=\k$ gives $\alpha_\k= 1+\sqrt 2$ and $\Wrr^2(\tilde \rho^\k_0,\rho^\k_0)=\mathcal O(\k^2)$, hence from \eqref{eq:Wrr_rho_tilde_rho_leq_eps}\eqref{eq:Wrr_rho_tilde_rho_leq_k}
$$
\Wrr(\rho_0,\rho^\k_0)
\leq 
\Wrr(\rho_0,\tilde\rho^\k_0)
+\Wrr(\tilde\rho^\k_0,\rho^\k_0)
=\mathcal O(\k)\to 0.
$$
The very same construction allows to perturb the terminal endpoint $\Wrr(\rho_1,\rho^\k_1)=\mathcal O(\k)\to 0$ as well.
In order to fully establish \eqref{eq:Wrr_approx_W_perturb_endpoints_interiors} it remains to check that $\Wrr(\rho^\k_0,\rho^\k_1)\to\W^2_\bO(\varrho_0,\varrho_1)$.
Because $\rho^\k_0=(\o^\k_0,0),\rho^\k_1=(\o^\k_1,0)$ are supported away from the boundary, Proposition~\ref{prop:Wrr=W_no_boundary_mass} gives first $\Wrr(\rho^\k_0,\rho^\k_1)=\W_\bO^2(\o^\k_0,\o^\k_1)=\W^2_\bO(\varrho^k_0,\varrho^\k_1)$.
Next, since the previous construction gives $\Wrr(\rho_0^\k,\rho_0)\to 0$ and $\Wrr(\rho^\k_1,\rho_1)\to 0$, Theorem~\ref{theo:double_narrow_cv} guarantees that $\o^\k_0=\varrho^\k_0\narrowcv\varrho_0$ and $\o^\k_1=\varrho^\k_1\narrowcv\varrho_1$ in $\bO$.
We conclude by recalling that the Wasserstein distance metrizes the narrow convergence, hence $\Wrr^2(\rho^\k_0,\rho^\k_1)=\W^2_\bO(\varrho^\k_0,\varrho^\k_1)\to \W^2_\bO(\varrho_0,\varrho_1)$ and \eqref{eq:Wrr_cv_to_W} follows.

Let us now focus on the convergence of the geodesics themselves.
  By monotonicity in $\k$ (Proposition~\ref{prop:monotonicity_k}) we control
  $$
 \iint_\QO\frac{|F^\k|^2}{2\o^\k}
  +\iint_\QG \frac{|G^\k|^2}{2\g^\k}
  \leq \Wrr^2(\rho_0,\rho_1)
  \leq \mathcal W_{\k_0}^2(\rho_0,\rho_1)<+\infty
  $$
  for small $\k\leq \k_0$.
  By our favorite Jensen's inequality, with as usual $\TV{\o^\k},\TV{\g^\k}\leq 1$, we get the total variation estimate
  $$
  \TV{F^\k}^2 + \TV{G^\k}^2 \leq C
  $$
  uniformly in $\k\to 0$.
  Since the total variation of the extension $\TV{\bar G^\k}=\TV{G^\k}$ we see that, up to a subsequence if needed, $F^\k\narrowcv F$, $\bar G^\k\narrowcv \bar G$ for some limits $F,\bar G\in\M(\QO)^d$ and
  $$
  H^\k:=F^\k+\bar G^\k \narrowcv F+\bar G=: H
  \qquad \mbox{in }\M(\QO)^d.
  $$
  Since $\TV{\varrho^k}=1$ we have $\varrho^k\narrowcv\varrho$ as well, for some limit $\varrho\in\P(\QO)$.
  From Proposition~\ref{prop:CE_rho} we know that $\partial_t\varrho^\k +\dive H^\k=0$ with zero flux, hence the limit automatically solves $\partial_t\varrho +\dive H=0$ with endpoints $\varrho_0,\varrho_1$.
  Moreover by definition \eqref{eq:def_W2_rho_bO} of the Wasserstein distance and lower-semicontinuity there holds
  $$
  \W^2_\bO(\varrho_0,\varrho_1)
  \leq \iint_\QO \frac{|H|^2}{2\varrho}
  \leq \liminf\limits_{\k\to 0} \iint_\QO \frac{|H^\k|^2}{2\varrho^\k}.
  $$
  Being convex and $1$-homogeneous, the map $(\varrho,H)\mapsto \iint_\QO \frac{|H|^2}{2\varrho}$ is subadditive.
  Hence for fixed $\k>0$
  \begin{multline*}
  \iint_\QO \frac{|H^\k|^2}{2\varrho^\k}
  =\iint_\QO \frac{|F^\k+\bar G^\k|^2}{2(\o^\k+\g^\k)}
  \\
  \leq \iint_\QO \frac{|F^\k|^2}{2\o^\k}  + \iint_\QO \frac{|\bar G^\k|^2}{2\g^\k}
  = \iint_\QO \frac{|F^\k|^2}{2\o^\k}  + \iint_\QG \frac{|G^\k|^2}{2\g^\k}
  \\
  \leq \iint_\QO \frac{|F^\k|^2}{2\o^\k}  + \iint_\QG \frac{|G^\k|^2}{2\g^\k} +\k^2 \iint_\QG \frac{|f^\k|^2}{2\g^\k} 
  =\Wrr^2(\rho_0,\rho_1).
  \end{multline*}
  By the first step of the proof these two inequalities give altogether
  $$
  \W^2_\bO(\varrho_0,\varrho_1)
  \leq \iint_\QO \frac{|H|^2}{2\rho}
  \leq \liminf\limits_{\k\to 0} \Wrr^2(\rho_0,\rho_1)
  \overset{\eqref{eq:Wrr_cv_to_W}}{=}
  \W^2_\bO(\varrho_0,\varrho_1).
  $$
  As a consequence $\W^2_\bO(\varrho_0,\varrho_1)=\iint_\QO \frac{|H|^2}{2\varrho}$, meaning that the pair $(\varrho,H)$ is a Wasserstein geodesic.
\end{proof}
%

%%%%%%%%%%%%%%%%%%%%%%%%%%%%%%%%%%%%%%%%%%%%%%%%%%%%%%%%%%%%%%%%%%%%%%%%%%%%%%%%%%%%%%%%%%%%%%%%%%%%%%%%%%%%%%
%%%%%%%%%%%%%%%%%%%%%%%%%%%%%%%%%%%%%%%%%%%%%%%%%%%%%%%%%%%%%%%%%%%%%%%%%%%%%%%%%%%%%%%%%%%%%%%%%%%%%%%%%%%%%%
%%%%%%%%%%%%%%%%%%%%%%%%%%%%%%%%%%%%%%%%%%%%%%%%%%%%%%%%%%%%%%%%%%%%%%%%%%%%%%%%%%%%%%%%%%%%%%%%%%%%%%%%%%%%%%
%%%%%%%%%%%%%%%%%%%%%%%%%%%%%%%%%%%%%%%%%%%%%%%%%%%%%%%%%%%%%%%%%%%%%%%%%%%%%%%%%%%%%%%%%%%%%%%%%%%%%%%%%%%%%%
\section{Riemannian formalism}
\label{sec:Riemannian_formalism}
Theorem~\ref{theo:characterization_geodesics} strongly suggests that geodesics should be characterized by \eqref{eq:FGf_geodesics_potential_phi_psi}\eqref{eq:assumption_phi_subsol}\eqref{eq:assumption_psi_subsol}, or, written more concisely (and ignoring all the regularity and vacuum issues):
\begin{equation}
  \label{eq:geodesics_CE}
\left\{
\begin{array}{ll}
  \partial_t \o +\dive(\o\nabla\phi)=0  &  \mbox{in }\O
  \\
  \o\nabla\phi\cdot n =\g\frac{\psi-\phi}{\k^2}  & \mbox{in }\partial\O
\end{array}
\right.
\qqtext{and}
\partial_t \g + \dive(\g \nabla\psi)=\g \frac{\psi-\phi}{\k^2}    \quad\mbox{in }\G
\end{equation}
for some potentials $\phi,\psi$ satisfying the system of Hamilton-Jacobi equations
\begin{equation}
  \label{eq:geodesics_HJ_formal}
\left\{
\begin{array}{ll}
  \partial_t\phi+\frac 12|\nabla\phi|^2 =0   &   \mbox{in }\O\
  \\
  \o\nabla\phi\cdot n =\g \frac{\psi-\phi}{\k^2} & \mbox{in }\partial\O
\end{array}
\qqtext{and}
\partial_t\psi+\frac 12 |\nabla\psi|^2+\frac 1{2\k^2}|\psi-\phi|^2=0  \quad \mbox{in }\G
\right..
\end{equation}
As in Otto's formalism for the Wasserstein setting \cite{otto2001geometry}, we wish now to view $\Pp(\bO)$ as a formal Riemannian manifold, whose Riemannian distance should agree with $\Wrr$.
In other words, we would like to define a scalar product $<.,.>_\rho$ and a norm $\|.\|_\rho$ in the tangent space $T_\rho\Pp(\bO)$ at each point $\rho\in\Pp(\bO)$ so that
\begin{equation}
  \label{eq:Wrr=int_norm_rho_t}
  \Wrr^2(\rho_0,\rho_1)=\inf\left\{\int_0^1 \|\partial_t\rho_t\|^2_{\rho_t}\rd t
  :\qquad \rho=(\rho_t)_{t\in[0,1]} \mbox{ has endpoints }
  \rho_{0},\rho_{1}
  \right\}. 
\end{equation}

%%%%%%%%%%%%%%%%%%%%%%%%%%%%%%%%%%%%%%%%%%%%%%%
\subsection{Scalar product}
Here we argue at a static level.
The measures $F,G,f$ appearing below are measures acting in space only, and can be thought as generating the infinitesimal variations $\partial_t\rho=(\partial_t\o,\partial_t\g)\in T_\rho\Pp(\bO)$ given by $\partial_t\o+\dive F=0$ and $\partial_t\g+\dive G=f$ (with our usual flux condition $F\cdot n=f$).
Formula \eqref{eq:geodesics_CE} strongly suggests that such a given tangent vector $\partial_t\rho=(\partial_t\o,\partial_t\g) $ should identify with a unique pair of potentials $(\phi,\psi)$ solving the elliptic system
\begin{equation}
  \label{eq:identification_rho_t_phi_psi}
\left\{
\begin{array}{ll}
  -\dive(\o\nabla\phi)=\partial_t \o   &  \mbox{in }\O
  \\
  \o\nabla\phi\cdot n =\g\frac{\psi-\phi}{\k^2})  & \mbox{in }\partial\O
\end{array}
\right.
\qqtext{and}
-\dive(\g \nabla\psi)+\g \frac{\psi-\phi}{\k^2}=\partial_t \g     \quad\mbox{in }\G.
\end{equation}
Moreover in our of construction of $\Wrr$ we decided to measure the dissipation as $\frac{|u|^2}{2\o} + \frac{|v|^2+\k^2r^2}{2}\g$ in the ``velocity'' variables $u=F/\o,v=G/\g,r=f/\g$.
Expressed in terms of the potentials $\phi,\psi$ with $u=\nabla\phi$, $v=\nabla\psi$, $r=\g\frac{\psi-\phi}{\k^2}$ this naturally suggests the squared norm
\begin{equation}
\label{eq:norm_phi_psi}
\|\partial_t\rho\|^2_{\rho}
:= \frac 12 \int_\bO|\nabla\phi|^2\rd\o
+\frac{1}{2}\int_\G|\nabla\psi|^2\rd\g
+\frac 1{2\k^2} \int_\G |\psi-\phi|^2\rd \g.
\end{equation}
By polarization the scalar product between two tangent vectors $\partial_t\rho^1=(\partial_t\o^1,\partial_t\g^1)$ and $\partial_t\rho^2=(\partial_t\o^2,\partial_t\g^2)$ is simply
\begin{equation}
\label{eq:def_prod_scal}
\langle\partial_t\rho^1,\partial_t\rho^1\rangle_\rho :=
\frac{1}{2}\int_\O \nabla\phi^1\cdot\nabla\phi^2\,\rd\o
+ \frac{1}{2}\int_\G \nabla\psi^1\cdot\nabla\psi^2\,\rd\g
+ \frac{1}{2\k^2}\int_\G (\psi^1-\phi^1)(\psi^2-\phi^2)\,\rd\g
\end{equation}
with the corresponding identifications $\partial_t\rho^i=(\partial_t\o^i,\partial_t\g^i)\leftrightarrow (\phi^i,\psi^i)$.
This is fortunately consistent with a variational representation in the momentum variables:
\begin{lem}
  For a given $\partial_t\rho=(\partial_t\o,\partial_t\g)$ let $\|\partial_t\rho\|^2_\rho$ be given by \eqref{eq:identification_rho_t_phi_psi}\eqref{eq:norm_phi_psi}.
 Then 
 \begin{multline}
 \label{eq:norm_tangent_momentum}
\|\partial_t\rho\|^2_\rho=
\min\limits_{F,G,f}\Bigg\{
\frac 12\int_\bO \frac{|F|^2}{\o}
+ \frac 12\int_\G \frac{|G|^2}{\g}
+ \frac {\k^2}2\int_\G \frac{|f|^2}{\gamma}
\\
\mbox{s.t.}\qquad
-\dive F=\partial_t\o,\quad 
-\dive G+f=\partial_t\g,\quad
f=F\cdot n|_{\partial\O}
\Bigg\}.
\end{multline}
\end{lem}
\noindent
All regularity issues left aside, this is precisely what allows to recast the definition \eqref{eq:def_Wrr} of $\Wrr$ as \eqref{eq:Wrr=int_norm_rho_t} and justifies the (formal) Riemannian point of view.
\begin{proof}
  For convenience we choose to use the flux $f$ as a primary variable, and we will compute a first variation with respect to $f$ in order to extract some information about the global minimizers $(F ,G ,f )$.
  To this end let us define first the auxiliary functionals
  $$
  E_{\O}(f):=\min\limits_F\Bigg\{
  \frac{1}{2}\int_\bO\frac{|F|^2}{\o}
  :\qquad
  -\dive(F)=\partial_t \o
  \qtext{and}
  F\cdot n=f
  \Bigg\}
  $$
  $$
  E_{\G}(f):=\min\limits_G\Bigg\{
  \frac{1}{2}\int_\G\frac{|G|^2}{\g}:
  \qquad
  -\dive(G)=\partial_t \g-f
  \Bigg\}.
  $$
  It is clear that
  \begin{equation}
  \label{eq:norm_f_primary}
  \|\partial_t\rho\|^2_\rho=\inf\limits_f \Bigg\{E_{\O}(f)+E_\G(f) +\frac{\k^2}{2}\int_\G\frac{f^2}{\g}
  \Bigg\},
  \end{equation}
  and we shall compute the first variation of each term below.
  \\
  A straightforward and classical computation first shows that, for an arbitrary $f$, the unique minimizer $F=F[f]$ in the definition of $E_\O$ is obtained in potential form by solving the elliptic equation for $\phi$
  $$
  F[f]=\o\nabla\phi
  \qqtext{with}
  \left\{
  \begin{array}{ll}
  -\dive (\o\nabla\phi)=\partial_t\o & \mbox{in }\O\\
  \o\nabla\phi\cdot n =f & \mbox{in }\partial\O
  \end{array}
  \right.
  $$
  Observe that, given $\partial_t\o$, this problem is affine in $f$.
  Fixing $f $ and varying $f_\eps=f +\eps \tilde f$ for arbitrary directions $\tilde f$, the corresponding $F_\eps$ is therefore of the form
  $$
  F_\eps=F +\eps \tilde F=\o\left(\nabla\phi +\eps\nabla\tilde\phi\right)
  \qquad \mbox{with} \qquad
  \left\{
  \begin{array}{ll}
  -\dive (\o\nabla\tilde\phi)=0 & \mbox{in }\O\\
  \o\nabla\tilde\phi\cdot n =\tilde f & \mbox{in }\partial\O
  \end{array}
  \right.
  .
  $$
  As a consequence we have the first variation formula
  \begin{equation}
  \label{eq:DE_O}
  \left.\frac{d}{d\eps} \right|_{\eps=0} E_\O(f_\eps)
  =\int_\O\frac{F }{\o}\cdot \tilde F
  =\int_\O\nabla\phi \cdot\nabla\tilde\phi \,\o
  =-\int_\O \phi  \underbrace{\dive(\o\nabla\tilde\phi)}_{=0}+\int_{\partial\O}\phi \underbrace{\o\nabla\tilde\phi\cdot n}_{=\tilde f}
  =\int_{\partial\O}\phi \tilde f.
  \end{equation}
  Similarly, the minimizer $G=G[f]$ for $E_\G$ is obtained by solving
  $$
  G[f]=\g\nabla\psi
  \qquad
  \mbox{with}\qquad
  -\dive (\g\nabla\psi)=\partial_t\g -f  \mbox{ in }\G,
  $$
  which is again affine in $f$.
  Fixing $f $ and varying $f_\eps=f +\eps \tilde f$, the corresponding $G_\eps$ is now obtained as
  $$
  G_\eps=G +\eps \tilde G=\o\left(\nabla\psi +\eps\nabla\tilde\psi\right)
  \qquad
  \mbox{with}\qquad
  \dive (\g\nabla\tilde\psi)=\tilde f  \mbox{ in }\G
  $$
  and the first variation reads
  \begin{equation}
  \label{eq:DE_G}
  \left.\frac{d}{d\eps} \right|_{\eps=0} E_\G(f_\eps)
  =\int_\G\frac{G }{\g}\cdot \tilde G
  =\int_\G\nabla\psi \cdot\nabla\tilde\psi \,\g
  =-\int_\G \psi  \dive(\g \nabla\tilde\psi) 
  =-\int_\G \psi  \tilde f.
  \end{equation}
  Going back to \eqref{eq:norm_f_primary}, let $f$ be the minimizer.
  Choosing an arbitrary directions $\tilde f$ to perturb $f_\eps =f+\eps\tilde f$ and exploiting \eqref{eq:DE_O}\eqref{eq:DE_G}, we get the first order optimality condition
  \begin{equation*}
  0=\left .\frac{d}{d\eps}\right|_{\eps=0}  
  \left(
  E_{\O}(f_\eps)+ E_\G(f_\eps)+\frac{\k^2}{2}\int_{\G}\frac{|f_\eps|^2}{\g}
  \right)
  = \int_{\partial\O} \phi \tilde f -\int_\G \psi  \tilde f + \k^2\int_\G \frac{f }{\g}\tilde f.
  \end{equation*}
  Since $\tilde f$ was arbitrary this simply means
  $$
  f =\gamma\frac{\psi -\phi }{\k^2}.
  $$
  In other words, the minimizer $(F,G,f)$ in the initial problem \eqref{eq:norm_tangent_momentum} is characterized by
  $$
  F=\o\nabla \phi,
  \qquad
  G=\g\nabla\psi,
  \qquad
  f=\gamma\frac{\psi -\phi }{\k^2},
  $$
  where the pair of Kantorovich potentials $\phi,\psi$ should solve the elliptic system
  \eqref{eq:identification_rho_t_phi_psi}.
  Evaluating the right-hand side of \eqref{eq:norm_tangent_momentum} for these optimal values gives exactly \eqref{eq:norm_phi_psi} and the proof is complete.
\end{proof}
Let us now check that the previous computations are consistent with the metric notion of constant-speed geodesics in Proposition~\ref{prop:cst_speed}:
\begin{lem}
  Let $(\o,\g)$ and $(\phi,\psi)$ solve \eqref{eq:geodesics_CE}\eqref{eq:geodesics_HJ_formal} for $t\in[0,1]$.
  Then $\|\partial_t \rho_t\|^2_{\rho_t}=cst$.
\end{lem}
\begin{proof}
We compute
\begin{multline*}
  \frac{d}{dt}\|\partial_t \rho_t\|^2_{\rho_t}
  =\frac{d}{dt}\left(
  \frac 12 \int_\O|\nabla\phi_t|^2\, \o_t
+\frac{1}{2}\int_\G|\nabla\psi_t|^2\, \g_t
+\frac 1{2\k^2} \int_\G |\psi_t-\phi_t|^2\, \g_t
  \right)\\
  =
  \underbrace{\left[
\frac{1}{2}\int_\O\partial_t\o_t|\nabla\phi_t|^2  +  \int_\O\o_t\nabla\phi_t\cdot\nabla\partial_t\phi_t
\right]}_{:=A}
+ \underbrace{ \left[
\frac{1}{2}\int_\G\partial_t\g_t|\nabla\psi_t|^2  +  \int_\G \g_t\nabla\psi_t\cdot\nabla\partial_t\psi_t
\right]}_{:=B}\\
+\underbrace{  \frac{1}{\k^2}\left[
  \frac{1}{2}\int_\G\partial_t\g_t|\psi_t-\phi_t|^2  +  \int_\G \g_t(\psi_t-\phi_t)\partial_t(\psi_t-\phi_t)
\right]}_{:=C}
\end{multline*}
Using the $\o,\phi$ equations in \eqref{eq:geodesics_CE}\eqref{eq:geodesics_HJ_formal}, some rather tedious but straightforward computations and integrations by parts give
\begin{align*}
 & A= -\frac 12 \int_{\partial\O}|\nabla\phi_t|^2 f_t,
 \\
 & B= \frac 12 \int_\G |\nabla\psi_t|^2f_t  - \frac 1{2\k^2} \int_\G\nabla\psi_t\cdot \nabla\left((\psi_t-\phi_t)^2\right)\g_t,
 \\
 & C=\frac{1}{2\k^2}\int_\G \nabla\psi_t\cdot\nabla\left((\psi_t-\phi_t)^2\right)\g_t 
+ \frac 12 \int_\G  |\nabla\phi_t|^2f_t - \frac 12 \int_\G  |\nabla\psi_t|^2 f_t,
\end{align*}
whence the desired cancellation $\frac{d}{dt}\|\partial_t \rho_t\|^2_{\rho_t} = A+B+C
=0$.
\end{proof}

%%%%%%%%%%%%%%%%%%%%%%%%%%%%%%%%%%%%%%%%%%%%%%%%%%%%%%%%%%%%%%%%%%%%%%%%%%
%%%%%%%%%%%%%%%%%%%%%%%%%%%%%%%%%%%%%%%%%%%%%%%%%%%%%%%%%%%%%%%%%%%%%%%%%%
\subsection{Gradients}
With a Riemannian metric at hand we can now try to make sense of Riemannian gradients for functionals $\E:\Pp(\bO)\to\R$.
For simplicity let us restrict here to energy functionals of the form bulk + interface
$$
\E(\rho)
=\E_\O (\o)+\E_\G(\g)
:=\int_\O E_\O(\o(x))\rd x +\int_\G E_\G(\g(x))\rd x,
$$
with the convention that $\E_\O=+\infty$ or $\E_\G=+\infty$ whenever $\o$ or $\g$ are not absolutely continuous w.r.t. the Lebesgue measure on $\O,\G$, respectively.
\\

The gradient $\grad_{\Wrr} \E$ can be computed by a (formal) chain rule as follows.
Let $\rho_t=(\o_t,\g_t)$ be a curve defined for times $t\in(-\eps,\eps)$ with $\rho(0)=\rho$, solving the continuity equations \eqref{eq:geodesics_CE} for some fixed $\phi(x),\psi(x)$ representing the tangent vector $\partial_t\rho(0)$ at the base-point $\rho\in \Pp(\bO)$.
Then
\begin{multline*}
  \left\langle \grad_{\Wrr} \E(\rho),\partial_t\rho\right\rangle_\rho
  =
  \left.\frac{d}{dt}\right|_{t=0}\E(\rho_t)
%   \\ 
  =
  \left.\frac{d}{dt}\right|_{t=0}\int_\O E_\O(\o_t(x))\rd x
  +\left.\frac{d}{dt}\right|_{t=0}\int_\G E_\G(\g_t(x))\rd x
  \\
  =\int_\O E_\O'(\o)\partial_t\o
  +\int_\G E_\G'(\g)\partial_t\g
  =\int_\O E_\O'(\o)\{-\dive(\o\nabla\phi)\}
  +\int_\G E_\G'(\g)\{-\dive(\g\nabla\psi)+f\}
  \\
  =
  \Bigg(
  \int_\O \nabla E_\O'(\o)\cdot\o\nabla \phi
  -
  \int_{\partial\O} E_\O'(\o) \underbrace{\o \nabla\phi\cdot n}_{=f}
  \Bigg)
  + 
  \int_\G\{\g \nabla E_\G'(\g) \cdot \nabla\psi + E_\G'(\g) f\}
  \\
  = 
  \int_\O \o \nabla E_\O'(\o)\cdot\nabla \phi 
  +
  \int_\G\g \nabla E_\G'(\g) \cdot \nabla\psi 
  +
  \int _\G (E_\G'(\g)-E_\O'(\o)) f
  \\
  = 
  \int_\O \nabla E_\O'(\o)\cdot\nabla \phi \,\rd\o
  +
  \int_\G \nabla E_\G'(\g) \cdot \nabla\psi \,\rd\g
  +
  \frac{1}{\k^2}\int _\G \left(E_\G'(\g)-E_\O'(\o)\right)\cdot (\psi-\phi)\,\rd\g.
\end{multline*}
By definition \eqref{eq:def_prod_scal} of the scalar product this immediately identifies the object $\grad_{\Wrr}  \E(\rho)$ as the pair of potentials
$$
(\phi,\psi)=\big(E_\O'(\o),E_\G'(\g)\big)
$$
given by the usual ($L^2$) first variation of $\E_\O,\E_G$.
More explicitly, this means
\begin{equation}
\label{eq:computation_gradient}
  \grad_{\Wrr}\E(\rho)=\left(
  -\dive(\o\nabla E_\O'(\o))
  \quad, \quad
  -\dive(\g\nabla E_\G'(\g)) +\g\frac{E_\G'(\g)-E_\O'(\o)}{\k^2}
\right)
\end{equation}
with the implicit compatibility assumption that
$$
\o\nabla E_\O'(\o)\cdot n= \g \frac{E_\G'(\g)-E_\O'(\o) }{\k^2}
\hspace{1.5cm}
\mbox{on }\pO.
$$
The next step natural step would be to consider gradient flows
$$
\partial_t\rho=-\grad_{\Wrr}\E(\rho).
$$
This will be investigated in a subsequent work,
but for the record let us discuss here two popular cases:
\begin{enumerate}
 \item 
 \emph{The relative Boltzmann entropy}.
 Take any two potentials $V_\O\in \C^1(\bO),V_\G\in\C^1(\G)$, and define the Gibbs measure
 $$
 \pi:=(\pi_\O,\pi_\G)=\frac{1}{Z}\left(e^{-V_\O}\operatorname{Leb}_\O,e^{-V_\G}\operatorname{Leb}_\G\right)\in\Pp(\bO)
 $$
 for a unique normalizing constant $Z$ such that $\pi_\O+\pi_\G\in\P(\bO)$ has mass one.
 For $H(z):=z\log z$ we define the relative entropy
 \begin{multline*}
 \E(\rho)=\H(\rho|\pi):=
 \int_\O H\left(\frac{\o}{\pi_\O}\right)\rd\pi_\O
 +\int_\G H\left(\frac{\g}{\pi_\G}\right)\rd\pi_\G
 \\
 =
 \int_\O \left\{\o\log\o +\o V_\O\right\}\rd x
 +\int_\G \left\{\g\log\g +\o V_\G\right\}\rd x.
 \end{multline*}
 This gives $E_\O'(\o)=\log\o +V_\O +1$, $E_\G'(\g)=\log\g +V_\G +1$, and the resulting system of PDEs is the coupled system of heat equations
 $$
 \left\{
 \begin{array}{ll}
  \partial_t\o =\Delta\o +\dive(\o\nabla V_\O) & \mbox{in }\O
  \\
  \frac{\partial \o}{\partial n}+\o\frac{\partial V_\O}{\partial n} 
  =
  \frac{\g}{\k^2} (\log \g +V_{\G}-\log\o -V_\O)  & \mbox{on }\pO
 \end{array}
 \right.
 $$
 and
 $$
 \partial_t\g =\Delta \g +\dive(\g\nabla V_\G)-\frac{\g}{\k^2}(\log \g +V_{\G}-\log\o -V_\O)
 \quad\mbox{in }\G.
 $$
 \item
 \emph{The R\'enyi entropy} .
 Choose a pair of exponents $m=(m_\O,m_\G)$ with $m_\O,m_\G>1$, and let
 $$
 \E_m(\rho):=
 \int_\O \frac{\o^{m_\O}}{m_\O-1}\rd x
 +
 \int_\G \frac{\g^{m_\G}}{m_\G-1}\rd x .
 $$
 This gives $E_\O'(\o)=\frac{m_\O}{m_\O-1}\o^{m_\O-1}, E_\G'(\g)=\frac{m_\G}{m_\G-1}\g^{m_\G-1}$ and the gradient flow read now as a system of Porous Medium Equations \cite{otto2001geometry}
 $$
\left\{
 \begin{array}{ll}
  \partial_t\o =\Delta\o^{m_\O}  & \mbox{in }\O
  \\
  \frac{\partial\left(\o^{m_\O} \right)}{\partial n}
  =
  \frac{\g}{\k^2} \left(\frac{m_\G}{m_\G-1}\g^{m_\G-1}  -  \frac{m_\O}{m_\O-1}\o^{m_\O-1}\right)  & \mbox{on }\pO
 \end{array}
 \right.
 $$
 and
 $$
 \partial_t\g =\Delta \g^{m_\G}-\frac{\g}{\k^2}\left( \frac{m_\G}{m_\G-1}\g^{m_\G-1}  -  \frac{m_\O}{m_\O-1}\o^{m_\O-1}\right)
 \qqtext{in }\G.
 $$
\end{enumerate}

%%%%%%%%%%%%%%%%%%%%%%%%%%%%%%%%%%%%%%%%%%%%%%%%%%%%%%%%%%%%%%%%%%%%%%%%%%
%%%%%%%%%%%%%%%%%%%%%%%%%%%%%%%%%%%%%%%%%%%%%%%%%%%%%%%%%%%%%%%%%%%%%%%%%%

%%%%%%%%%%%%%%%%%%%%%%%%%%%%%%%%%%%%%%%%%%%%%%%%%%%%%%%%%%%%%%%%%%%%%%%%%%
%%%%%%%%%%%%%%%%%%%%%%%%%%%%%%%%%%%%%%%%%%%%%%%%%%%%%%%%%%%%%%%%%%%%%%%%%%
\subsection*{Acknowledgments}
The author warmly thanks C. Canc\`es and B. Merlet for discussions pertaining to this work, and D. Vorotnikov for his help with the proof of Theorem~\ref{theo:explicit_1D_geodesics}.
This work was funded by the Portuguese Science Foundation through FCT Project PTDC/MAT-STA/28812/2007 {\it Schr\"oMoka}.

%%%%%%%%%%%%%%%%%%%%%%%%%%%%%%%%%%%%%%%%%%%%%%%%%%%%%%%%%%%%%%%%%%%%%%%%%%
% \bibliographystyle{plain}
% \bibliography{./biblio}

\begin{thebibliography}{10}

\bibitem{AGS}
Luigi Ambrosio, Nicola Gigli, and Giuseppe Savar{\'e}.
\newblock {\em Gradient flows: in metric spaces and in the space of probability
  measures}.
\newblock Springer Science \& Business Media, 2008.

\bibitem{BB}
Jean-David Benamou and Yann Brenier.
\newblock A computational fluid mechanics solution to the {M}onge-{K}antorovich
  mass transfer problem.
\newblock {\em Numerische Mathematik}, 84(3):375--393, 2000.

\bibitem{bouchitte1990new}
Guy Bouchitt{\'e} and Giuseppe Buttazzo.
\newblock New lower semicontinuity results for nonconvex functionals defined on
  measures.
\newblock {\em Nonlinear Analysis: Theory, Methods \& Applications},
  15(7):679--692, 1990.

\bibitem{brenier2003extended}
Yann Brenier.
\newblock Extended {M}onge-{K}antorovich theory.
\newblock In {\em Optimal transportation and applications}, pages 91--121.
  Springer, 2003.

\bibitem{caffarelli2010free}
Luis~A Caffarelli and Robert~J McCann.
\newblock Free boundaries in optimal transport and {M}onge--{A}mp{\`e}re
  obstacle problems.
\newblock {\em Annals of mathematics}, pages 673--730, 2010.

\bibitem{chizat2018interpolating}
L{\'e}na{\"i}c Chizat, Gabriel Peyr{\'e}, Bernhard Schmitzer, and
  Fran{\c{c}}ois-Xavier Vialard.
\newblock An interpolating distance between optimal transport and fisher--rao
  metrics.
\newblock {\em Foundations of Computational Mathematics}, 18(1):1--44, 2018.

\bibitem{chizat2018scaling}
L{\'e}na{\"i}c Chizat, Gabriel Peyr{\'e}, Bernhard Schmitzer, and
  Fran{\c{c}}ois-Xavier Vialard.
\newblock Scaling algorithms for unbalanced optimal transport problems.
\newblock {\em Mathematics of Computation}, 87(314):2563--2609, 2018.

\bibitem{chizat2018unbalanced}
L{\'e}na{\"i}c Chizat, Gabriel Peyr{\'e}, Bernhard Schmitzer, and
  Fran{\c{c}}ois-Xavier Vialard.
\newblock Unbalanced optimal transport: Dynamic and {K}antorovich formulations.
\newblock {\em Journal of Functional Analysis}, 274(11):3090--3123, 2018.

\bibitem{dolbeault2009new}
Jean Dolbeault, Bruno Nazaret, and Giuseppe Savar{\'e}.
\newblock A new class of transport distances between measures.
\newblock {\em Calculus of Variations and Partial Differential Equations},
  34(2):193--231, 2009.

\bibitem{dudley2018real}
Richard~M Dudley.
\newblock {\em Real analysis and probability}.
\newblock Chapman and Hall/CRC, 2018.

\bibitem{figalli2010optimal}
Alessio Figalli.
\newblock The optimal partial transport problem.
\newblock {\em Archive for rational mechanics and analysis}, 195(2):533--560,
  2010.

\bibitem{figalli2010new}
Alessio Figalli and Nicola Gigli.
\newblock A new transportation distance between non-negative measures, with
  applications to gradients flows with dirichlet boundary conditions.
\newblock {\em J. Math. Pures Appl.(9)}, 94(2):107--130, 2010.

\bibitem{gallouet2019unbalanced}
Thomas Gallou{\"e}t, Maxime Laborde, and L{\'e}onard Monsaingeon.
\newblock An unbalanced optimal transport splitting scheme for general
  advection-reaction-diffusion problems.
\newblock {\em ESAIM: Control, Optimisation and Calculus of Variations}, 25:8,
  2019.

\bibitem{gallouet2017jko}
Thomas Gallou{\"e}t and L{\'e}onard Monsaingeon.
\newblock A {JKO} splitting scheme for {K}antorovich--{F}isher--{R}ao gradient
  flows.
\newblock {\em SIAM Journal on Mathematical Analysis}, 49(2):1100--1130, 2017.

\bibitem{gallouet2019generalized}
Thomas Gallou{\"e}t, Andrea Natale, and Fran{\c{c}}ois-Xavier Vialard.
\newblock Generalized compressible flows and solutions of the
  ${H}(\mathrm{div})$ geodesic problem.
\newblock {\em Archive for Rational Mechanics and Analysis}, pages 1--56, 2019.

\bibitem{gallouet2018camassa}
Thomas Gallou{\"e}t and Fran{\c{c}}ois-Xavier Vialard.
\newblock The {C}amassa--{H}olm equation as an incompressible {E}uler equation:
  {A} geometric point of view.
\newblock {\em Journal of Differential Equations}, 264(7):4199--4234, 2018.

\bibitem{gangbo2019unnormalized}
Wilfrid Gangbo, Wuchen Li, Stanley Osher, and Michael Puthawala.
\newblock Unnormalized optimal transport.
\newblock {\em Journal of Computational Physics}, 399:108940, 2019.

\bibitem{glitzky2013gradient}
Annegret Glitzky and Alexander Mielke.
\newblock A gradient structure for systems coupling reaction--diffusion effects
  in bulk and interfaces.
\newblock {\em Zeitschrift f{\"u}r angewandte Mathematik und Physik},
  64(1):29--52, 2013.

\bibitem{kantorovich1942translocation}
Leonid~V. {K}antorovich.
\newblock On the translocation of masses.
\newblock In {\em Dokl. Akad. Nauk. USSR (NS)}, volume~37, pages 199--201,
  1942.

\bibitem{kondratyev2016fitness}
Stanislav Kondratyev, L{\'e}onard Monsaingeon, and Dmitry Vorotnikov.
\newblock A fitness-driven cross-diffusion system from population dynamics as a
  gradient flow.
\newblock {\em Journal of Differential Equations}, 261(5):2784--2808, 2016.

\bibitem{kondratyev2016new}
Stanislav Kondratyev, L{\'e}onard Monsaingeon, and Dmitry Vorotnikov.
\newblock A new optimal transport distance on the space of finite {R}adon
  measures.
\newblock {\em Advances in Differential Equations}, 21(11/12):1117--1164, 2016.

\bibitem{kondratyev2017new}
Stanislav Kondratyev, L{\'e}onard Monsaingeon, and Dmitry Vorotnikov.
\newblock A new multicomponent {P}oincar{\'e}--{B}eckner inequality.
\newblock {\em Journal of Functional Analysis}, 272(8):3281--3310, 2017.

\bibitem{kondratyev2019convex}
Stanislav Kondratyev and Dmitry Vorotnikov.
\newblock Convex {S}obolev inequalities related to unbalanced optimal
  transport.
\newblock {\em Journal of Differential Equations}, 2019.

\bibitem{kondratyev2019spherical}
Stanislav Kondratyev and Dmitry Vorotnikov.
\newblock Spherical {H}ellinger--{K}antorovich gradient flows.
\newblock {\em SIAM Journal on Mathematical Analysis}, 51(3):2053--2084, 2019.

\bibitem{kondratyev2020nonlinear}
Stanislav Kondratyev and Dmitry Vorotnikov.
\newblock Nonlinear {F}okker-{P}lanck equations with reaction as gradient flows
  of the free energy.
\newblock {\em Journal of Functional Analysis}, 278(2):108310, 2020.

\bibitem{laschos2019geometric}
Vaios Laschos and Alexander Mielke.
\newblock Geometric properties of cones with applications on the
  {H}ellinger--{K}antorovich space, and a new distance on the space of
  probability measures.
\newblock {\em Journal of Functional Analysis}, 276(11):3529--3576, 2019.

\bibitem{lavenant2019unconditional}
Hugo Lavenant.
\newblock Unconditional convergence for discretizations of dynamical optimal
  transport.
\newblock {\em arXiv preprint arXiv:1909.08790}, 2019.

\bibitem{liero2013gradient}
Matthias Liero and Alexander Mielke.
\newblock Gradient structures and geodesic convexity for reaction--diffusion
  systems.
\newblock {\em Philosophical Transactions of the Royal Society A: Mathematical,
  Physical and Engineering Sciences}, 371(2005):20120346, 2013.

\bibitem{liero2016optimal}
Matthias Liero, Alexander Mielke, and Giuseppe Savar{\'e}.
\newblock Optimal transport in competition with reaction: The
  {H}ellinger--{K}antorovich distance and geodesic curves.
\newblock {\em SIAM Journal on Mathematical Analysis}, 48(4):2869--2911, 2016.

\bibitem{liero2018optimal}
Matthias Liero, Alexander Mielke, and Giuseppe Savar{\'e}.
\newblock Optimal entropy-transport problems and a new
  {H}ellinger--{K}antorovich distance between positive measures.
\newblock {\em Inventiones mathematicae}, 211(3):969--1117, 2018.

\bibitem{mielke2011thermomechanical}
Alexander Mielke.
\newblock Thermomechanical modeling of energy-reaction-diffusion systems,
  including bulk-interface interactions.
\newblock 2011.

\bibitem{monge1781memoire}
Gaspard {M}onge.
\newblock M{\'e}moire sur la th{\'e}orie des d{\'e}blais et des remblais.
\newblock {\em Histoire de l'Acad{\'e}mie Royale des Sciences de Paris}, 1781.

\bibitem{otto2001geometry}
Felix Otto.
\newblock The geometry of dissipative evolution equations: the porous medium
  equation.
\newblock 2001.

\bibitem{peyre2019computational}
Gabriel Peyr{\'e}, Marco Cuturi, et~al.
\newblock Computational optimal transport.
\newblock {\em Foundations and Trends{\textregistered} in Machine Learning},
  11(5-6):355--607, 2019.

\bibitem{piccoli2014generalized}
Benedetto Piccoli and Francesco Rossi.
\newblock Generalized {W}asserstein distance and its application to transport
  equations with source.
\newblock {\em Archive for Rational Mechanics and Analysis}, 211(1):335--358,
  2014.

\bibitem{piccoli2016properties}
Benedetto Piccoli and Francesco Rossi.
\newblock On properties of the generalized {W}asserstein distance.
\newblock {\em Archive for Rational Mechanics and Analysis}, 222(3):1339--1365,
  2016.

\bibitem{profeta2018heat}
Angelo Profeta and Karl-Theodor Sturm.
\newblock Heat flow with dirichlet boundary conditions via optimal transport
  and gluing of metric measure spaces.
\newblock {\em arXiv preprint arXiv:1809.00936}, 2018.

\bibitem{rockafellar1967duality}
Ralph Rockafellar.
\newblock Duality and stability in extremum problems involving convex
  functions.
\newblock {\em Pacific Journal of Mathematics}, 21(1):167--187, 1967.

\bibitem{rockafellar1971integrals}
Ralph Rockafellar.
\newblock Integrals which are convex functionals. ii.
\newblock {\em Pacific Journal of Mathematics}, 39(2):439--469, 1971.

\bibitem{OTAM}
Filippo Santambrogio.
\newblock Optimal transport for applied mathematicians.
\newblock {\em Birk{\"a}user, NY}, 55(58-63):94, 2015.

\bibitem{villani2003topics}
C{\'e}dric Villani.
\newblock {\em Topics in optimal transportation}.
\newblock Number~58. American Mathematical Soc., 2003.

\bibitem{villani_BIG}
C{\'e}dric Villani.
\newblock {\em Optimal transport: old and new}, volume 338.
\newblock Springer Science \& Business Media, 2008.

\bibitem{zurek2019problemes}
Antoine Zurek.
\newblock {\em Probl{\`e}mes {\`a} interface mobile pour la d{\'e}gradation de
  mat{\'e}riaux et la croissance de biofilms: analyse num{\'e}rique et
  mod{\'e}lisation}.
\newblock PhD thesis, 2019.

\end{thebibliography}

%%%%%%%%%%%%%%%%%%%%%%%%%%%%%%%%%%%%%%%%%%%%%%%%%%%%
\bigskip
\noindent
{\sc
L\'eonard Monsaingeon (\href{mailto:leonard.monsaingeon@univ-lorraine.fr}{\tt leonard.monsaingeon@univ-lorraine.fr}).
\\
Institut \'Elie Cartan de Lorraine, Universit\'e de Lorraine, Site de Nancy B.P. 70239, F-54506 Vandoeuvre-l\`es-Nancy Cedex, France
\\
Grupo de F\'isica Matem\'atica, GFMUL, Faculdade de Ci\^encias, Universidade de Lisboa, 1749-016 Lisbon, Portugal.
}

\end{document}